\documentclass[12pt, reqno]{amsart}
\usepackage{graphicx} 
\usepackage[a4paper, margin=0.95in]{geometry}  
\usepackage{amsfonts,amsmath,amsthm} 
\usepackage{amssymb}
\usepackage{tikz}
\usetikzlibrary{cd}
\usepackage[all]{xy}
\usetikzlibrary{arrows,shapes}
\usepackage{bbm}
\usepackage[utf8]{inputenc} 
\usepackage[T1]{fontenc}    
 \usepackage{media9}

\usepackage{caption}
\usepackage{subcaption}

\usepackage{sidecap}

\usepackage{booktabs}

\usepackage[labelfont=bf, font=small]{caption}
\usepackage[subrefformat=parens, labelformat=parens]{subcaption}

\usepackage{algorithm}
\usepackage{algorithmic}

\usepackage{enumitem}
\setlist[itemize]{leftmargin=20pt}
\setlist[enumerate]{leftmargin=20pt}

\usepackage{cancel} 

\usepackage{graphicx}
\usepackage{wrapfig}

\usepackage[colorlinks=true, linkcolor=blue, citecolor=blue, urlcolor=blue]{hyperref}

\usepackage{float}

\usepackage{mathtools}

\newtheorem{theorem}{Theorem}[section]
\newtheorem{proposition}[theorem]{Proposition}
\newtheorem{lemma}[theorem]{Lemma}
\newtheorem{definition}[theorem]{Definition}
\newtheorem{corollary}[theorem]{Corollary}
\newtheorem{example}[theorem]{Example}

\newtheorem{remark}[theorem]{Remark}

\newcommand{\R}{\mathbb{R}}
\newcommand{\N}{\mathbb{N}}

\newcommand{\p}{\mathrm{p}}
\newcommand{\q}{\mathrm{q}}

\DeclareMathOperator*{\argmin}{arg\,min}

\usepackage[normalem]{ulem} 

\usepackage{lineno}

\author[R. Díaz Martín]{Rocío Díaz Martín}
\address{Rocío Díaz Martín, Florida State University, Department of Mathematics, Tallahassee, FL 32306, USA}
\email{rdiazmartin@fsu.edu}

\author[I. Medri]{Ivan V. Medri}
\address{Ivan V. Medri, University of Virginia, Department of Biomedical Engineering, Charlottesville, VA 22904, USA}
\email{pzr7pr@virginia.edu}

\author[J. Murphy]{James M. Murphy}
\address{James M. Murphy, Tufts University, Department of Mathematics, Medford, MA 02155, USA}
\email{jm.murphy@tufts.edu}

\begin{document}

\title{Gromov-Wasserstein Barycenters: The Analysis Problem}


\begin{abstract}
    This paper considers the problem of estimating a matrix that encodes pairwise distances in a finite metric space--or, more generally, the edge weight matrix of a network--under the barycentric coding model (BCM) with respect to the Gromov-Wasserstein (GW) distance function. We frame this task as estimating the unknown barycentric coordinates with respect to the GW distance, assuming that the target matrix (or kernel) belongs to the set of GW barycenters of a finite collection of known templates.    
    In the language of harmonic analysis, 
    if computing GW barycenters can be viewed as a synthesis problem, this paper aims to solve the corresponding analysis problem.
    We propose two methods: one utilizing fixed-point iteration for computing GW barycenters, and another employing a differentiation-based approach to the GW structure using a blow-up technique.  Finally, we demonstrate the application of the proposed GW analysis approach in a series of numerical experiments and applications to machine learning.
\end{abstract}

\maketitle

\section{Introduction}

The \emph{barycentric coding model} refers to a {data representation framework that leverages the concept of expressing points in a convex set as barycentric combinations of its vertices. This approach is widely used in geometry, signal processing, and machine learning (ML), where data 
are represented via coordinates relative to specific reference points--sometimes called anchors, landmarks, or templates--rather than in a traditional Cartesian coordinate system or in terms of an algebraic basis of vectors \cite{coxeter1961introduction}.
While these notions--which include principal component analysis \cite{pearson1901liii}, nonnegative matrix factorization \cite{lee2000algorithms}, and archetypal analysis \cite{cutler1994archetypal}--are most natural in a vector space, they can be adapted to metric spaces via notions of \emph{barycenter or Fr\'echet mean}, which do not always coincide with the convex combinations of some vertices of the set. The applications of modeling with barycentric coordinates range from data compression--where barycentric coding reduces dimensionality or redundancy by encoding data in terms of a few reference points (particularly in structured datasets)--to generative modeling, clustering, and embedding in ML--where data points are represented relative to cluster centroids--among other uses \cite{zhang2014local, hormann2017generalized, warren2007barycentric, DeRoseMeyer2006,tovsic2011dictionary}.   

These models allow for  \emph{synthesis} and \emph{analysis} of points in a metric space with distance $dist(\cdot,\cdot)$, through
variational problems relative to given reference points $\mathbf{x}^1, \dots, \mathbf{x}^S$. For the \emph{synthesis problem}, given a vector of non-negative coefficients $\lambda = (\lambda_1, \dots, \lambda_S)$ in the $(S-1)$-dimensional simplex $\Delta_{S-1}$, one seeks  a point in the space minimizing the  following objective:

\begin{equation}\label{eq: syn}
    \mathbf{x}_\lambda \in \argmin_{\mathbf{x}} 
    \sum_{s=1}^S \lambda_s \, dist^2(\mathbf{x}^s, \mathbf{x});
\end{equation}
where powers $p\ge 1$ other than 2 can also be considered. One can also `project' a point $\mathbf{y}$ onto the collection of such synthesized points $\{\mathbf x_{\lambda}\}_{\lambda\in\Delta_{S-1}}$ by solving the \emph{analysis problem}:

\begin{equation*}
    \lambda_{\mathbf{y}} \in \argmin_{\lambda \in \Delta_{S-1}} d_a^2(\mathbf{y}, \mathbf{x}_\lambda),
\end{equation*}
where $d_a(\cdot,\cdot)$ is a discrepancy function (e.g., $d_a = dist$ or another metric in the space or even a weaker notion of dissimilarity that still identifies indiscernibles). We refer to the collection of points generated via \eqref{eq: syn} for fixed templates $\{\mathbf{x}^s\}_{s=1}^S$ as a \emph{barycentric coding model} (BCM). Consequently, the barycenter of a set of templates for given weights $\lambda$ provides an alternative way of combining elements, one that better reflects the underlying geometry of the space while still preserving the dimensionality reduction properties of the representation.  This framework has been extensively used in classical optimal transport (OT), where the space of probability measures is endowed with the Wasserstein distance \cite{Santambrogio-OTAM,Villani2003Topics, Villani2009Optimal}. See Section \ref{sec: OT} in Suppl. Mat.  Foundational works on Wasserstein barycenters include \cite{agueh2011barycenters, cuturi2014fast, bonneel2016wasserstein, le2017existence, zemel2019frechet}, with more recent developments such as \cite{werenski2022measure, gunsilius2024tangential, mallery2025synthesis, werenski2025linearized} and extensions to dictionary learning \cite{schmitz2018wasserstein, mueller2023geometrically}. 

While classical OT relies on a shared underlying domain to compare probability measures, many modern applications involve distributions supported on different, potentially non-Euclidean, metric spaces. In such settings, Gromov-Wasserstein (GW) theory offers a natural generalization, enabling meaningful comparisons by aligning the intrinsic geometries of the underlying spaces themselves. The GW distance \cite{memoli2011gromov, memoli2004comparing} has emerged as a powerful tool in object matching problems, as it compares probability measures across different metric spaces by jointly aligning both structural (edge-based) and nodal (feature-based) similarities \cite{hendrikson2016using}. This perspective builds upon classical notions of shape similarity--such as the Hausdorff distance, the Wasserstein distance, and the Hausdorff-Wasserstein distance \cite{memoli2011gromov}--and can be viewed as a measure-theoretic relaxation of the Gromov-Hausdorff distance introduced in \cite{gromov2001metric}. 

These innovations have led to applications across diverse domains, ranging from quantum chemistry \cite{peyre2016gromov} to natural language processing \cite{alvarez2018gromov}, enhancing fundamental tasks such as clustering \cite{Chowdhury2021}, dimensionality reduction \cite{van2024distributional}, shape correspondence \cite{Kong2024}, shape analysis \cite{memoli2022distance}, computer vision \cite{schmitzer2013modelling}, and inference-based simulation \cite{hur2024reversible}.
Extensions to unbalanced settings have been explored \cite{bai2024efficient, chapel2020partial, Liu2020, sejourne2021unbalanced, zhang2022cycle}, along with other variants, such as Sliced GW \cite{titouan2019sliced}, and Fused GW--a simultaneous extension of Wasserstein and GW distances--\cite{vayer2020fused, vayer2020contribution, thual2022aligning}.
However, 
GW distance computation is challenging because it requires solving a non-convex quadratic program; approximations include linearized \cite{beier2022linear} and entropy-regularized  \cite{peyre2016gromov, solomon2016entropic, rioux2024entropic} GW.

In the context of GW, data consist of different configurations of nodes with weighted edges. Thus, instead of working with measures $\mu$ defined on a common underlying space, as in classical OT, we consider objects of the form $\mathbb X ={(X, \omega_{X}, \mu_{X})}$, where $(X, \mu_{X})$ is a probability space and $\omega_{X}: X \times X \to \mathbb{R}$ is a $\mu_{X} \otimes \mu_{X}$-measurable function. The function $\omega_{X}$ may represent a metric on $X$, in which case $\mathbb X$ is referred to as a \emph{metric measure space} (mm-space). More generally, $\omega_{X}$ can be any function encoding \emph{edge weights}, in which case the 3-tuple is called a \emph{network} or a \emph{shape}. If the set $X$ consists of finitely many points or \emph{nodes}, then $\omega_X$ and $\mu_X$ admit matrix and vector representations, denoted by $\mathbf{X}$ and  $\p$, respectively. For short, we will sometimes use $(\mathbf X, \p)$ to encode ${(X, \omega_{X}, \mu_{X})}$. We refer to the space of networks endowed with the GW metric (see Section \ref{sec: background}) as  \emph{GW space}.

Using these objects, the \emph{GW synthesis problem} takes the form of \eqref{eq: syn} with the distance $dist(\cdot,\cdot)$ replaced by the GW metric:  given  templates $\{\mathbb X^s\}_{s=1}^S$, the goal is to solve

\begin{equation}\label{eq: syn gw}
     \Delta_{S-1} \ni \lambda \longmapsto \mathbb{Y}_\lambda\in\argmin_{\mathbb Y}\sum_{s=1}^S\lambda_s \, GW^2(\mathbb X^s,\mathbb Y).
\end{equation}
We call the collection generated by \eqref{eq: syn gw}, for fixed templates and varying coefficients $\lambda \in \Delta_{S-1}$,  a \emph{GW barycenter space} or  \emph{Gromov-Wasserstein barycentric coding model} (GW-BCM).

In this article, we focus on estimating the unknown GW barycentric coordinates of a point (or, more precisely, a network) under the barycentric coding model within the Gromov-Wasserstein framework (GW-BCM), a task we refer to as the \emph{GW analysis problem}.
That is, given a network $\mathbb Y$ we aim to formalize and solve

\begin{equation}\label{eq: analysis gw}
    \lambda_{\mathbb{Y}} \in \argmin_{\lambda \in \Delta_{S-1}} d_a^2(\mathbb{Y}, \mathbb{Y}_\lambda),
\end{equation}
where $\mathbb{Y}_\lambda$ is generated by the GW-BCM \eqref{eq: syn gw}, 
for some appropriate dissimilarity metric $d_a(\cdot, \cdot)$.
In other words, assuming a given network is a GW barycenter of a set of templates $\{\mathbb X^s\}_{s=1}^S$, our goal is to recover the corresponding vector of weights $\lambda$. See Fig. \ref{fig: GW bary space} for an illustration. 

\begin{figure}[ht!]
    \centering

    \includegraphics[width=1\linewidth]{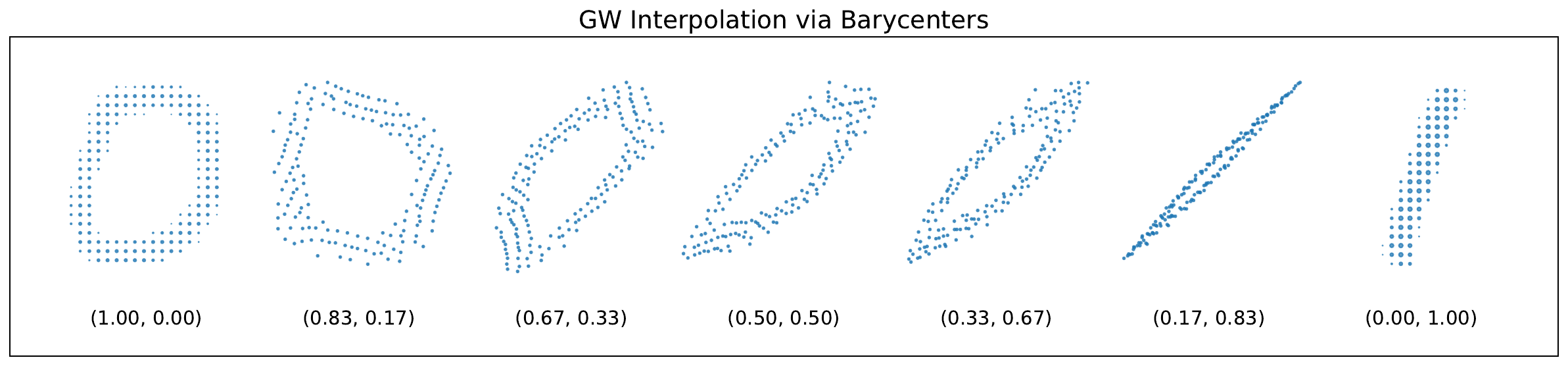}

    \includegraphics[width=1\linewidth]{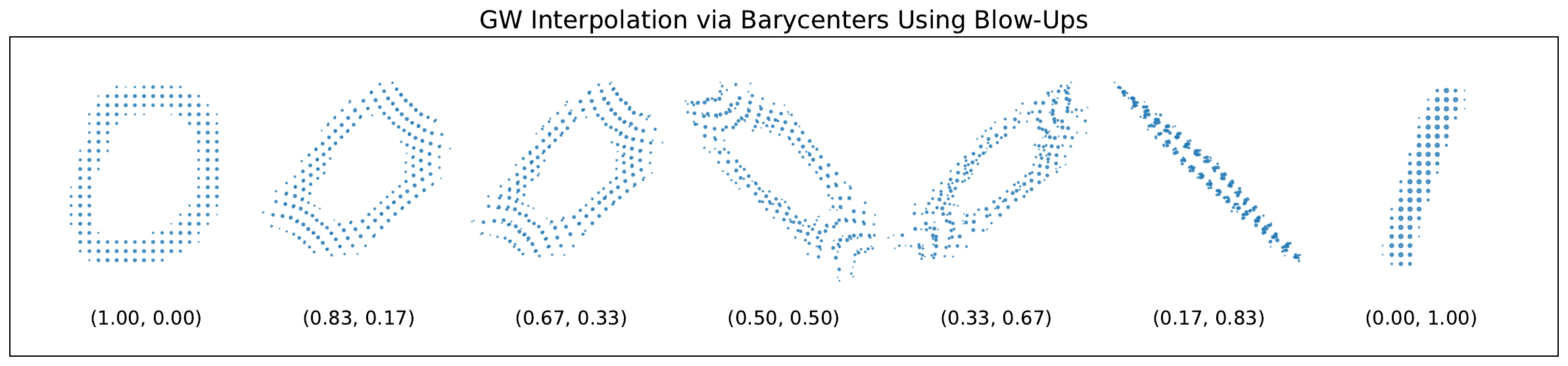}

    \caption{\small{Illustration of GW barycenters  between two point clouds from the Point Cloud MNIST 2D dataset \cite{Garcia2023PointCloudMNIST2D}, shown  for various $t \in [0,1]$ with interpolation coordinates $(1-t, t)$ on the horizontal axis. The top subplot uses the function \texttt{ot.gromov.gromov\_barycenters} from the POT Library \cite{flamary2021pot} (based on \cite{peyre2016gromov}).  The bottom subplot uses the blow-up technique from \cite{chowdhury2019gromov}, which appropriately realigns the nodes and enlarges the template matrices $\mathbf{X}^0$ (zero shape) and $\mathbf{X}^1$ (one shape), creating new versions $\mathbf{X}^0_b$ and $\mathbf{X}^1_b$ of the same size, so that the GW barycenters can be interpreted as convex combinations $(1-t)\mathbf{X}^0_b + t\mathbf{X}^1_b$ (see Remark \ref{rem: geod} in Suppl. Mat. \ref{app: weak}). MDS embedding is used for visualization.}}
    \label{fig: GW bary space}

\end{figure} 

In this work, the synthesis problem--constructing GW barycenters--is addressed primarily to support the formulation of the GW-BCM assumption and the corresponding analysis problem. 
Since the minimization in \eqref{eq: syn gw} ranges over a very large \emph{space of spaces} $\mathbb Y$, we restrict to spaces with finite many nodes for tractability.
We further consider simplified versions of the general GW barycenter space generated by \eqref{eq: syn gw} (see also \eqref{eq: bary space very general}).
In the first simplification (\eqref{eq: bary space for M and q fixed} in Section \ref{sec: fixed point approach}), inspired by \cite{peyre2016gromov}, we fix the cardinality of the set $Y$ and the measure $\mu_Y$ in $\mathbb{Y} = (Y, \omega_Y, \mu_Y)$. In the second (\eqref{eq: bary bu} in Section \ref{sec: analysis grad}), we focus on critical points of the objective in \eqref{eq: syn gw} as in \cite{chowdhury2020gromov}. Relationships are discussed in Section \ref{sec: relation}.

From a machine learning perspective, a substantial literature has developed on the use of the GW distance for graph representation. In particular, the works \cite{B-xu2020gromov, A-xu2022representing} introduce GW factorization models in which each graph is encoded as a barycentric combination of learned graph atoms, yielding permutation-invariant embeddings. Related extensions and applications can be found in \cite{C-zhou2025fedgf}.
In these approaches, the primary goal is to design neural network architectures that perform GW dictionary learning, so that the atoms (references or templates) are learned jointly with the associated weights. Consequently, the emphasis lies on the latent representation space and efficient learning algorithms.  In supervised graph prediction \cite{E-brogat2022learning}, fused-GW barycenters models are used. A closely related line of work based on an online dictionary-learning framework \cite{D-vincent2021online} similarly trains a dictionary of GW and fused-GW atoms and embeds graphs in the corresponding atom-weight space by using a linear approach. Indeed, the model is a relaxation used for computational tractability: each data point $\mathbb Y$ is approximated as a convex combination of learned atoms with respect to GW distance ($\min_{\lambda\in \Delta_{S-1}} GW(\mathbb Y, \sum_{s=1}^S \lambda_s \mathbb X_s)$), where the authors explicitly impose the condition that the dictionary consists of pairwise relation matrices of the same size, but meanwhile, each dataset graph $\mathbb Y$ can have its own size.    
These frameworks consider an empirical form or, more precisely, a deep learning form of the analysis problem in the latent space, which is in contrast to our geometric and analytic approaches that pertain to the $\lambda$ rather than a latent representation thereof and in particular provide exact expressions for recovering $\lambda$ in certain cases. 
On the one hand, our framework is simpler in that the templates are fixed in advance; on the other hand, our aim is to characterize the analysis problem (barycentric weights recovery) through closed-form expressions.
Our methodology is closer to that of \cite{bonneel2016wasserstein}, but with the GW geometry replacing the Wasserstein one. In their setting, the atoms (histograms) are fixed a priori, and the barycentric weights are obtained by solving OT-based optimization problems, rather than via learned encoders. Regression is then performed in the barycentric coordinate space through classical optimization, not by backpropagation through a deep model.

\subsection{Summary of Contributions}
We primarily consider finite spaces and propose two distinct approaches to tackle the GW analysis problem \eqref{eq: analysis gw}: 

\begin{enumerate}[leftmargin=*] 
    \item \, A method developed from a fixed-point scheme for the GW synthesis problem \eqref{eq: syn gw}, relying on matrix optimization over the networks' weight matrices $\mathbf{Y}$.

    \item \, A gradient-based method motivated interpreting GW barycenters not necessarily as exact minimizers of \eqref{eq: syn gw}, but rather as critical points of a functional on the GW space.

\end{enumerate}

 Regarding (1), we adopt the following simple idea, which can be applied beyond the context of GW.  
Consider a finite set of templates $\mathbb X^s=(\mathbf{X}^s,\p^s)$, $s=1,\dots, S$. Given a weight vector $\lambda\in \Delta_{S-1}$, consider GW synthesis problem \eqref{eq: syn gw}, where the minimization is performed over $\mathbb Y=(\mathbf Y, \q)$, with fixed matrix size $M$ for $\mathbf{Y}$ and a fixed probability vector $\q$ \cite{peyre2016gromov}. 
Assume that the GW synthesis problem \eqref{eq: syn gw} can be solved via a fixed-point iteration of the form $\mathbf Y_{n}=\rho_\lambda(\mathbf{Y}_{n-1})$.
Then, if $(\mathbf{Y}, \q)$ lies in the GW barycenter space generated by the templates, the matrix $\mathbf{Y}$ must be a fixed point of the corresponding iterative scheme. That is, there exists an unknown `coordinate' vector  $ \lambda_{\mathbf{Y}}$ such that  $\mathbf{Y} = \rho_{\lambda_{\mathbf{Y}}}(\mathbf{Y})$. To recover $\lambda_{\mathbf{Y}}$, we differentiate a chosen matrix norm $\|\mathbf{Y} - \rho_{\lambda}(\mathbf{Y})\|$ with respect to $ \lambda $ and find its minimum at $\lambda=\lambda_{\mathbf Y}$.  
In summary, to generate GW barycenters, we employ a simplified version of the iterative scheme proposed in \cite{peyre2016gromov}, and show in Theorem \ref{prop: formula_bary} that elements of the barycenter space are fixed points of this iteration. Building on this, Theorem \ref{thm: analysis} provides an exact recovery method for $\lambda_{\mathbf{Y}}$, for any $\mathbf{Y}$ in the GW barycenter space, formulated as a convex quadratic program and implemented in Algorithm \ref{alg: analysis}. See Fig. \ref{fig: baryspace01} for an illustration of its applicability.  

Regarding (2), we interpret GW barycenters $\mathbb{Y}_\lambda$ as \emph{Karcher means} of the template set $\{\mathbb{X}^s\}_{s=1}^S$ with respect to the weights $\lambda$, in the sense that $\mathbb{Y}_\lambda$ is a critical point of the objective functional in the GW synthesis problem \eqref{eq: syn gw}--that is, it satisfies the first-order condition \cite{karcher1977riemannian}.
To define a notion of gradients in GW space, we leverage a `Riemannian-like structure', that endows the space of networks equipped with the GW distance with a tangent space framework \cite{chowdhury2019gromov, chowdhury2020gromov, sturm2023space}. 
To this end, we employ the \emph{blow-up} technique introduced in \cite{chowdhury2020gromov} in the context of finite spaces. This approach enables the comparison of networks $\mathbb X=(\mathbf X,\p)$ and $\mathbb Y=(\mathbf Y,\q)$ of different sizes (i.e., where $\mathbf X$ is an $N\times N$ matrix and $\mathbf Y$ is an $M\times M$ matrix potentially with  $M\neq N$) by converting GW OT plans into maps between new networks $\mathbb X_b=(\mathbf X_b,\p_b)$ and $\mathbb Y_b = (\mathbf Y_b,\q_b)$ of the same size. These new networks are \emph{weakly isomorphic} (see Definition \ref{def: weak iso}) to the original ones, and thus the GW dissimilarity is preserved: $GW(\mathbb X, \mathbb Y) = GW(\mathbb X_b,\mathbb Y_b).$  After establishing how to define directional derivatives and gradients in GW space, and drawing inspiration from \cite{werenski2022measure, werenski2025linearized}, we solve the GW analysis problem \eqref{eq: analysis gw} as another convex quadratic program; see Theorem \ref{prop: gradient method} and the corresponding implementation in Algorithm \ref{alg: analysis blow up}. 

\begin{figure}[ht!]
    \centering
    \includegraphics[width=0.8\linewidth]{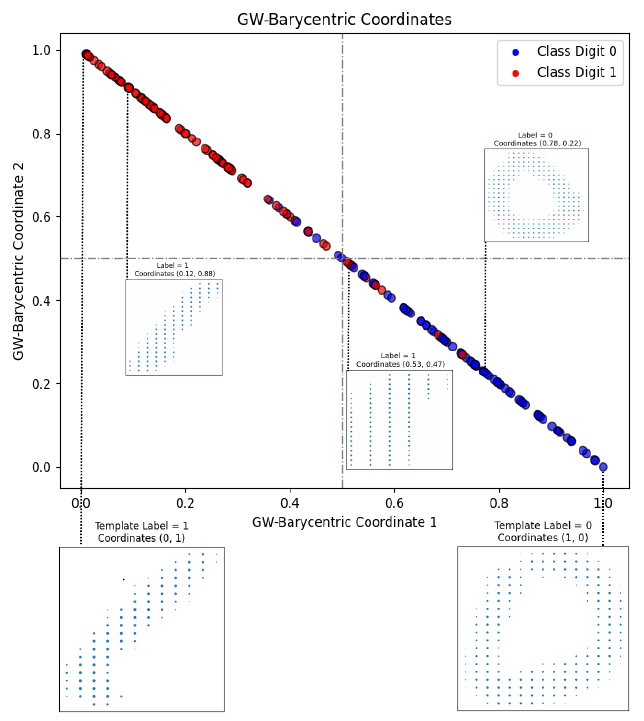}

    \caption{\small{Visualization of the GW barycenter space as GW barycenter coordinates, using the Point Cloud MNIST 2D dataset \cite{Garcia2023PointCloudMNIST2D}. Two point clouds serve as templates: one for digit 0s at $(1,0)$ and one for digit 1s at $(0,1)$. Using our proposed Algorithm \ref{alg: analysis}, we compute GW barycenter coordinates for 200 random samples and plot them with blue dots for label 0 and red dots for label 1. Additionally, three corresponding point clouds are shown for illustration. Point clouds clearly corresponding to a digit 0 lie closer to the 0 template (coordinates near (1,0)), while those representing digit 1 are closer to the 1 template (coordinates near (0,1)).}}
    \label{fig: baryspace01}
\end{figure}

\subsection{Organization of the Paper} Section \ref{sec: background} reviews the GW distance. Our main results, their proofs, and new algorithms are presented in Sections \ref{sec: fixed point approach} and \ref{sec: analysis grad}, which address the GW analysis problem via two approaches: a fixed-point method (Section \ref{sec: fixed point approach}) and a gradient-based method (Section \ref{sec: analysis grad}).
Section \ref{sec: experiments} presents experiments on 2D and 3D point clouds that serve as a proof-of-concept for the analytical results. Conclusions follow in
Section \ref{sec: conclusions}.  Supplementary material includes: additional context on the BCM framework in Section \ref{sec: BCM}; background proofs and examples in Suppl. Mat. \ref{app: generalities}; computational complexity of the proposed GW analysis algorithms in Suppl. Mat. \ref{app: comp comp}; further experimental results in Suppl. Mat. \ref{app: experiments}; experimental exploration beyond the GW-BCM hypothesis (i.e., for targets outside the GW barycenter space) \ref{sec supp: projection}; and an additional discussion on the computation of GW  and  $W_2$ barycenters in Suppl. Mat. \ref{app: GW Was}.
The code to reproduce the experiments of this paper is available
online at \url{https://github.com/RocioDM/GW_BCM}.

\subsection{Notation} 

In general, superscripts will be used for the templates and subscripts for the synthesized objects.  We denote by $[S]:=\{1,\dots,S\}$ the set of indices from $1$ to $S$. The space of \emph{weights} $\Delta_{S-1}$, also known as a \emph{standard simplex} or \emph{probability simplex}, is $\Delta_{S-1}:=\{\lambda=(\lambda_1,\dots,\lambda_S)\in \mathbb R^S\mid \lambda_s\geq 0 \,\,\,\, \forall s\in[S], \text{ and } \sum_{s\in [S]}\lambda_s=1\}.$  If $\p$ is a vector and $\mathbf{X}$ a matrix, when there is no ambiguity regarding whether we are referring to entries or coordinates, we will write $\p[i] = \p_i$ and $\mathbf{X}[i,j] = \mathbf{X}_{ij}$. 
 We recall that the \emph{Frobenius norm} of a real-valued square matrix $\mathbf{X}$ is defined as $\| \mathbf X\|_{\mathrm{Frob}}:=\sqrt{\mathrm{tr}( \mathbf X^T  \mathbf X)}$, where $\mathbf X^T$ stands for transpose.  The component-wise (Hadamard) product between matrices is denoted by $ (\mathbf X\odot  \mathbf Y)_{jl}:=\mathbf X_{jl} \mathbf Y_{jl}.$  For a vector $\q$, the matrix $\frac{1}{\q\q^T}$ is defined by $\left(\frac{1}{\q\q^T}\right)[{j,l}]:=\frac{1}{\q_j\q_l}.$

Given two (Borel) measure spaces $(X,\mu_X)$ and $(Y,\mu_Y)$, we say that the measure $\mu_Y$  is the \emph{pushforward} of the measure $\mu_X$ by a measurable map $\varphi : X \to Y$, denoted by $\varphi_\#\mu_X = \mu_Y$, if $\mu_Y(A) = \mu_X(\{x \in X \mid \varphi(x) \in A\})$ for every measurable set $A \subseteq Y$. 
Given $\omega:Y\times Y\to \R$, we  use the \emph{pullback} notation $\varphi^*\omega(x,x') := \omega(\varphi(x),\varphi(x'))$, for all $x,x'\in X$.  A \emph{coupling} or \emph{transport plan} $\pi$ between $\mu_X$ and $\mu_Y$, denoted $\pi\in \Pi(\mu,\nu)$, is a probability measure on the product space $X\times Y$ with marginals $\mu_X$, $\mu_Y$. In the case of finite discrete measures,  $\mu_X=\sum_{i=1}^N\p_i\delta_{x_i}$ and $\mu_Y=\sum_{j=1}^M\q_j\delta_{y_j}$, each coupling can be identified with a matrix in $\Pi(\p,\q):=\{\pi\in\R^{N\times M}\mid \, \pi_{ij}\geq 0, \, \forall i,j \, ; \,  \sum_{j\in[M]}{\pi_{ij}}=\p_i \, \forall i \, \text{ and } \sum_{i\in [N]} \pi_{ij}=\q_j \, \forall j\}.$

\section{Gromov-Wasserstein (GW) Distances}\label{sec: background}\label{sec: intro GW}

The Gromov-Wasserstein (GW) distance endows the space of metric measure spaces with a well-defined metric structure. Moreover, it serves as a discrepancy  measure for comparing more general structures. 

\begin{definition}[\cite{memoli2011gromov,beier2022linear, sturm2023space, chowdhury2019gromov}]\label{def: mm gm net}
    Let $(X, \mu_X)$ be a probability space, where $X$ is a Polish space\footnote{(i.e., a separable and completely metrizable topological space equipped with the Borel $\sigma$-algebra)}, and $\mu_X$ is a fully supported Borel probability measure on $X$.
    \begin{enumerate}[leftmargin=*]
        \item If the set $X$ is endowed with a metric $d_X: X \times X \to \mathbb{R}_{\geq 0}$, then the triple $\mathbb X=(X, d_X, \mu_X)$ is called a \emph{metric measure space} (mm-space).

        \item If $g_X$ is a symmetric function in $L^2(X \times X, \mu_X \otimes \mu_X)$ (i.e., $g_X(x, x') = g_X(x', x)$), then the triple $\mathbb X=(X, g_X, \mu_X)$ is called a \emph{gauge measure space} (gm-space).\footnote{For example, one may take squared-distances $g_X(x, x') = d_X^2(x, x')$.}

        \item If $\omega_X: X \times X \to \mathbb{R}$ is any square-integrable function, i.e., $\omega_X \in L^2(X \times X, \mu_X \otimes \mu_X)$, then the triple $\mathbb{X} = (X, \omega_X, \mu_X)$ is called a \emph{network} (or \emph{shape}).
    \end{enumerate}    
\end{definition}

The natural hierarchy among the above structures can be expressed by the following chain of inclusions: $\text{mm-spaces}\subset \text{gm-spaces}\subset\text{networks}$.

\begin{definition}[\cite{memoli2011gromov, chowdhury2019gromov}] Let $1\leq p<\infty$.
 The $p$-\emph{Gromov-Wasserstein problem} between two networks $\mathbb X =(X, \omega_X,\mu_X)$ and $\mathbb Y=(Y, \omega_Y,\mu_Y)$ is formulated as\footnote{(for simplicity, assume $X$ and $Y$ to be compact, so that \eqref{eq: gw for networks} takes a finite value)}:

    \begin{equation}\label{eq: gw for networks}
    GW(\mathbb X, \mathbb Y)=\inf_{\pi\in \Pi(\mu_X,\mu_Y)}{\left(\int_{X\times Y}\int_{X\times Y}|\omega_X(x,x')-\omega_Y(y,y')|^p d\pi(x,y)d\pi(x',y')\right)^{\frac{1}{p}}}.
\end{equation}
\end{definition}

Roughly speaking, if $\pi(x,y)$ encodes the mass transported from location $x\in X$ to the point $y\in Y$, we can interpret that \eqref{eq: gw for networks} optimizes over all possible plans $\pi$ so that when transporting mass
from the points $x$ and $x'$ to $y$ and $y'$, the internal weights/dissimilarities $\omega_X(x,x')$ and $\omega_Y(y,y')$ are preserved or are as close as possible. The typical choice is $p=2$, and for simplicity, we will assume this case. Accordingly, we denote $GW_2(\cdot,\cdot)$ simply as $GW(\cdot,\cdot)$.
More general loss functions beyond $|\cdot|^p$, such as the KL-divergence \cite{peyre2016gromov} or the inner product GW problem \cite{zhang2024gradient}, have also been considered in the literature.

\begin{definition}[\cite{chowdhury2019gromov}]\label{def: strong iso}\label{def: weak iso}
 We say that two mm-spaces $\mathbb X=(X,d_X,\mu_X)$,  $\mathbb Y=(Y,d_Y,\mu_Y)$ are \emph{strongly isomorphic}, denoted $\mathbb{X} \sim \mathbb{Y}$, if there exists a Borel measurable surjective isometry $\varphi: X \to Y$ ($d_X(x,x') = d_Y(\varphi(x),\varphi(x'))$) with Borel measurable inverse such that
$\varphi_\# \mu_X = \mu_Y$.  We say that two networks $\mathbb{X} = (X,\omega_X,\mu_X)$, $\mathbb{Y} = (Y,\omega_Y,\mu_Y)$ are \emph{weakly isomorphic}, denoted $\mathbb{X} \sim^w \mathbb{Y}$, if there exists a Borel probability space $(Z, \mu_Z)$ and measurable maps $\varphi_X: Z \to X$ and $\varphi_Y: Z \to Y$ such that  $(\varphi_X)_\# \mu_Z = \mu_X$,  $(\varphi_Y)_\# \mu_Z = \mu_Y$, and $\| (\varphi_X)^* \omega_X - (\varphi_Y)^* \omega_Y \|_\infty = 0$.

\begin{small}
\begin{equation*}
  \xymatrix{
&(Z,\mu_Z) \ar[rd]^{\varphi_Y} \ar[ld]_{\varphi_X}&\\
(X,\omega_X,\mu_X) \ar@{<-->}[rr]^{\sim^w}   &  & (Y,\omega_Y,\mu_Y) } 
\end{equation*}    
\end{small}
\end{definition}

\noindent For a simple visualization of weak isomorphic networks, see Figure \ref{fig: weakiso} (and also Figure \ref{fig: weakiso_2} described by Example \ref{example: weakiso} in Suppl. Mat. \ref{app: generalities}).
It is clear that if two networks are strongly isomorphic, then they are also weakly isomorphic. Moreover, if two networks are weakly isomorphic and have the structure of mm-spaces, then they are strongly isomorphic \cite{memoli2011gromov,chowdhury2017distances} (see Proposition \ref{prop: weak+ metric structure implies strong} in Suppl. Mat. \ref{app: weak} for the finite case).

\begin{SCfigure}[5][ht!]
  \centering
  \includegraphics[width=0.22\textwidth]{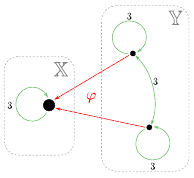} 
  \caption{\small{Illustration of weakly isomorphic networks $\mathbb X$ and $\mathbb Y$.  The network $\mathbb X$ consists of one point (i.e., $X$ is a singleton) with unit mass (i.e., $\mu_X$ is the delta-measure),  and $\omega_X$ represented by the $1\times 1$ matrix $\mathbf{X}=(3)$. The network $\mathbb Y$ is given by the probability space $(Y,\mu_Y)$ consisting of two points with equal mass, and $\omega_Y$ represented by the $2\times 2$ matrix $\mathbf{Y}=\left(\begin{smallmatrix}
3 & 3 \\
3 & 3
\end{smallmatrix}\right)$. The pivoting space $(Z,\mu_Z)$ in Definition \ref{def: weak iso} can be taken as $(Y,\mu_Y)$, with $\varphi_X=\varphi$ (red assignment) and $\varphi_Y$ the identity map.
     }}
  \label{fig: weakiso}
\end{SCfigure}

Consider the equivalence relation on the space of networks given by $\mathbb X\sim \mathbb X'$ if and only if $GW(\mathbb X, \mathbb X')=0.$  Then, $GW(\cdot,\cdot)$  defines a metric on the quotient space of  mm-spaces under this equivalence relation. Moreover, this equivalence relation coincides with the notion of \emph{strong isomorphism} \cite[Thm. 5.1]{memoli2011gromov}.
The extension of the GW discrepancy from mm-spaces to general networks is not necessarily a metric modulo strong isomorphisms \cite[Rem. 2.5 and Fig. 2]{chowdhury2019gromov}. However, in analogy to Gromov-Hausdorff distance \cite{chowdhury2017distances}, in \cite[Thms. 2.3 and 2.4]{chowdhury2019gromov} the authors prove that GW in \eqref{eq: gw for networks} is a pseudo-metric in the space of networks (i.e, non-negative, symmetric and satisfies the triangle inequality). Moreover, $GW(\mathbb X,\mathbb Y)=0$ if and only if $\mathbb X$ and $\mathbb Y$ are \emph{weakly isomorphic}. As a summary, we provide the following diagram:
\begin{small}
    \begin{align*}
\begin{array}{c@{\quad}c}
\mathbb{X}, \mathbb{Y} \textbf{ mm-spaces} & \mathbb{X}, \mathbb{Y} \textbf{ networks} \\[0.4em]
\xymatrix{
& GW(\mathbb{X}, \mathbb{Y}) = 0 \ar@{<=>}[rd] \ar@{<=>}[ld] & \\
\mathbb{X} \sim \mathbb{Y} \ar@{<=>}[rr] & & \mathbb{X} \sim^w \mathbb{Y}
}
&
\xymatrix{
& GW(\mathbb{X}, \mathbb{Y}) = 0 \ar@{<=>}[rd] \ar@{<=}[ld] & \\
\mathbb{X} \sim \mathbb{Y} \ar@{=>}[rr] & & \mathbb{X} \sim^w \mathbb{Y}
}
\end{array}
\end{align*}
\end{small}

In this paper, we consider general \emph{networks} as described in Definition \ref{def: mm gm net}, up to the equivalence relation induced by \emph{weak isomorphisms}. Specifically, let $\mathcal{GW}$ denote the space of equivalence classes of networks under the weak isomorphism relation $\sim^w$. Given a network $\mathbb{Y} = (Y, \omega_Y, \mu_Y)$, we denote its equivalence class by $[\mathbb{Y}]$.

\subsection{GW on Finite Spaces} We follow the approach in \cite[Sec. 7]{memoli2011gromov}, where the author introduces a computationally motivated formulation of the GW problem (see also \cite{peyre2016gromov}).

From now on, we focus on \emph{finite} spaces $X = \{x_i\}_{i=1}^{N_{\mathbf{X}}}$, where $N_{\mathbf{X}} \in \mathbb{N}$ denotes the finite number of nodes. Finite networks can be represented as square matrices equipped with a probability measure on the columns (or equivalently, on the rows). Specifically, we encode the information of the triple $(X, \omega_X, \mu_X)$ as a pair $(\mathbf{X}, \mathrm{p}_{\mathbf{X}})$, where:
\begin{itemize}[leftmargin=*]

    \item $\mathrm p_{\mathbf X}=(\mathrm p_{\mathbf X}[1],\dots,\mathrm p_{\mathbf X}[N_{\mathbf X}])\in \mathbb R^{N_{\mathbf X}}$, with $\mathrm p_{\mathbf X}[j]> 0$  $\forall j\in[N_{\mathbf X}]$ and $\sum_{j\in [N_{\mathbf X}]}\mathrm p_{\mathbf X}[j]=1$, encodes the discrete measure $\mu_X$ supported on the $N_{\mathbf X}$ points; that is,  $\mu_X=\sum_{j=1}^{N_{\mathbf X}} \mathrm p_\mathbf{X}[j]\delta_{x_j}.$
    \item $\mathbf X\in\R^{N_{\mathbf X}\times N_{\mathbf X}}$ is a real  $N_{\mathbf X}\times N_{\mathbf X}$ matrix, sometimes referred to as a  \emph{dissimilarity matrix}, where each entry $(i, j)$ represents the interaction between nodes $x_i$ and $x_j$; that is, $\omega_X(x_i, x_j) = \mathbf{X}_{ij}.$
\end{itemize}
To simplify the notation, we omit the subscript and simply write $\mathrm p$ instead of $\mathrm p_\mathbf{X}$, and $N$ instead of $N_{\mathbf X}$. For convenience, we denote the space of such probability vectors $\p$ by $\mathcal{P}_{N}$.

To match the definition of mm-spaces, $\mathbf X$ should correspond with a distance matrix, 
$\mathbf{X}_{ij}= d_X(x_i, x_j)$, $\forall i, j \in [N]$, must therefore satisfy certain structural properties (at least, symmetry, non-negativity, and a zero diagonal). In this work, we adopt a more general abstraction where 
$\mathbf X$ need not represent a distance matrix and may encode arbitrary pairwise relations\footnote{This framework enables applications  in settings such as graphs, where $\mathbf X$ could represent adjacency matrices or Laplacian matrices, thereby broadening the scope of potential use cases.}. 

Given $(\mathbf X,\p)\in \R^{N\times N}\times\mathcal{P}_N$, $( \mathbf Y, \q)\in \R^{M\times M}\times\mathcal{P}_M$, the GW optimization problem in \eqref{eq: gw for networks}  adopts the form:

\begin{equation}\label{eq: rest_gw}
    GW((\mathbf X,\p),( \mathbf Y, \q))^2=\min_{\pi\in\Pi(\p,\q)}
    \sum_{i\in[N],j\in[M]}\sum_{k\in[N],l\in[M]}|\mathbf X_{ik}- \mathbf Y_{jl}|^2\pi_{ij}\pi_{kl}.
\end{equation}
It is a quadratic optimization problem with linear constraints, with no guarantees to be convex in general, and a minimizer is always attained \cite{memoli2011gromov}. 
When considering symmetric positive (semi)definite matrices, \eqref{eq: rest_gw} is convex. For completeness, we provide a proof of this result in Proposition \ref{prop: convex} in Suppl. Mat. \ref{app: GW dist}. However, we note that distance matrices $\mathbf X_{ij}=d_X(x_i,x_j)$ are not generally positive semidefinite.

\section{Fixed-Point Iteration for Gromov-Wasserstein Barycenters}\label{sec: fixed point approach}

In this section, we begin by revisiting one of the most widely used approaches for synthesizing GW barycenters  \cite{peyre2016gromov}, but without incorporating any regularization term (Section \ref{sec: Synthesis}). We then show that this iterative scheme constitutes a fixed-point iteration (Theorem \ref{prop: formula_bary} and its Corollary \ref{coro: fixed point}). Building on this, in Section \ref{sec: Analysis} we present our new contributions on solving the GW analysis problem, including Theorem \ref{thm: analysis} and Algorithm \ref{alg: analysis}.

\subsection{GW Barycenters via a Fixed-Point Approach: The Synthesis Problem}\label{sec: Synthesis} 

\noindent Consider finitely many templates $( \mathbf X^s, \p^s)\in \R^{N^s\times N^s}\times\mathcal{P}_{N^s}$, for $s\in [S]$, where $S\in\N$. 
The GW synthesis problem studied in \cite{peyre2016gromov} can be formulated as follows.  For $M\in\N$ and $\q\in\mathcal{P}_M$, consider the following restricted version of \eqref{eq: syn gw}:

\begin{equation}\label{eq: solomon_barycenter}
    \Delta_{S-1}\ni\lambda\longmapsto  \mathbf Y_\lambda \in \argmin_{ \mathbf Y\in\R^{M\times M}}\sum_{s\in[S]}\lambda_s \, GW(( \mathbf X^s, \p^s),( \mathbf Y, \q))^2.
\end{equation}
See Prop. \ref{remark: existence min synth} in  Suppl. Mat. \ref{app: gw fp synth} for a discussion on the existence of a minimum in \eqref{eq: solomon_barycenter}. 
If the general synthesis problem \eqref{eq: syn} is  replaced by \eqref{eq: solomon_barycenter}, we define the following barycentric space.

\begin{definition}\label{def: bary space}
    Let $\{( \mathbf X^s, \p^s)\in \R^{N^s\times N^s} \times \mathcal{P}_{N^s}\}_{s\in[S]}$ be a finite set of templates. For $M\in \N$ and $\q\in \mathcal{P}_M$ a probability vector with all entries strictly positive, we define:

    \begin{equation}\label{eq: bary space for M and q fixed}
        \mathrm{Bary}_{M,\q}(\{(\mathbf X^s,\p^s)\}_{s\in [S]}):=\{\mathbf{Y}_\lambda \text{ as in }\eqref{eq: solomon_barycenter}\mid \, \lambda\in \Delta_{S-1}\}.
\end{equation}    
 That is, $ \mathrm{Bary}_{M,\q}(\{(\mathbf X^s,\p^s)\}_{s\in [S]})$ is the set of solutions $\mathbf{Y}_\lambda$ of the GW synthesis problem \eqref{eq: solomon_barycenter} varying $\lambda\in \Delta_{S-1}$,  where the size $M$ of $\mathbf{Y}_\lambda$ is fixed, as well as $\mathrm q$.
\end{definition}

Notice that the problem in \eqref{eq: solomon_barycenter} involves minimizing over multiple variables:

  \begin{small}
  {\setlength{\jot}{0pt}
    \begin{align}\label{eq: multi_min}
  &\min_{ \mathbf Y\in\R^{M\times M}}\sum_{s\in[S]}\lambda_s GW((\mathbf X^s, \p^s),( \mathbf Y, \q))^2
    \notag\\[-3pt]
    &
    = \min_{ \mathbf Y\in\R^{M\times M}}\min_{\{\pi^s\in\Pi(\p^s,\q)\}_{s=1}^S} \sum_{s\in [S]}\lambda_s    \sum_{i\in[N^s],j\in[M]}\sum_{k\in[N^s],l\in[M]}|\mathbf X_{ik}^s- \mathbf Y_{jl}|^2\pi_{ij}^s\pi_{kl}^s .
    \end{align}}
    \end{small}
As mentioned in \cite{peyre2016gromov}, when the variables are treated separately, this optimization problem is convex with respect to $ \mathbf{Y}$, but not with respect to $\{\pi^{s}\}_{s=1}^S$ (it is quadratic, but the associated form is not necessarily positive semidefinite)\footnote{However, if the set of couplings $\{\pi^s\}_{s\in[S]}$ is not fixed in advance, then each $\pi^s$ may depend on $\mathbf{Y}$ (since it is optimal for $GW(( \mathbf X^s, \p^s),( \mathbf Y, \q))$), implying that \eqref{eq: multi_min} is not necessarily convex in $ \mathbf{Y}$. This is due to the general fact that while the supremum of convex functions is convex, the minimum is not; that is, even if $ x \mapsto h(x, y)$ is convex for each $y$, the function $g(x) = \min_y h(x, y)$ need not be convex.}.
In what follows we recall \cite[Prop. 3]{peyre2016gromov}\footnote{Although \cite[Prop. 3]{peyre2016gromov} allows for more general loss functions in the GW problem, we restrict our attention to the standard case $|\cdot|^2$ for clarity and simplicity.}.

\begin{proposition}[Prop. 3 \cite{peyre2016gromov} - Revisited proof in Suppl. Mat. \ref{app: gw fp synth}]\label{remark: bary_fix_pi}
Given $S\in \N$, consider templates  $( \mathbf X^s, \p^s)\in \R^{N^s\times N^s}\times\mathcal{P}_{N^s}$ for each $s\in [S]$, and $\lambda\in \Delta_{S-1}$. Let $M\in\N$ and let $\q\in \mathcal{P}_M$  be a probability vector. For each $s\in [S]$,
let $\pi^s$ be a coupling in $\Pi(\p^s,\q)$.
Consider the functional\footnote{If we were able to choose the couplings $\pi^s$ as the OT plans for each problem $GW((\mathbf{X}^s, \p^s), (\mathbf{Y}, \q))$, then the functional $J_\lambda$ could be interpreted as a simplified version of the objective in \eqref{eq: syn gw} (see  $G_\lambda$  in \eqref{eq: barycenter functional},  Section \ref{sec: analysis grad}) now acting over the space of matrices of fixed size, rather than over the space of networks, where the sizes of the dissimilarity matrices may vary and the probability distributions over the nodes are not predefined.} 
   \begin{equation}\label{eq: J functional}
      \R^{M\times M}\ni\mathbf Y\longmapsto  J_\lambda (\mathbf Y) :=
       \sum_{s\in [S]}\lambda_s    \sum_{i\in[N^s],j\in[M]}\sum_{k\in[N^s],l\in[M]}|\mathbf X_{ik}^s- \mathbf Y_{jl}|^2\pi_{ij}^s\pi_{kl}^s .
\end{equation}
    Then, the minimization problem $\displaystyle\min_{\mathbf Y\in \R^{M\times M}} J_\lambda (\mathbf Y)$ is well-defined. 
    Moreover, if all the entries of $\q$ are strictly positive\footnote{The assumption that all entries of $\q$ are strictly positive ensures that every node in the shape carries some mass, which is a natural requirement. Indeed, Definition \ref{def: mm gm net} already requires fully supported measures.}, there exists a unique minimizer $\mathbf Y^*\in \mathbb R^{M\times M}$ given by:

    \begin{equation}    
    \mathbf Y^*_{j  l}:= 
    \frac{1}{\q_{j}\q_{l}}\sum_{s\in [S]}\lambda_s \left((\pi^s)^T \mathbf X^s \pi^s\right)_{j l}, \qquad \forall l,j\in [M]. 
    \label{eq: bary_fix_pi}
\end{equation}
\end{proposition}

The following result provides a closed-form expression for the elements in \eqref{eq: bary space for M and q fixed}. 

\begin{theorem}\label{prop: formula_bary}
    Let $M\in\N$ and  $\q\in \mathcal{P}_M$ with all entries strictly positive. Let $ \mathbf Y_{\mathbf{\lambda}}\in  \mathrm{Bary}_{M,\q}(\{(\mathbf X^s,\p^s)\}_{s\in [S]})$ for some $\mathbf{\lambda}\in \Delta_{S-1}$. For each $s\in[S]$, let $\pi^{*s}$ be an optimal plan for $GW((\mathbf X^s,\p^s),( \mathbf Y_{\mathbf{\lambda}},\q))$, and consider 
    $\mathbf Y^*_{\mathbf{\lambda}}\in \R^{M\times M}$ as in \eqref{eq: bary_fix_pi}, i.e., $\mathbf Y^*_{\mathbf{\lambda}}:=\frac{1}{\q\q^T}\odot\sum_{s\in [S]}\mathbf{\lambda}_s (\pi^{*s})^T \mathbf X^s \pi^{*s}$.  
Then $\mathbf Y_{\mathbf{\lambda}}
 = \mathbf{Y}^*_{\mathbf{\lambda}}$.
\end{theorem}

\begin{proof}
 As in \eqref{eq: J functional}, consider the functional 
$J_\lambda$ on $\R^{M\times M}$, defined in this case by:

    \begin{equation}\label{eq: func_fix_pi}
         \mathbf Y\longmapsto J_\lambda (\mathbf Y):= \sum_{s\in [S]}\mathbf{\lambda}_s    \sum_{i\in[N^s], \, j\in[M]}\sum_{k\in[N^s], \, l\in[M]}|\mathbf X_{ik}^s-{ \mathbf Y}_{jl}|^2\pi_{ij}^{*s}\pi_{kl}^{*s}.
\end{equation}
    By Proposition \ref{remark: bary_fix_pi}, we know that it has a unique minimum given by $\mathbf Y^*_{\mathbf{\lambda}}$. Indeed, from the fact that $\q$ has all its entries strictly positive, it follows that $J_\lambda$ is strictly convex.
Therefore,

    \begin{align*}
        J_\lambda(\mathbf{Y}_\lambda^*)&\leq J_\lambda(\mathbf Y_\lambda) \qquad (\text{due to Proposition \ref{remark: bary_fix_pi}})\\
        &=\sum_{s\in [S]}\mathbf{\lambda}_s    \sum_{i\in[N^s],j\in[M]}\sum_{k\in[N^s],l\in[M]}|\mathbf X_{ik}^s-\mathrm (\mathbf Y_{\mathbf{\lambda}})_{jl}|^2\pi_{ij}^{*s}\pi_{kl}^{*s}\\
        &=\sum_{s\in[S]}\mathbf{\lambda}_s GW(( \mathbf X^s,\p^s),( \mathbf Y_{\mathbf{\lambda}}, \q))^2 \quad (\text{since } \pi^{*s} \text{ optimal for } GW(( \mathbf X^s,\p^s),( \mathbf Y_{\mathbf{\lambda}}, \q)))\\
        &=\min_{ \mathbf Y\in \R^{M\times M}}     \sum_{s\in[S]}\mathbf{\lambda}_s GW((\mathbf X^s,\p^s),( \mathbf Y, \q))^2 \quad (\text{since }  \mathbf Y_{\mathbf{\lambda}}\in  \mathrm{Bary}_{M,\q}(\{(\mathbf X^s,\p^s)\}_{s\in [S]}))\\
        &=\min_{ \mathbf Y\in\R^{M\times M}}\sum_{s\in [S]}\mathbf{\lambda}_s\min_{\pi^s\in\Pi(\p^s,\q)}    \sum_{i\in[N^s],j\in[M]}\sum_{k\in[N^s],l\in[M]}|\mathbf X_{ik}^s- \mathbf Y_{jl}|^2\pi_{ij}^s\pi_{kl}^s\\
        &\leq \min_{ \mathbf Y\in\R^{M\times M}} \sum_{s\in [S]}\mathbf{\lambda}_s    \sum_{i\in[N^s],j\in[M]}\sum_{k\in[N^s],l\in[M]}|\mathbf X_{ik}^s- \mathbf Y_{jl}|^2\pi_{ij}^{*s}\pi_{kl}^{*s}\\
        &= \min_{ \mathbf Y\in\R^{M\times M}}J_\lambda(\mathbf Y)\\ 
        &=J_\lambda(\mathbf Y_\lambda^*) \qquad (\text{due to Proposition \ref{remark: bary_fix_pi}}).
    \end{align*}
     As a consequence, ${ \mathbf Y}_{\mathbf{\lambda}}$ is another minimum of the function \eqref{eq: func_fix_pi}. Therefore, ${ \mathbf Y}_{\mathbf{\lambda}}=\mathbf Y^*_{\mathbf{\lambda}}$. 
\end{proof}

As a direct consequence, the algebraic form of GW barycenters in \eqref{eq: bary space for M and q fixed} is captured by:

\begin{equation*}
         \mathrm{Bary}_{M,\q}(\{(\mathbf X^s,\p^s)\}_{s\in [S]})\subseteq \left\{  \frac{1}{\q\q^T}\odot\sum_{s\in [S]}\lambda_s (\pi^s)^T \mathbf X^s \pi^s\mid \, \pi^s\in \Pi(\mathrm p^s, \mathrm q), \, \lambda\in \Delta_{S-1}\right\}.
\end{equation*}

The following corollary is a simple consequence of Theorem \ref{prop: formula_bary}. It guarantees that 
Algorithm \ref{alg: synthesis}, which is a simplified version of the iteration scheme
proposed in \cite{peyre2016gromov} to compute GW barycenters, has a fixed point (albeit does not guarantee convergence). 

\begin{algorithm}[ht!]
   \caption{Iterative GW Barycenter Synthesis}
   \label{alg: synthesis}
   \begin{algorithmic}[1]
   \STATE {\bfseries \underline{Inputs:}}\\
   - Templates $\{(\mathbf{X}^s,\p^s)\}_{s\in [S]}$\\ 
   - Weight vector $\lambda\in \Delta_{S-1}$\\
   - Predefined size of the output barycentric matrix $M$\\ 
   - $\q\in \mathcal{P}_M$ (usually uniform, i.e.,  $\q=\frac{1}{M}(1,\dots,1)\in \R^M$)\\
   - Initialization $\mathbf{Y}_0\in \mathbb R^{M\times M}$ 

   \smallskip

   \STATE {\bfseries \underline{Output:}} GW barycentric matrix 
   $\mathbf{Y} \in \R^{M\times M}$

   \smallskip

   \STATE Initialize $\mathbf{Y}\gets\mathbf Y_0$\\ 
   \REPEAT
   \FOR{$s=1,2,\ldots, S$}
   \STATE $\pi(\mathbf{Y},s) \gets$ optimal coupling for $GW((\mathbf{X}^s,\p^s),(\mathbf Y,\q))$ 
   \ENDFOR
   \STATE 
Update $\mathbf{Y}\gets \frac{1}{\q\q^T}\odot\sum_{s\in [S]}\lambda_s (\pi(\mathbf{Y},s))^T \mathbf X^s \pi(\mathbf{Y},s)$
   \UNTIL convergence
   \RETURN $\mathbf Y$
\end{algorithmic}
\end{algorithm}

\begin{corollary}\label{coro: fixed point}
Consider templates $\{(\mathbf X^s,\p^s)\}_{s=1}^S$, $M\in\N$, $\q\in \mathcal{P}_M$ with all entries strictly positive, and $\lambda\in \Delta_{S-1}$.
Given any real matrix $\mathbf Y$ of size $M\times M$, let $\pi(\mathbf Y,{s})\in \Pi(\p^s,\q)$ denote an optimal plan for $GW((\mathbf X^s,\p^s),( \mathbf Y,\q))$.  Consider the following operator in $\R^{M\times M}$:

\begin{equation}\label{eq: fixed point iteration}
      \rho_\lambda(\mathbf Y):= \frac{1}{\q\q^T}\odot\sum_{s\in [S]}\lambda_s \, (\pi( \mathbf Y,{s}))^T \, \mathbf X^s \, \pi( \mathbf Y,{s}) \qquad \forall \, \mathbf Y\in \R^{M\times M}.         
\end{equation}
Then, every matrix in $\mathrm{Bary}_{M,\q}(\{(\mathbf X^s,\p^s)\}_{s\in [S]})$ that is a solution for \eqref{eq: solomon_barycenter} with weights $\lambda$ is a fixed point of the map $\rho_\lambda$.
\end{corollary}

\begin{proof}
    Given $\lambda\in \Delta_{S-1}$,
    let $\mathbf Y_\lambda$ be a solution to the synthesis problem \eqref{eq: solomon_barycenter}. Thus, $\mathbf Y_\lambda\in \mathrm{Bary}_{M,\q}(\{(\mathbf X^s,\p^s)\}_{s\in [S]})$. Then, by Theorem \ref{prop: formula_bary} we have 
    that $\mathbf Y_\lambda$ is a fixed point of $\rho_\lambda$:

    \begin{align*}
        \rho_\lambda(\mathbf Y_\lambda)= \frac{1}{\q\q^T}\odot\sum_{s\in [S]}\lambda_s (\pi( \mathbf Y_\lambda,{s}))^T \mathbf X^s \pi( \mathbf Y_\lambda,{s})=\mathbf Y_\lambda^*=\mathbf Y_\lambda.
\end{align*}
\end{proof}

Algorithm \ref{alg: synthesis} can be viewed as applying  the function $\rho_\lambda$ iteratively, which is guaranteed to have a fixed point. In the statement of Corollary \ref{coro: fixed point}, we use the notation $\pi(\mathbf Y,s)$ to emphasize the dependence on $\mathbf{Y}$, as it is optimal for the GW problem $GW((\mathbf X^s,\p^s),(\mathbf Y, \q))$.

\begin{remark}[See also Example \ref{example: simple} in Suppl. Mat. \ref{app: gw fp synth}]\label{remark: does not depend on the plan}
    Theorem \ref{prop: formula_bary} implies that formula \eqref{eq: bary_fix_pi} for $\mathbf Y^*_{\mathbf{\lambda}}$ 
    does not depend on the initial choice of GW OT plans $\{\pi^{*s}\}_{s\in [S]}$.  Indeed, 
    let  $ \mathbf Y_{\mathbf{\lambda}}\in  \mathrm{Bary}_{M,\q}(\{(\mathbf X^s,\p^s)\}_{s\in [S]})$ for some $\mathbf{\lambda}\in \Delta_{S-1}$. For each $s\in [S]$,  let $\pi^{*s}$ and $\widetilde \pi^{*s}$ be couplings in $\Pi(\p^s,\q)$ both optimal for $GW((\mathbf X^s,\p^s),( \mathbf Y_{\mathbf{\lambda}},\q))$. Then, Theorem \ref{prop: formula_bary} implies that

    \begin{equation*}
        \frac{1}{\q\q^T}\odot\sum_{s\in [S]}\mathbf{\lambda}_s (\pi^{*s})^T \mathbf X^s \pi^{*s}=\mathbf Y_\lambda = \frac{1}{\q\q^T}\odot\sum_{s\in [S]}\mathbf{\lambda}_s (\widetilde\pi^{*s})^T \mathbf X^s \widetilde\pi^{*s}.
\end{equation*}
\end{remark}

\subsection{GW Barycenters via a Fixed-Point Approach: The Analysis Problem}\label{sec: Analysis}

Let us fix finitely many templates $( \mathbf X^s, \p^s)\in \R^{N^s\times N^s}\times\mathcal{P}_{N^s}$, for $s\in [S]$. Let $M\in \N$ and $\q\in \mathcal{P}_M$. In this section, we consider the following version of the \emph{GW analysis problem} introduced in \eqref{eq: analysis gw}:

\begin{equation}\label{eq: analysis_gw_bary_problem}
    \R^{M\times M}\ni \mathbf Y\longmapsto \lambda_{ \mathbf Y}:=\argmin_{\lambda\in \Delta_{S-1}} d_a( \mathbf Y, \mathbf Y_\lambda)^2,
\end{equation}
where $\mathbf{Y}_\lambda\in \mathrm{Bary}_{M,\q}(\{(\mathbf X^s,\p^s)\}_{s\in [S]})$, i.e., it is a solution of the GW synthesis problem \eqref{eq: solomon_barycenter} corresponding to weight vector $\lambda$. The function  $d_a(\cdot,\cdot)$ denotes a divergence measure on the space of $M\times M$ real matrices, which is to be specified.  

Theorem \ref{thm: analysis} below solves the GW analysis problem \eqref{eq: analysis_gw_bary_problem} under the GW-BCM hypothesis, that is, assuming that $\mathbf Y\in \R^{M\times M}$ belongs to the GW barycenter space given in Definition \ref{def: bary space} (in symbols, $\mathbf Y\in \mathrm{Bary}_{M,\q}(\{(\mathbf X^s,\p^s)\}_{s\in [S]})$). 
Under this hypothesis, Theorem \ref{prop: formula_bary} implies that 

    \begin{equation}\label{eq: Y and lambda_y}
      \mathbf Y=\sum_{s\in [S]}(\lambda_{\mathbf Y})_s \, F(\mathbf Y,s), \qquad \text{ for some } \lambda_{\mathbf Y}\in \Delta_{S-1}, 
\end{equation}
    where, for each $s\in [S]$, one considers $\pi^s$ is an optimal plan for $GW((\mathbf X^s,\p^s),(\mathbf Y,\q))$ to define

    \begin{equation}\label{eq: notation F}
    F( \mathbf Y, s):=\frac{1}{\q\q^T}\odot (\pi^s)^T \mathbf X^s \pi^s \in \R^{M\times M}.   
\end{equation}
    It is important to notice that each $\pi^s$ depends on the input $ \mathbf Y$, but, by Remark \ref{remark: does not depend on the plan}, we could consider another set of optimal plans $\{\widetilde{\pi^s}\}_{s=1}^S$ and the expression \eqref{eq: Y and lambda_y} would not differ. 
    This implies that  problem \eqref{eq: analysis_gw_bary_problem} for the given $ \mathbf Y\in  \mathrm{Bary}_{M,\q}(\{(\mathbf X^s, \p^s)\}_{s=1}^S)$ coincides with

    \begin{equation}\label{eq: our_analysis_bary_problem}
     \lambda_{ \mathbf Y}:=\argmin_{\lambda\in \Delta_{S-1}} d_a^2\left( \mathbf Y, \,  \sum_{s\in [S]}\lambda_s F( \mathbf Y,s)\right),
\end{equation}
where, by the non-negativity of $d_a(\cdot,\cdot)$ and by the GW-BCM hypothesis, we know in advance that for such $\mathbf Y$, $\min_{\lambda\in \Delta_{S-1}} d_a\left( \mathbf Y, \,  \sum_{s\in [S]}\lambda_s F( \mathbf Y,s)\right)=0.$
 In other words, 
 we aim to minimize, over the convex set $\Delta_{S-1}$, the functional  

 \begin{equation}\label{eq: lambda functional}
     \Theta:\R^S\to [0,\infty) \quad \text{ defined by } \qquad \Theta(\lambda):=d_a^2\left( \mathbf Y,\sum_{s\in [S]}\lambda_s F( \mathbf Y,s)\right).
\end{equation}

Notice that one can verify that $\mathbf{Y}\in \mathrm{Bary}_{M,\q}(\{(\mathbf X^s,\p^s)\}_{s\in [S]})$ if and only if $d_a(\mathbf{Y},\mathbf{Y}_{\mathbf{\lambda}})=0$ for some $\mathbf{Y}_\lambda\in \mathrm{Bary}_{M,\q}(\{(\mathbf X^s,\p^s)\}_{s\in [S]})$, \emph{for any} choice of divergence $d_a(\cdot,\cdot)$ (see Remark \ref{rem: for any div} for its formulation within the general BCM framework). In other words, to determine whether an element `strictly' belongs to the set in \eqref{eq: bary space for M and q fixed}, it suffices to use any nonnegative dissimilarity measure $d_a$ satisfying the identity of indiscernibles.

In this section, we have adopted the simplification that  $\mathrm{Bary}_{M,\q}(\{(\mathbf X^s,\p^s)\}_{s\in [S]})$ is a subset of $\mathbb R^{M\times M}$ (since dissimilarity matrices have fixed size $M$, and the weight vector $\q$ is also fixed). As a consequence, because the choice of the divergence $d_a(\cdot,\cdot)$ on $\mathbb R^{M\times M}$ does not affect the final result, for the following Theorem \ref{thm: analysis}, we adopt 
$d_a(\mathbf X, \mathbf Y)=\|\mathbf Y-\mathbf X\|_{\mathrm{Frob}}$. Considering this particular norm allows us to achieve our goal of providing closed-form expressions for the weights $\lambda$ (see \eqref{eq: klambda=b} below)) under the GW-BCM hypothesis (i.e., in the exact case, when the input is assumed to be a true barycenter of the fixed templates). See also Corollary \ref{coro: any} reflecting the independence of  solution $\lambda$ in Theorem \ref{thm: analysis} from the choice of $d_a$.

This follows the ideas of \cite{bonneel2016wasserstein} in the context of the Wasserstein structure rather than GW. In that article, the loss function is chosen so that its gradient is easy to compute: the authors test different divergences ($L^1$, $L^2$, and $\mathrm{KL}$) instead of considering the Wasserstein distance (between histograms) as the loss function for the corresponding analysis problem. These choices still lead to optimization problems requiring gradient-based methods for the barycentric weights, and no simplified closed-form expression for
$\lambda$ (as in \eqref{eq: klambda=b} below) is provided.

\begin{theorem}\label{thm: analysis} 
Let $\{( \mathbf X^s, \p^s)\in \R^{N^s\times N^s}\times\mathcal{P}_{N^s}\}_{s\in[S]}$, $M\in \N$, $(\mathbf Y, \q)\in  \R^{M\times M}\times \mathcal{P}_M$. Let $F(\mathbf{Y},s )$  be as in \eqref{eq: notation F}, for each $s\in [S]$, and consider the loss $d_a(\mathbf X, \mathbf Y)=\|\mathbf Y-\mathbf X\|_{\mathrm{Frob}}$ in the formulation of the GW analysis problem \eqref{eq:  analysis_gw_bary_problem}. If $\mathbf Y\in \mathrm{Bary}_{M,\q}(\{(\mathbf X^s,\p^s)\}_{s\in [S]})$, 
then a corresponding weight vector $\lambda_{\mathbf{Y}} \in \Delta_{S-1}$ associated with $\mathbf{Y}$ is given by a solution of the following convex quadratic minimization problem: 

\begin{equation}\label{eq: quad frob}
    \min_{\lambda\in \Delta_{S-1}}\lambda^T\mathcal{Q}\lambda, \quad \text{where } \,  \mathcal{Q}_{sr}:= \mathrm{tr}\left(\left( F(\mathbf Y,s)-\mathbf Y\right)^T\left(F( \mathbf Y,r)-\mathbf Y\right)\right) \, \, \,  \forall s,r\in [S].
\end{equation}
Equivalently, $\lambda_{\mathbf Y}$ is a solution of the following linear problem:

\begin{equation}\label{eq: klambda=b}
    {K\lambda = b} \qquad \text{ subject to } \lambda\in \Delta_{S-1}\subset \R^S,
\end{equation}
where $K\in \R^{S\times S}$ and $b\in \R^S$ are given by

\begin{equation}\label{eq: K and b}
    K_{sr}:= \mathrm{tr}\left(F( \mathbf Y,s)^T F( \mathbf Y,r)\right) \quad \forall s,r\in[S], \quad 
    b_s:=\mathrm{tr}\left( \mathbf Y^T \, F( \mathbf Y,s) \right) \quad \forall s\in [S].
\end{equation}
\end{theorem}

\begin{proof} 

Let $ \mathbf Y\in  \mathrm{Bary}_{M,\q}(\{(\mathbf X^s, \p^s)\}_{s=1}^S)$.
Under the choice of the Frobenius norm 
$\|\cdot\|_{\mathrm{Frob}}$ in $\R^{M\times M}$, the nonnegative functional $\Theta$ in  \eqref{eq: lambda functional} can be written as

\begin{align*}
   \Theta(\lambda)&= \mathrm{tr}\left(\left( \sum_{s\in [S]}\lambda_s F( \mathbf Y,s)-\mathbf Y\right)^T\left( \sum_{r\in [S]}\lambda_r F( \mathbf Y,r)-\mathbf Y\right)\right) \notag\\ 
    & =\sum_{s,r\in [S]}\lambda_s\lambda_r \underbrace{\mathrm{tr}\left(\left( F(\mathbf Y,s)-\mathbf Y\right)^T\left(F( \mathbf Y,r)-\mathbf Y\right)\right)}_{{\mathcal{Q}}_{sr}} \notag \\
    &=\lambda^T\mathcal{Q}\lambda .\notag\\
\end{align*}
Due to \eqref{eq: Y and lambda_y}, the GW barycentric coordinate vector $\lambda_{\mathbf Y}$ is such that

\begin{equation}\label{eq: zero}
    \Theta(\lambda_{\mathbf Y})=\lambda_{\mathbf Y}^T\mathcal Q \lambda_{\mathbf Y}=0.
\end{equation}
 Equivalently, $\lambda_{\mathbf Y}$ is a solution  of the  quadratic minimization problem \eqref{eq: quad frob}.
The matrix $\mathcal{Q}$ is symmetric and  positive semidefinite. This holds since it is the Gram matrix of the system of matrices $\{F(\mathbf{Y},s)-\mathbf Y\}_{s=1}^S$ with respect to the inner product in $\R^{M\times M}$ given by $\langle A,B\rangle=\mathrm{tr}(A^TB)$, and it is well-known that all Gram matrices are symmetric positive semidefinite \cite{schwerdtfeger1961introduction}. Thus, $\Theta$ is convex and non-negative as a function over the whole $\R^S$.
So, critical points of $\Theta$ are global minima.
Moreover, we can decompose

\begin{equation*}
    \lambda^T\mathcal{Q}\lambda=\mathrm{tr}\left( \mathbf Y^T\mathbf Y\right)-2\sum_{s\in[S]}\lambda_sb_s+\sum_{s\in [S]}\sum_{r\in [S]}\lambda_s\lambda_rK_{sr},
\end{equation*}    
where $K\in \R^{S\times S}$ and $b\in \R^{S}$ are given by \eqref{eq: K and b}. 
Thus, minimizing $\Theta(\lambda)$ over the simplex $\Delta_{S-1}$  is equivalent to solving the quadratic program

\begin{equation}\label{eq: quadratic program 1}
    \min_{\lambda\in \Delta_{S-1}} \lambda^TK\lambda-2\lambda^Tb,
\end{equation}
and the critical points of $\Theta$ are the solutions of the linear problem $K\lambda=b$. 
As a conclusion, the vector solutions $\lambda\in \R^S$ of the linear system $K\lambda=b$ are global minima of $\Theta$. These solutions might not be unique but under the GW-BCM condition (i.e., assuming $\mathbf Y\in \mathrm{Bary}_{M,\q}(\{(\mathbf X^s, \p^s)\}_{s=1}^S)$), it is guaranteed that there exists at least one solution $\lambda_{\mathbf Y}$ in $\Delta_{S-1}$. In other words, the solutions of \eqref{eq: quad frob} (or \eqref{eq: quadratic program 1}) are given by \eqref{eq: klambda=b}.
\end{proof}

Algorithm \ref{alg: analysis} implements the analysis scheme laid out in Theorem \ref{thm: analysis}, and we refer the reader to Suppl. Mat. \ref{app: comp comp} for an analysis of its computational complexity.

\begin{corollary}\label{coro: any}
    Let $\{( \mathbf X^s, \p^s)\in \R^{N^s\times N^s}\times\mathcal{P}_{N^s}\}_{s\in[S]}$, $M\in \N$, $(\mathbf Y, \q)\in  \R^{M\times M}\times \mathcal{P}_M$. In the formulation of the GW analysis problem \eqref{eq:  analysis_gw_bary_problem} consider \textbf{any} nonnegative loss function $d_a(\cdot,\cdot)$ in $\R^{M\times M}\times \R^{M\times M}$ satisfying the identity of indiscernibles. If $\mathbf Y\in \mathrm{Bary}_{M,\q}(\{(\mathbf X^s,\p^s)\}_{s\in [S]})$, 
then a  weight vector $\lambda_{\mathbf{Y}} \in \Delta_{S-1}$ associated with $\mathbf{Y}$ is given by a solution of the minimization problem \eqref{eq: quad frob} provided in Theorem \ref{thm: analysis}.
\end{corollary}

\begin{proof}
    If $\mathbf Y\in \mathrm{Bary}_{M,\q}(\{(\mathbf X^s,\p^s)\}_{s\in [S]})$, then
    it is given by the expression \eqref{eq: Y and lambda_y} for some weight vector $\in \Delta_{S-1}$.
    Therefore, for \textbf{any} choice of loss $d_a(\cdot,\cdot)$ that is nonnegative and satisfies the identity of indiscernibles, we have
    $d_a(\mathbf{Y}, \sum_{s\in [S]}\lambda_s \, F(\mathbf Y, s))=0$ for some $\lambda\in \Delta_{S-1}$ (which is not necessarily unique). If we consider $\lambda_{\mathbf{Y}}$ as  a solution of the minimization problem \eqref{eq: quad frob} in Theorem \ref{thm: analysis}, then, from the proof of such theorem, this is equivalent to have \eqref{eq: zero}, i.e,  

    \begin{equation}\label{eq: fro}
        \left\|\mathbf Y-\sum_{s\in [S]}(\lambda_{\mathbf Y})_s \, F(\mathbf Y, s)\right\|_{\mathrm{Frob}}=\lambda_{\mathbf Y}^T\mathcal Q \lambda_{\mathbf Y}=0.
\end{equation}
    Since $\|\cdot\|_{\mathrm{Frob}}$ is a norm in the space of $M\times M$ square matrices, \eqref{eq: fro} implies $\mathbf Y=\sum_{s\in [S]}(\lambda_{\mathbf Y})_s \, F(\mathbf Y, s)$. So, $d_a(\mathbf Y,\sum_{s\in [S]}(\lambda_{\mathbf Y})_s \, F(\mathbf Y, s))=0$, and therefore the non-negativity of $d_a$ leads to
    $\lambda_{\mathbf Y}\in\argmin_{\lambda\in \Delta_{S-1}} d_a( \mathbf Y, \mathbf Y_\lambda)^2$ (i.e., the solution $\lambda_{\mathbf Y}$ provided in Theorem \ref{thm: analysis} is also a solution for the GW analysis problem \eqref{eq:  analysis_gw_bary_problem} with arbitrary divergence $d_a$).
\end{proof}

\begin{remark}\label{remark: permutation}
Node relabeling corresponds, at the matrix level, to \emph{permutation similarity}: if $P$ is a permutation matrix and $\mathbf Y$ is a dissimilarity matrix, then $P\mathbf Y P^T$ represents the same object with relabeled nodes.
If in \eqref{eq: analysis_gw_bary_problem} we take $d_a$ to be the Frobenius norm, then the analysis objective is invariant under such relabeling. Indeed, the Frobenius norm is invariant under permutation similarity (i.e., 
$
\|\mathbf Y\|_{\mathrm{Frob}}=\|P\mathbf Y P^T\|_{\mathrm{Frob}}$
for  all permutation matrices $P$).
Moreover, the (fixed-size) GW barycenter set is closed under relabeling: if $\mathbf Y_\lambda$ belongs to the GW barycenter family generated by certain templates, then so does $P\mathbf Y_\lambda P^T$ (which corresponds to the same parameter $\lambda$ as $GW(\mathbf Y_\lambda,\,P\mathbf Y_\lambda P^T)=0$).
Therefore,

\begin{equation*}
\min_{\lambda}\|\mathbf Y-\mathbf Y_\lambda\|_{\mathrm{Frob}}
=\min_{\lambda}\|P(\mathbf Y-\mathbf Y_\lambda)P^T\|_{\mathrm{Frob}}
=\min_{\lambda}\|P\mathbf YP^T-\mathbf Y_\lambda\|_{\mathrm{Frob}}.
\end{equation*}
    Moreover, our surrogate analysis problem \eqref{eq: our_analysis_bary_problem} is also permutation/relabeling invariant for any target $(\mathbf Y,\q)$ (inside or outside the barycenter space) when considering $d_a$ as the Frobenius norm. Indeed, let $P$ be an arbitrary permutation that produces $(P\mathbf YP^T,P\q)$ (which encodes the same information as the original shape after relabeling the nodes). If $\pi^s\in \Pi(\p^s,\q)$ is an optimal plan for $GW((\mathbf X^s,\p^s),(\mathbf{Y},\q))$, it holds that $\pi^sP^T\in \Pi(\p^s,P\q)$ is optimal for $GW((\mathbf X^s,\p^s),(P\mathbf{Y}P^T,P\q))$. In particular, by using  $(PAP^T)\odot (PBP^T)=P(A\odot B)P^T$ (equivariance of the Hadamard product under simultaneous permutation via conjugation), we  have $F(P\mathbf YP^T, s)=\frac{1}{P\q\q^TP^T}\odot P(\pi^s)^T \mathbf X^s \pi^sP^T=P\left(\frac{1}{\q\q^T}\odot (\pi^s)^T \mathbf X^s \pi^s\right)P^T=PF(\mathbf Y, s)P^T$. Therefore, $\|P\mathbf YP^T-\sum_{s}\lambda_s F(P\mathbf YP^T,s)\|_{\mathrm{Frob}}=\|\mathbf Y-\sum_{s}\lambda_s F(\mathbf Y,s)\|_{\mathrm{Frob}}$. We refer to Section \ref{sec: permutation} in the Suppl. Mat. for empirical demonstration. 
\end{remark}

\begin{algorithm}[H]
   \caption{Analysis: GW Barycenter Coordinates}
   \label{alg: analysis}
   \begin{algorithmic}[1] 
   \STATE {\bfseries \underline{Inputs:}}\\
   - Templates $\{(\mathbf{X}^s,\p^s)_{s\in [S]}\}$\\ 
   - Shape/Network $\mathbf Y\in \R^{M\times M}$, $\q\in \mathcal{P}_M$ \\

   \smallskip

   \STATE {\bfseries \underline{Output:}} Weight vector $\lambda$\\

   \smallskip

   \FOR{$s=1,2,\ldots, S$}
   \STATE $\pi(\mathbf Y, s )\gets$  optimal coupling for $GW((\mathbf{X}^s,\p^s),(\mathbf Y,\q))$ \\
   \STATE  $F(\mathbf Y,s)\gets \pi(\mathbf Y, s)^T\mathbf{X}^s\pi(\mathbf Y, s)$\\
   \ENDFOR
   \FOR{$s=1,2,\ldots, S$}
    \FOR{$r=s, s+1,\ldots, S$}
        \STATE $\mathcal{Q}_{sr}=\mathrm{tr}\left((F(\mathbf Y,s)-\mathbf Y)^T(F(\mathbf Y,s)-\mathbf Y)\right)$ 
        \STATE $\mathcal{Q}_{rs}=\mathcal{Q}_{sr}$
    \ENDFOR  
\ENDFOR
   \STATE $\lambda\gets$  $\arg\min_{\lambda\in\Delta_{S-1}} \lambda^T \mathcal{Q}\lambda$ 
   \RETURN $\lambda$ 
   \end{algorithmic}
\end{algorithm}

\section{A Gradient Method for Gromov-Wasserstein Barycenters}\label{sec: analysis grad}

In \cite{chowdhury2020gromov}, given a network $\mathbb X$, the authors provide a way of computing directional derivatives of the functional $\mathbb Y \longmapsto GW^2( \mathbb X, \mathbb Y)$ in the direction of  vectors in a `tangent space' at $\mathbb Y$ 
\cite{beier2024tangential, chowdhury2019gromov, sturm2023space}. 
Thus, in 
Subsections \ref{sec: tan}, \ref{sec: bu}, and \ref{sec: synthesis bu} we shift to a different approach that allows to compute gradients of functionals in GW space using a blow-up technique, and thus defining a weaker version of GW barycenters as Karcher means. This culminates in our new contributions to the GW analysis problem (Section \ref{sec: analysis bu}), including Theorem \ref{prop: gradient method} and Algorithm \ref{alg: analysis blow up}. Connections between the fixed-point and gradient-based approaches are discussed in Section \ref{sec: relation}.

\subsection{Tangent Space, Exponential Map,  Directional Derivatives, and Gradients}\label{sec: tan}

As introduced in Section \ref{sec: intro GW}, let $\mathcal{GW}$ be the Gromov-Wasserstein space of equivalence classes of networks under weak isomorphism $\sim^w$. Thus, two networks $\mathbb Y$ and $\mathbb Y'$ belong to the same class $[\mathbb Y]$ if and only if $GW(\mathbb Y,\mathbb Y')=0$. In this case,  a GW OT plan $\pi$ realizing $GW(\mathbb Y,\mathbb Y')=0$ is referred to as a \emph{self-coupling}. 
In \cite{chowdhury2020gromov}, the authors provide a \emph{formal tangent structure} on the GW space $\mathcal{GW}$.In particular, they define the \emph{tangent space} of $\mathcal{GW}$ at a network class $[\mathbb Y]$  as the  \emph{quotient space}\footnote{If we consider only the space of classes of mm-spaces or gm-spaces under the strong isomorphism defined in Definition \ref{def: strong iso}, then each $L^2$-space in \eqref{eq: tan2} is restricted to its subspace of symmetric functions.}

\begin{equation}\label{eq: tan2}
  \mathrm{Tan}_{[\mathbb{Y}]}:=\displaystyle\left(\bigcup_{(Y,\omega_Y,\mu_Y)\in[\mathbb{Y}]} L^2(Y\times Y,\mu_Y\otimes \mu_Y)\right)/\simeq
\end{equation}
\emph{under the equivalence relation} $\simeq$ defined as follows: for $(Y,\omega_Y,\mu_Y),(Y',\omega_{Y'},\mu_{Y'})\in[\mathbb{Y}]$,
let $v\in L^2(Y\times Y,\mu_Y\otimes \mu_Y)$, $v'\in L^2(Y'\times Y',\mu_{Y'}\otimes \mu_{Y'})$, 
     we say that $v\simeq v'$ if and only if there is a self-coupling $\pi\in \Pi(\mu_Y,\mu_{Y'})$  such that 

\begin{equation*}
    \omega_{Y}(y_1,y_2)=\omega_{Y'}(y_1',y_2') \quad \text{ and } \quad v(y_1,y_2)=v'(y_1',y_2') \qquad   a.e. \,  \pi(y_1,y_1') \pi(y_2,y_2').
\end{equation*}
We denote the class of $v$ as $[v]$. Each  space $\mathrm{Tan}_{[\mathbb{Y}]}$ is equipped with the inner product

\begin{equation*}\label{eq: inner product}
    \langle [v],[v']\rangle_{\mathrm{Tan}_{[\mathbb{Y}]}}:=\langle(\varphi_Y)^*v,(\varphi_{Y'})^*v'\rangle_{L^2(Z\times Z,\mu_Z\otimes \mu_Z)},
\end{equation*}
where the measure space
$(Z,\mu_Z)$ and the functions $\varphi_{Y}:Z\to Y$, $\varphi_{Y'}:Z\to Y'$ are as in Definition \ref{def: weak iso} of weak isomorphism between $\mathbb Y$ and $\mathbb Y'$ (i.e., a \emph{common resolution}).
Notice that, when considering finite networks $(Y=\{y_j\}_{j=1}^M,\omega_Y,\mu_Y=\sum_{j=}^M\q_j\delta_{y_j})$, the $L^2$-space $L^2(Y\times Y,\mu_Y\otimes\mu_Y)$ can be identified with  $\R^{M\times M}$ equipped with an inner product induced by a weighted version of the Frobenius norm:

\begin{equation}\label{eq: inner product = trace}
  \langle A, B\rangle_{L^2(Y\times Y,\mu_Y\otimes \mu_Y)}=\sum_{i,j\in [M]}A_{ij}B_{ij}\q_i\q_j=:\mathrm{ tr}_{\q}(A^TB), 
\end{equation}
where $\mathrm{ tr}_{\q}$ denotes the trace inner product weighted by $\q\otimes \q\in \R^{M\times M}$.{To ensure that this defines a norm, we assume that $\q_j > 0$ for all $j \in [M]$, meaning that the support of $\mu_Y$ consists precisely of the $M$ distinct points $\{y_j\}_{j=1}^M$, as required in Definition \ref{def: mm gm net}.

The \emph{exponential} map $Exp_{[\mathbb Y]}: \mathrm{Tan}_{[\mathbb{Y}]} \longrightarrow \mathcal{GM}$ is defined as:

\begin{equation}\label{eq: exp}
    [v]\longmapsto [(Y,\omega_Y+v,\mu_Y)]  \qquad \text{ for } v\in L^2(Y\times Y,\mu_Y\otimes \mu_Y).
\end{equation}
Using our shorthand notation $(\mathbf Y, \q)\in \R^{M\times M}\times \mathcal{P}_M$ to denote a representative of $[\mathbb Y]$, \eqref{eq: exp} can be identified with the map $v\mapsto (\mathbf Y+v,\q)$ for all $v\in \R^{M\times M}$. 

\begin{definition}[\cite{chowdhury2019gromov}]  \label{def: grad}
    Given a functional over the GW space, $G:\mathcal{GW}\to \R$, its  \emph{directional derivative} in the direction of $[v]\in \mathrm{Tan}_{[\mathbb Y]}$ is defined by

\begin{equation}\label{eq: dir deriv}
    D_vG([\mathbb Y]):=\lim_{t\searrow 0}\frac{G(Exp_{[\mathbb Y]}([tv]))-G([\mathbb Y])}{t},
\end{equation}
and its \emph{gradient}\footnote{If $G:\mathcal{GW}\to \R$ is differentiable, i.e, if all the directional derivatives \eqref{eq: dir deriv} exist, and if, , its gradient at $[\mathbb Y]$, as defined via \emph{Riesz representation} \eqref{eq: grad def}, exists, then it is unique (see \cite[Lem. 7]{chowdhury2020gromov} and the references therein).
} by a vector $\nabla G([\mathbb Y])\in \mathrm{Tan}_{[\mathbb Y]}$ satisfying 
\begin{equation}\label{eq: grad def}
     D_vG([\mathbb Y])=\langle [v],\nabla G([\mathbb Y])\rangle_{\mathrm{Tan}_{[\mathbb Y]}}, \qquad \forall [v]\in \mathrm{Tan}_{[\mathbb Y]}.
\end{equation}
\end{definition}

\subsection{Blow-Ups}\label{sec: bu}
In \cite{chowdhury2020gromov}, particular representatives of the equivalence classes under weak isomorphisms were considered using a \emph{blow-up} technique, which we will describe. 
This blow-up idea avoids the need to fix a size $M$ in advance, as required in the approach of \cite{peyre2016gromov} and reviewed in this paper when studying the GW synthesis problem  \eqref{eq: bary space for M and q fixed} in Section \ref{sec: Synthesis}.

\begin{SCfigure}[5][ht!]
    \centering
    \vspace{0.3in}
    \includegraphics[width=0.25\linewidth]{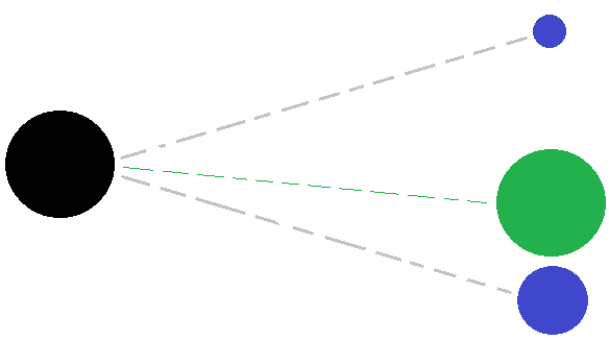}
   \caption{\small{Sketch to visualize the classical \emph{barycentric projection}. Gray dashed lines indicate the optimal transport plan $\pi$ between the source measure $\mu$ (black) and the target measure $\nu$ (blue). In general, if $\mu=\sum_i^n a_i \delta_{x_i}$ and $\nu=\sum_{j}^mb_j\delta_{y_j}$, then the barycentric projection measure (green) is supported at $\widehat{y_i}:=\frac{1}{a_i}\sum_{j}\pi_{ij}y_j$ and the mass at $\widehat{y_i}$
 coincides with that of
    $x_i$ for all $i\in[n]$, thus, the optimal transport between the source measure and its barycentric projection is induced by a map (green dashed lines).}}
    \label{fig:barycentric proj}    
\end{SCfigure}

In a broader sense, this idea is closely related to the notion of \emph{barycentric projection} widely used in linear OT \cite{wang2013linear} to recover transport maps from OT plans. 
While the barycentric projection effectively ``collapses'' mass (see Fig. \ref{fig:barycentric proj}), we will see that the blow-up technique instead ``multiplies'' the number of nodes (cf. Fig. \ref{fig: bu} below). The parallelism is that, in both cases, one is led to consider optimal couplings that are induced by maps after obtaining distributions with the same number of particles and mass.
For the GW case, we refer to the relevant work \cite{beier2022linear}. There, a single reference mm-space is fixed, and each object is embedded into a common linear space via a GW optimal coupling followed by a barycentric projection. This construction yields a local linear model of GW space and provides a computationally efficient approximation of the inherently nonlinear GW geometry. The focus of \cite{beier2022linear} is on defining a linear embedding of GW space and computing a single GW optimal coupling to a fixed reference, followed by a barycentric projection. Instead of embedding, we rely on the notion of weakly isomorphic networks: we consider multiple references and utilize the blow-up technique (which requires multiple optimal couplings) to obtain new weakly isomorphic networks with which we can compute ``sums'' and, moreover, ``weighted averages''.

\begin{definition}[\cite{chowdhury2019gromov, chowdhury2020gromov}]  \label{def: blowup} Consider 
$(X=\{x_i\}_{i=1}^N,\omega_X,\mu_X)$, $(Y=\{y_j\}_{j=1}^M,\omega_Y,\mu_Y)$ finite networks,
which we simply denote as
$(\mathbf X,\p)\in \R^{N\times N}\times \mathcal{P}_N$, $(\mathbf Y,\q)\in \R^{M\times M}\times \mathcal{P}_M$ (we recall the notation $\omega_X(x_i,x_{i'})=\mathbf X[i,i']$,  $\omega_Y(y_j,y_{j'})=\mathbf Y[j,j']$,  $\mu_X=\sum_{i=1}^N\p[i]\delta_{x_i}$, $\mu_Y=\sum_{i=1}^M\q[j]\delta_{y_j}$). 
Let $\pi\in \Pi(\p,\q)$ be a coupling between  $\p$ and $\q$. Consider the vectors  
$u({X})=(u(x_1),\dots u(x_N))$ and $v({Y})=(v(y_1),\dots, v(y_M))$ given by
\begin{equation*}
  u(x_i):=\left|\{j\in[M]\mid \, \pi(x_i,y_j)>0\}\right|\quad \text{and} \quad v(y_j):=\left|\{i\in[N]\mid  \, \pi(x_i,y_j)>0\}\right|.
\end{equation*}
We define:
\begin{itemize}[leftmargin=*]
    \item New sets of nodes

    \begin{equation*}
        X_b:=\bigcup_{i\in[N], \,  1\leq k_i\leq u(x_i)}(x_i,k_i)\qquad \text{and} \qquad Y_b:=\bigcup_{j\in[M], \, 1\leq l_j\leq v(y_j)}(y_j,l_j).
\end{equation*}
\item New probability vectors supported on such nodes

\begin{equation*}
  \p_b [i,k_i]:=\pi(\{x_i,y_{k_i}\}) \qquad\text{and} \qquad  \q_b [j,l_j]:=\pi(\{x_{l_j},y_j\}),  
\end{equation*}
along with new probability measures $\mu_{X_b}:=\sum \p_b[i,k_i]\delta_{(x_i,k_i)}$, $\mu_{Y_b}:=\sum \q_b[j,l_j]\delta_{(y_j,l_j)}$.
\item New weight functions $\omega_{X_b}$, $\omega_{Y_b}$, represented by the following matrices 

\begin{equation*}
  \mathbf X_b[(i,k_i),(i',k_{i'})]:=\mathbf X[i,i'] \qquad \text{and} \qquad \mathbf Y_b[(j,l_j),(j',l_{j'})]:=\mathbf Y[j,j'].
\end{equation*}
\end{itemize}

The new spaces $\mathbb X_b=(X_b,\omega_{X_b},\mu_{X_b})$ and 
$\mathbb Y_b=(Y_b,\omega_{Y_b},\mu_{Y_b})$
are called a \emph{blow-up} of the original networks ${\mathbb X}=({X},{\omega_X},{\mu_X})$ and 
${\mathbb Y}=({Y},{\omega_Y},{\mu_Y})$.
\end{definition}

Let $\mathbb X$ and $\mathbb Y$ be finite networks with $N$ and $M$ nodes, respectively, and 
$\mathbb X_b=(X_b,\omega_{X_b},\mu_{X_b})$ and 
$\mathbb Y_b=(Y_b,\omega_{Y_b},\mu_{Y_b})$, simply denoted by $(\mathbf X_b, \p_b)$ and $(\mathbf Y_b,\q_b)$, be as in Definition \ref{def: blowup} with $\pi\in \Pi(\p,\q)$ an $N\times M$ \emph{optimal} coupling for $GW(\mathbb X,\mathbb Y)$. Then, we have the following properties:

\begin{enumerate}[leftmargin=*]
    \item $\mathbb X_b\sim^w\mathbb X$ and $\mathbb Y_b\sim^w\mathbb Y$. Thus, $GW(\mathbb X, \mathbb Y)\notag = GW(\mathbb X_b, \mathbb Y_b)$.
    \smallskip
    
    \item Same size: $|X_b|=|\mathrm{supp}{(\pi)}|=|Y_b|$. We denote this number by $M_b$.
    
    \smallskip
    
    \item Alignment:  One can construct a (possibly enlarged) optimal coupling $\pi_b\in \Pi(\p_b,\q_b)$ for $GW(\mathbb X_b, \mathbb Y_b)$ that induces an optimal GW-transport map $T:X_b\to Y_b$, 
    which is a bijection, i.e.,  a permutation of the {indices} in $[M_b]$ (see \cite{dumont2024existence} for existence of optimal GW maps in the context of mm-spaces under strong isomorphisms). 
    As highlighted in \cite{chowdhury2020gromov}, the crux of the blow-up construction relies on the observation that while $\mathbb X_b$ and $\mathbb Y_b$ maintain weak isomorphisms with $\mathbb X$ and $\mathbb Y$, respectively, the initial (optimal) transport plan $\pi$ between the original measures $\mu_X,\mu_Y$  supported on the original nodes $X,Y$ (represented the vectors $\p,\q$),
lifts to a
 map ${T}$ between the new sets of nodes $X_b, Y_b$ such that $T_\#\mu_{X_b}=\mu_{Y_b}$:

\begin{equation*}
\xymatrix{
\mathbb X \ar[r]^{\qquad \text{ plan } \qquad } \ar[d]_{\sim^w} 
& \mathbb Y \ar[d]^{\sim^w} \\
\mathbb X_b \ar[r]^{\qquad \text{ map } \qquad} 
& \mathbb Y_b 
}    
\end{equation*}    
    For simplicity, after permuting and relabeling the nodes, if necessary, we will assume that $T$ is the identity map (i.e., we  identify $X_b=Y_b$). Consequently,  $\p_b=\q_b$, and $\pi_b$ can be regarded as the diagonal matrix with $\q_b$ as its diagonal entries, denoted as $\pi_b=\mathrm{diag}(\q_b)$. See Fig. \ref{fig: bu} for an illustration.
    Moreover, 

\begin{equation}\label{eq: GW as Frob}
    GW^2(\mathbb X, \mathbb Y) = \sum_{j,k\in [M_b]}|\mathbf X_b[j,k]- \mathbf Y_b[j,k]|^2 \, \q_b[j]\q_b[k]=\mathrm{tr}_{\q_b}\left((\mathbf X_b -\mathbf Y_b)^T(\mathbf X_b -\mathbf Y_b)\right).
\end{equation}
Expression \eqref{eq: GW as Frob} shows that, after blow-up, the GW distance 
becomes a Frobenius norm with weight depending on an optimal GW assignment between the pair of networks. 

\end{enumerate}

\begin{figure}[ht!]
    \centering

    \includegraphics[width=0.9\linewidth]{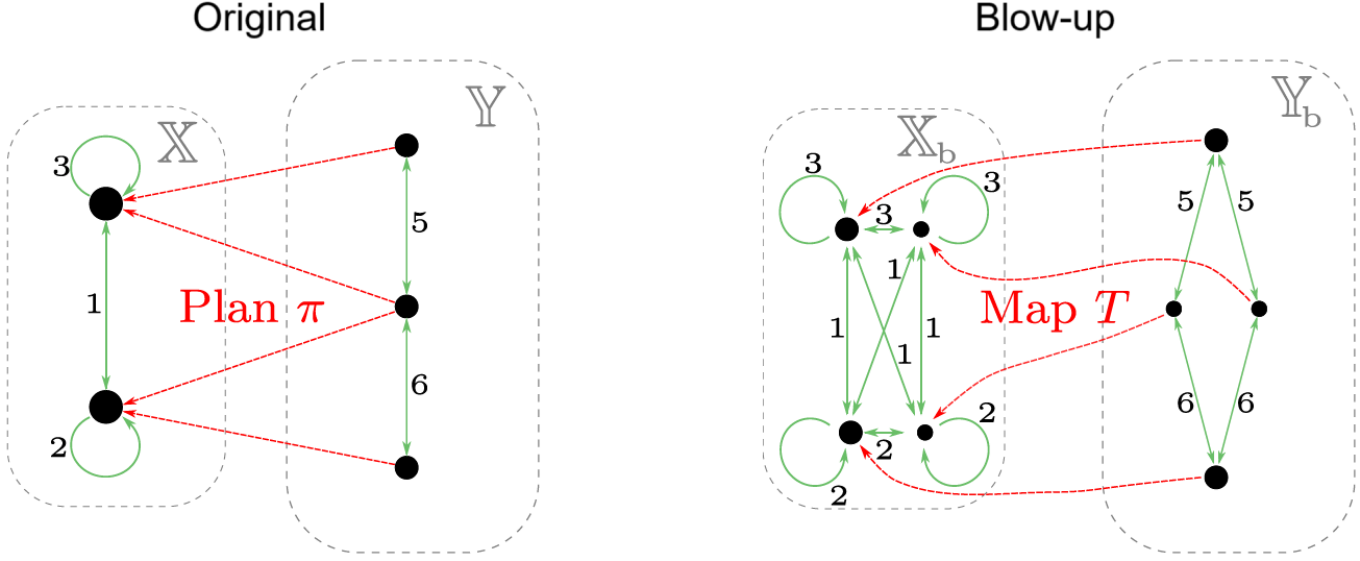}

    \caption{\small{Blow-up illustration.}}
    \label{fig: bu}

\end{figure}

\noindent If we consider multiple templates $\{[\mathbb X^s]\}_{s=1}^S$ for $S\geq 1$, given a class $[\mathbb Y]$, 
Algorithm \ref{alg: blow up} is a \emph{general blow-up}, which facilitates the alignment of each template representative $\mathbb X^s = (X^s, \omega_{X^s}, \mu_{X^s})$ with a representative $\mathbb Y = (Y, \omega_Y, \mu_Y)$, ensuring that the new spaces $Y_b={X^1_b}=\dots= {X^S_b}$  all contain the same number of nodes ($M_b$), share the same mass distribution $\q_b$ across the nodes, and the identity map is optimal for each GW problem between $\mathbb X^s_b$ and $\mathbb Y_b$ ($s\in [S]$). Note that the optimal plans between $\mathbb X^s_b, \mathbb X^r_b $, for $s\neq r$, are not necessarily the identity. The network $\mathbb Y$ plays a special role in this algorithm, and thus when we apply it to templates $\{[\mathbb X^s]\}_{s=1}^S$, we say that we perform \emph{an optimal blow-up with respect to $\mathbb Y$}.

\begin{algorithm}[H]
   \caption{An Optimal Blow-Up of Multiple Templates With Respect to a Fixed Network}
   \label{alg: blow up}
   \begin{algorithmic}[1]
   \STATE {\bfseries \underline{Inputs:}}\\
   - Templates $\{(\mathbf{X}^s,\p^s)\in \R^{N^s\times N^s}\times \mathcal{P}_{N^s}\}_{s\in [S]}$\\ 
   - Network $(\mathbf Y,\q)\in \mathbb R^{M\times M}\times\mathcal{P}_M$
   \smallskip

   \STATE {\bfseries \underline{Output:}} \\
    - Blow-up size $M_b\in \N$\\
    - Blow-up template matrices  $\{\mathbf{X}^s_b\}_{s\in [S]}$ (all of size $M_b\times M_b$)\\
    - Blow-up network $(\mathbf Y_b,\q_b)\in \mathbb R^{M_b\times M_b}\times\mathcal{P}_{M_b} $

   \smallskip

    \FOR{$s=1,2,\ldots, S$}
        \STATE $\mathbf{X}\gets\mathbf X^s$
        \STATE $\p\gets \p^s$
        \STATE $\pi$ $\gets$ optimal plan for $GW((\mathbf{X},\p),(\mathbf Y,\q))$ \\
        \STATE $v_X,v_Y\gets $ {indices} of non-null entries of $\pi$ \\
        \STATE $M_b = \mathrm{len}(v_X)$
        \STATE Set $\q_b$ such that $\q_b[l]= \pi[v_X[l],v_Y[l]]$ for all $l\in [M_b]$\\
        \STATE $\q\gets \q_b$\\
        \STATE Set $\mathbf X_b$ such that $\mathbf X_b[i,j] = \mathbf X[v_X[i],v_X[j]]$ for all $i,j\in [M_b]$
        \STATE Set $\mathbf Y_b$ such that $\mathbf Y_b[i,j] = \mathbf Y[v_Y[i],v_Y[j]]$ for all $i,j\in [M_b]$
        \STATE $\mathbf X^s_b \gets \mathbf X_b$
        \STATE $\mathbf Y \gets \mathbf Y_b$
        \FOR{$r=1,2,\ldots s-1$}
            \STATE Set $\mathbf X_b$ such that $\mathbf X_b[i,j] = \mathbf X^r_b[v_Y[i],v_Y[j]]$ for all $i,j\in [M_b]$
            \STATE $\mathbf X^r_b \gets \mathbf X_b$
        \ENDFOR
    \ENDFOR
    \RETURN $M_b$, $\mathbf Y_b,\q_b$,  $\mathbf X^1_b,\dots, \mathbf X^S_b$

\end{algorithmic}
\end{algorithm}

\begin{proposition}[Prop. 8 \cite{chowdhury2020gromov} - Revisited proof in Suppl. Mat. \ref{app: weak}] \label{prop: grad via blow up}
For a fixed family of classes of finite templates $\{[\mathbb X^s]\}_{s=1}^S$ and  $\lambda\in \Delta_{S-1}$, consider the objective functional in \eqref{eq: syn gw} given by

\begin{equation}\label{eq: barycenter functional}
 G_\lambda([\mathbb Y]):=\frac{1}{2}\sum_{s\in[S]}\lambda_sGW^2(\mathbb X^s,\mathbb Y). 
\end{equation}
Then $\nabla G_\lambda([\mathbb Y])$   can be identified, after an optimal blow-up of the templates with respect to $\mathbb Y$, with the square matrix
    $\sum_{s=1}^S\lambda_s{\mathbf X^s_b}-\mathbf Y_b$.
\end{proposition}

\subsection{GW Barycenters via Blow-Up: The Synthesis Problem}\label{sec: synthesis bu} 

In the same spirit as Definition \ref{def: bary space}, we introduce the following 
general and weak forms of GW barycenters. 

\begin{definition}\label{def: bary bu}
Consider a collection of reference classes of finite networks  $\{[\mathbb{X}^s]\}_{s\in [S]}$. The \emph{GW barycenter space} generated by $\{[\mathbb{X}^s]\}_{s\in [S]}$ is defined by:

\begin{equation}\label{eq: bary space very general}
        {\mathrm{Bary}}(\{[\mathbb{X}^s]\}_{s=1}^S):=\left\{[\mathbb{Y}_\lambda] \mid \mathbb{Y}_\lambda \text{ is as in } \eqref{eq: syn gw} \text{ for some } \lambda\in \Delta_{S-1}\right\}.
\end{equation}   
The \emph{weak GW barycenter space} generated by  $\{[\mathbb{X}^s]\}_{s\in [S]}$ is defined by:

    \begin{equation}\label{eq: bary bu}
        \widetilde{\mathrm{Bary}}(\{[\mathbb{X}^s]\}_{s=1}^S):=\left\{[\mathbb{Y}_\lambda] \mid \nabla G_\lambda([\mathbb Y_\lambda])=0 \text{ for some } \lambda\in \Delta_{S-1}\right\}.
\end{equation}    
\end{definition}

In  \eqref{eq: bary space very general} we consider all minimizers $[\mathbb{Y}_\lambda]$ of the general GW synthesis problem \eqref{eq: syn gw}, with $\lambda \in \Delta_{S-1}$, that is, $[\mathbb Y_\lambda]\in \argmin_{[\mathbb Y]} G_\lambda([\mathbb Y])$ (also referred to as \emph{Fr\'echet means}). 
The number of nodes in the networks from the set \eqref{eq: bary space very general}, as well as the distribution of mass among them, is not predetermined in advance, unlike in the set $\mathrm{Bary}_{M,\q}(\{(\mathbf X^s,\p^s)\}_{s=1}^S)$ defined in \eqref{eq: bary space for M and q fixed}.
In contrast, the set given by \eqref{eq: bary bu} consists of the classes $[\mathbb{Y}_\lambda]$ that are critical points of the objective functional in the minimization problem \eqref{eq: syn gw}, i.e., of the functional $G_\lambda$ defined in \eqref{eq: barycenter functional} for some $\lambda \in \Delta_{S-1}$ (also referred to as \emph{Karcher means}).

\begin{lemma}
\label{thm: equiv}
     In the space of finite networks, consider a collection of templates $\{\mathbb{X}^s\}_{s \in [S]}$, and another network $\mathbb Y^*$.
    Let $\mathbb X_b^s=(\mathbf X^s_b, \q_b)$, for  $s\in [S]$, and $\mathbb Y^*_b=(\mathbf Y_b^*, \q_b)$ be an optimal blow-up of $\{\mathbb{X}^s\}_{s \in [S]}$ with respect to $\mathbb Y^*$ as in Algorithm \ref{alg: blow up}. 
    \begin{itemize}[leftmargin=*]       
        \item  If $[\mathbb Y^*]\in \widetilde{\mathrm{Bary}}(\{[\mathbb{X}^s]\}_{s=1}^S)$, then $\mathbb Y^*\sim^w \left(\sum_{s=1}^S\lambda_s \mathbf X^s_b,\q_b\right)$ for some vector $\lambda \in \Delta_{S-1}$.
        \item If $[\mathbb Y^*]\in {\mathrm{Bary}}(\{[\mathbb{X}^s]\}_{s=1}^S)$, then $\mathbb Y^*\sim^w \left(\sum_{s=1}^S\lambda_s \mathbf X^s_b,\q_b\right)$ for some vector $\lambda \in \Delta_{S-1}$.
    \end{itemize}
\end{lemma}
\begin{proof}
    If $[\mathbb Y^*]\in \widetilde{\mathrm{Bary}}(\{[\mathbb{X}^s]\}_{s=1}^S)$, then, using Proposition \ref{prop: grad via blow up}, we have $0=\nabla G_\lambda[\mathbb Y^*]=\sum_{s=1}^S\lambda_s \mathbf X^s_b-\mathbf Y_b^*$. Thus, $\mathbb Y^*\sim^w(\mathbf Y_b^*,\q_b)=\left(\sum_{s=1}^S\lambda_s \mathbf X^s_b,\q_b\right)$, for some vector $\lambda \in \Delta_{S-1}$. 

    If $[\mathbb Y^*]\in {\mathrm{Bary}}(\{[\mathbb{X}^s]\}_{s=1}^S)$ then, {due to first-order optimality conditions \cite{chowdhury2020gromov}} (see Remark \ref{rem: first order calc} in Suppl. Mat. \ref{app: weak}), $[\mathbb Y^*]\in \widetilde{\mathrm{Bary}}(\{[\mathbb{X}^s]\}_{s=1}^S)$.  Hence, $\mathbb Y^*\sim^w\left(\sum_{s=1}^S\lambda_s \mathbf X^s_b,\q_b\right)$. 
\end{proof}

In Suppl. Mat. \ref{app: weak}, we provide a direct proof of the last bullet point in Lemma \ref{thm: equiv} that does not rely on the first-order optimality condition, but instead builds on the fixed-point approach following the ideas of the proof of Theorem \ref{prop: formula_bary}.

In Lemma \ref{thm: equiv}, the blow-up technique can be interpreted as a \emph{`change of coordinates'}, in which all nodes are re-aligned, relabeled, or reordered with respect to those of a GW barycenter $\mathbb Y^*$. Once this alignment is established, the new templates (which are weakly isomorphic to the original ones) have the same number of nodes and share a common mass distribution $\mathbf{q}_b$. The GW barycenter formula (for the weight matrix) will then take the form $\sum_{s=1}^S\lambda_s \mathbf X^s_b$, which resembles the classical Euclidean barycenter {(see \eqref{eq: bary ad convex comb2} in Suppl. Mat. \ref{sec: BCM})}. This may indicate that the GW framework corresponds to the Euclidean geometry on the space of real square matrices of fixed size ($M_b \times M_b$).
However, there are many caveats. Appropriate mass must be distributed at the nodes of the network $\left(\sum_{s=1}^S\lambda_s \mathbf X^s_b,\q_b\right)$, namely the probability vector $\mathbf{q}_b$, which does not only depend on  the templates but also on $\mathbb{Y}^*$. Also, the order or labels of the coordinates of $\sum_{s=1}^S\lambda_s \mathbf X^s_b$ depend on the blow-up template matrices $\{\mathbf X_b^s\}_{s\in [S]}$, which, in turn, are determined by optimal GW assignments with respect to $\mathbb Y^*$. Moreover, if a new vector $\lambda \in \Delta_{S-1}$ is chosen, the formula of the barycenter does no longer apply and new blow-ups of the templates must be taken respect to a new $\mathbb Y^*$. For further discussion, see Remark \ref{rem: remark trivial} in Suppl. Mat. \ref{app: weak}.  We also refer the reader to Suppl. Mat. \ref{app: GW Was} for an additional discussion on the parallels and differences between Wasserstein and GW barycenters, as well as the methods used to compute them.

\subsection{GW Barycenters via Blow-Up: The Analysis Problem}\label{sec: analysis bu}

As in Theorem \ref{thm: analysis}, where we used the characterization of elements in $\mathrm{Bary}_{M,\q}(\{(\mathbf X^s,\p^s)\}_{s\in [S]})$ given by Theorem \ref{prop: formula_bary} to derive a method that exactly recovers the barycentric coordinates $\lambda$--leading to Algorithm \ref{alg: analysis} which also applies beyond strict barycenter inputs--we now use \ref{prop: Karcher} to derive a method that exactly recovers the barycentric coordinates $\lambda$ for elements in $\widetilde{\mathrm{Bary}}(\{[\mathbb X^s]\}_{s\in [S]})$, 
and which likewise extends beyond exact barycenter inputs, giving rise to a subsequent procedure described in Algorithm \ref{alg: analysis blow up}.

\begin{proposition}\label{prop: Karcher}
    Let  $\{\mathbb X^s\}_{s=1}^S$ be finite templates and let $\lambda\in \Delta_{S-1}$.  Let $\mathbb Y$ be a finite network class. Consider the blow-up matrices  $\mathbf X^1_b,\dots, \mathbf X^S_b$ and network $(\mathbf Y_b,\q_b)$ obtained  with  Algorithm \ref{alg: blow up}. 
    Then, 
$\|\nabla G_\lambda([\mathbb Y])\|_{\mathrm{Tan}_{[\mathbb Y]}}^2=\lambda^T\mathcal{A}\lambda,$
    where 

    \begin{equation}\label{eq: mat A}
        \mathcal{A}_{sr}:=\mathrm{tr}_{\q_b}\left((\mathbf X^s_b-\mathbf Y_b)^T(\mathbf X^r_b-\mathbf Y_b)\right), \qquad \forall s,r\in [S].
\end{equation}
\end{proposition}

\begin{proof}
    This proof is inspired by the analysis in \cite{werenski2022measure}, and a consequence of Proposition \ref{prop: grad via blow up}.  Since the networks involved in the statement have finitely many nodes, we can use \eqref{eq: inner product = trace} to express the inner product in  the tangent space $\mathrm{Tan}_{[\mathbb Y]}$ as a weighted trace of matrices. Indeed, after blow-up  all the matrices have the same size $M_b\times M_b$. Hence, we can consider the variant of Frobenius norm in the space of $M_b\times M_b$ square matrices weighted by $\q_b\otimes \q_b$:

\begin{align*}
    \|\nabla G_\lambda([\mathbb Y])\|_{\mathrm{Tan}_{[\mathbb Y]}}^2 &=\mathrm{tr}_{\q_b}\left(\left(\sum_{s=1}^S\lambda_s\mathbf X^s_b-\mathbf Y_b\right)^T\left(\sum_{r=1}^S\lambda_r\mathbf X^r_b-\mathbf Y_b\right)\right)\\ 
    &=\sum_{s,r=1}^S\lambda_s\lambda_r\underbrace{\mathrm{tr}_{\q_b}\left(  (\mathbf X^s_b-\mathbf Y_b)^T  (\mathbf X^r_b-\mathbf Y_b)\right)}_{\mathcal \mathcal \mathcal A_{sr}}\\[-6pt]
    &=\lambda^T \mathcal{A} \lambda.
\end{align*}
\end{proof}

\begin{theorem}\label{prop: gradient method} 
Let $\{[\mathbb X^s]\}_{s\in[S]}$ be classes of finite networks. Given a finite class $[\mathbb Y]$, consider a blow-up $(\mathbf Y_b,\q_b)$,  $\mathbf X^1_b,\dots, \mathbf X^S_b$ as in Algorithm \ref{alg: blow up}, and
let $\mathcal A$ be the $S\times S$ matrix given by the expression \eqref{eq: mat A} in Proposition \ref{prop: Karcher}.
Then  $[\mathbb Y]\in \widetilde{\mathrm{Bary}}(\{[\mathbb X^s]\}_{s\in [S]})$ if and only if 

\begin{equation}\label{eq: blowup analysis prob}
  \min_{\lambda\in \Delta_{S-1}}\lambda^T \mathcal{A} \lambda = 0,  
\end{equation}
and the solution argument $\lambda_{\mathbb Y}\in\Delta_{S-1}$ is the corresponding weight vector to $[\mathbb Y]$ as a weak GW barycenter in the space given in \eqref{eq: bary bu}.
Equivalently, 
$\lambda_{\mathbb Y}$
   solves the linear  problem

\begin{equation}\label{eq: Llambda=c}
    {L \, \lambda = c} \qquad \text{ subject to } \lambda\in \Delta_{S-1}\subset \R^S,
\end{equation}
where $L\in \R^{S\times S}$ and $c\in \R^S$ are given by

\begin{equation}\label{eq: L and c}
    L_{sr}:= \mathrm{tr}_{\q_b}\left((\mathbf X^s_b)^T \, \mathbf X^r_b\right) \quad \forall s,r\in[S], \quad 
    c_s:=\mathrm{tr}_{\q_b}\left( \mathbf Y^T_b \, \mathbf X^s_b \right) \quad \forall s\in [S].
\end{equation}
\end{theorem}

\begin{proof}
First, we show that \eqref{eq: blowup analysis prob} defines a convex quadratic program.  On the one hand, the set of restrictions $\Delta_{S-1}$ convex. On the other hand,  $\mathcal A$ is a symmetric positive semidefinite $S\times S$ matrix as it is the Gram matrix of the system of matrices $\{\mathbf X^s_b-\mathbf Y_b\}_{s=1}^S$ with respect to the inner product in $L^2(Y_b\times Y_b,\mu_{Y_b}\otimes \mu_{Y_b})$ (i.e., $(A,B)\mapsto \mathrm{tr}_{\q_b}(A^TB)$ $\forall A,B\in \R^{M_b\times M_b}$)  \cite{schwerdtfeger1961introduction}.

By definition, $[\mathbb Y] \in         \widetilde{\mathrm{Bary}}(\{[\mathbb X^s]\}_{s=1}^S)$ if and only if $\nabla G_{\lambda}([\mathbb Y])=0$ for some $\lambda\in \Delta_{S-1}$. A vector is the zero vector if and only if its norm is zero, and  $\|\nabla G_{\lambda}([\mathbb Y])\|_{\mathrm{Tan}_{[\mathbb Y]}}=0$ for some $\lambda\in \Delta_{S-1}$  if and only if the problem \eqref{eq: blowup analysis prob} has a solution $\lambda\in \Delta_{S-1}$ (by Proposition \ref{prop: Karcher}). 

Finally,
we proceed as in the proof of Theorem \ref{thm: analysis}. We have just proven that   $\lambda\mapsto \lambda^T \mathcal{A}\lambda$ is non-negative and convex as a function over the whole $\R^S$.
Therefore, its critical points ($\lambda\in \R^S$ such that $\mathcal{A}\lambda=0$) are global minima.
Moreover, we can express:

\begin{equation*}
    \lambda^T\mathcal{A}\lambda = \mathrm{tr}_{\q_b}\left( \mathbf Y_b^T\mathbf Y_b\right)-2\sum_{s\in[S]}\lambda_sc_s+\sum_{s\in [S]}\sum_{r\in [S]}\lambda_s\lambda_rL_{sr},\notag
\end{equation*}
where the matrix $L$ and the vector $c$ are  given by \eqref{eq: L and c}.
As a conclusion, the critical points can be expressed as the solutions $\lambda\in \R^S$ of the linear system $L\lambda=c$. These solutions might not be unique but under the GW-BCM condition (i.e., assuming $[\mathbb Y]\in \widetilde{\mathrm{Bary}}(\{[\mathbb X^s]\}_{s=1}^S)$) there exists at least one solution $\lambda_{\mathbb Y}$ in $\Delta_{S-1}$. That is, the solutions of \eqref{eq: blowup analysis prob} are given by \eqref{eq: Llambda=c}.
\end{proof}

As a direct corollary, under the same hypothesis as  Theorem \ref{prop: gradient method}, if $[\mathbb Y]\in {\mathrm{Bary}}(\{[\mathbb X^s]\}_{s\in [S]})$ \emph{then} \eqref{eq: blowup analysis prob} is satisfied, 
and the solution argument $\lambda_{\mathbb Y}$ is its corresponding weight vector.

Algorithm \ref{alg: analysis blow up} implements the analysis scheme laid out in Theorem \ref{prop: gradient method}, and we refer the reader to Suppl. Mat. \ref{app: comp comp} for a discussion on its computational complexity.

\begin{algorithm}[ht!]
   \caption{Analysis: GW Barycenter Coordinates via Blow-up}
   \label{alg: analysis blow up}
   \begin{algorithmic}[1]
   \STATE {\bfseries \underline{Inputs:}}\\
   - Templates $\{(\mathbf{X}^s,\p^s)_{s\in [S]}\}$\\ 
   - Shape/Network: $\mathbf Y\in \R^{M\times M}$, $\q\in \mathcal{P}_M$  \\

   \smallskip

   \STATE {\bfseries \underline{Output:}} Weight vector $\lambda$\\

   \smallskip

   $\mathbf Y_b, \q_b$, $\mathbf X^1_b,\dots, \mathbf X^S_b\gets$ compute Blow-ups (i.e., apply Algorithm \ref{alg: blow up})
   \FOR{$s=1,2,\ldots, S$}
   \FOR{$r=s, s+1,\ldots, S$}
   \STATE $\mathcal \mathcal \mathcal A_{sr}=\mathrm{tr}_{\q_b}\left((\mathbf X^s_b-\mathbf Y_b)^T(\mathbf X^r_b-\mathbf Y_b)\right)$
   \STATE $\mathcal{A}_{rs}=\mathcal{A}_{sr}$
   \ENDFOR
   \ENDFOR
   \STATE $\lambda\gets \arg\min_{\lambda\in \Delta_{S-1}}\lambda^T \mathcal{A} \lambda$ 
   \RETURN $\lambda$
\end{algorithmic}
\end{algorithm}

\begin{remark}\label{remark: blowup_size}
It is worth mentioning that the blow-up step (Algorithm \ref{alg: blow up}) is essentially equivalent to solving $S$ Gromov-Wasserstein optimization problems. These $S$ concatenated GW problems are defined between networks whose sizes increase as the for-loop progressing from $s=1$ to $s=S$.  For simplicity, let us assume that the number of nodes is the same across the templates, i.e, $N^s=N$ for all $s\in [S]$. After performing the first blow-up between the first template $\mathbb X^1$, with $N$ nodes, and $\mathbb Y$, with $M$ nodes, we obtain new networks that are weakly isomorphic to the originals but have size $|\mathrm{supp}{(\pi^1)}|\geq\max\{N,M\}$, where $\pi^1\in \Pi(\p^1,\q)$ is an optimal coupling for $GW(\mathbb X^1,\mathbb Y)$. 
In the case of classical OT, any optimal coupling has support size at most 
$N+M-1$ (see, e.g., \cite[Prop. 3.4]{COTFNT}). Linear programming theory does not apply in the GW case due to the quadratic nature of the objective, even though the optimization is performed over the same set of couplings. Best-case scenarios for GW include situations in which an optimal plan is a permutation matrix (with support size $N$, when 
$N=M$), or, more generally, lies at an extreme point of the transport polytope $\Pi(\p^1,\q)$ and thus has at most $N+M-1$ nonzero entries. However, no support-size bound analogous to that in classical OT is currently known for GW minimizers, other than the trivial bound 
$NM$.
For empirical insight, we refer the reader, for e.g., to \cite[Suppl. Mat., Sec. C;  Figs. 9 $\&$ 10]{chowdhury2020gromov}, where the authors observe that representations often occur in a space much smaller than the naive requirement of size $NM$, in fact of order $N+M$. As a conclusion, as this process continues similarly for subsequent steps, involving blow-ups for the remaining templates $\mathbb X^2, \dots, \mathbb X^S$ with respect to the updated networks $\mathbb Y$, the number of nodes at each iteration $s\in [S]$ increases as well:  
In the worst case --when the GW plans are dense-- the size of the barycenter can be as large as $(NM)^S$. 
Even for sparse plans, an exponential growth of order $(M+N)^S$ is, in principle, possible. However, experimental evidence \cite{chowdhury2020gromov} supports the `linear' hypothesis that $M_b\approx SN+M$, and this has been corroborated in our experiments (cf. Section \ref{sec: experiments}). 
\end{remark}

\subsection{Relationship Between Fixed-Point and Gradient Approaches}\label{sec: relation}

In this section, we relate the two strategies considered for addressing the GW synthesis and analysis problems: the fixed-point approach, which arises from interpreting the GW barycenters as fixed points of an iteration scheme (Sections \ref{sec: Synthesis} and \ref{sec: Analysis}), and the gradient-based approach, where we define weak GW barycenters as Karcher means (Sections \ref{sec: synthesis bu} and \ref{sec: analysis bu}).

Let $(\mathbf X^s,\p^s)$, for $s\in [S]$, and $(\mathbf Y,\q)$ be representatives of the classes $\{[\mathbb X^s]\}_{s\in[S]}$ (templates) and $[\mathbb Y]$ (input shape) in $\mathcal{GM}$.
Let us preprocess them obtaining  $(\mathbf X^s_b,\q_b)$, $s\in [S]$, and $(\mathbf Y_b,\q_b)$ in  $\R^{M_b\times M_b}\times \mathcal{P}_{M_b}$, using the blow-up  Algorithm \ref{alg: blow up}. Consider GW OT plans  $\pi^s_b=\mathrm{diag}(\q_b)$ for each problem $GW((\mathbf X^s_b,\q_b),(\mathbf Y_b,\q_b))$, with $s\in [S]$. Finally, in the approach outlined in Section \ref{sec: fixed point approach}, consider $M_b$ and $\q_b$ in place of $M$ and $\q$. We observe that expression \eqref{eq: bary_fix_pi} becomes:

\begin{equation}\label{eq: relation}
    \frac{1}{\q_b (\q_b)^T}\odot\sum_{s\in [S]}\lambda_s (\pi^s_b)^T \, \mathbf X^s_b \, \pi^s_b=\frac{1}{\q_b(\q_b)^T}\odot\sum_{s\in [S]}\lambda_s \mathrm{diag}(\q_b) \,  \mathbf X^s_b  \, \mathrm{diag}(\q_b) = \sum_{s\in [S]}\lambda_s \mathbf X^s_b.
\end{equation}
In other words, under the respective GW-BCM frameworks,
the expression \eqref{eq: bary_fix_pi} from the fixed-point approach coincides with the expression in Lemma \ref{thm: equiv} satisfied by critical points $\nabla G_\lambda([\mathbb Y])=0$ or minimizers of \eqref{eq: syn gw}.  

Moreover, when formulating the analysis problem \eqref{eq: analysis_gw_bary_problem} (or, equivalently, \eqref{eq: our_analysis_bary_problem} under the GW-BCM assumption), we have the flexibility to consider any divergence. Therefore,  if 
instead of considering the classical Frobenius norm $d_a(A,B)=\|A-B\|_{\mathrm{Frob}}$ in 
Theorem \ref{thm: analysis}, we had used the variant of the Frobenius norm with weight $\q_b\otimes \q_b$,  we would have obtained that the convex quadratic problems \eqref{eq: quad frob} and  \eqref{eq: blowup analysis prob} coincide, as  $\mathcal{Q}=\mathcal A$ (to compare their expressions given in \eqref{eq: quad frob} and \eqref{eq: mat A}, using that, from \eqref{eq: relation}, one has $F(\mathbf Y_b, s)=\mathbf X_b^s$). 
Indeed, as functions from $\R^S$ to $[0,\infty)$, we would have $\Theta(\lambda)=\|G_\lambda([\mathbb Y])\|^2_{\mathrm{Tan}_{[\mathbb Y]}}$ (compare  \eqref{eq: lambda functional} and Proposition \ref{prop: Karcher}).
In particular, the linear optimization problems \eqref{eq: klambda=b} and \eqref{eq: Llambda=c} would coincide; that is, $K=L$ and $b=c$, as defined in \eqref{eq: K and b} and \eqref{eq: L and c}. 
If one tries to reformulate problem \eqref{eq: our_analysis_bary_problem} to coincide exactly with Section \ref{sec: analysis bu}, one needs to compute $\q_b$  in advance. This involves solving one GW problem per template  (as part of the blow-up procedure). The blow-up distribution $\q_b$ depends on $(\mathbf Y,\q)$ but but generally differs from $\q$.

\begin{remark}[Compatibility Test for GW Synthesis Algorithms]\label{remark: test}
 Theorem \ref{prop: gradient method} provides a \emph{necessary condition} for being a ``true'' GW barycenter. Precisely, if $[\mathbb Y]\in {\mathrm{Bary}}(\{[\mathbb{X}^s]\}_{s=1}^S)$ with corresponding coordinates $\lambda_{\mathbb Y}\in \Delta_{S-1}$, then 

   \begin{equation}\label{eq: test}
       \|\nabla G_\lambda([\mathbb Y])\|_{\mathrm{Tan}_{[\mathbb Y]}}^2=\lambda_{\mathbb Y}^T \, \mathcal{A} \, \lambda_{\mathbb Y}=0, 
\end{equation}
   where the $S\times S$ matrix $\mathcal A$ depends on the fixed templates $\{\mathbb{X}^s\}_{s=1}^S$ and on the input $\mathbb Y$, and requires preforming the blow-up representation (see \eqref{eq: mat A}). This test, however, does not provide a sufficient guarantee for being a point synthesized through \eqref{eq: syn gw}. By definition, it is instead necessary and sufficient for addressing if a network belongs to  $\widetilde{\mathrm{Bary}}(\{[\mathbb{X}^s]\}_{s=1}^S)$. In particular, when generating synthetic barycenters (e.g., via Algorithm \ref{alg: synthesis}), one can test through \eqref{eq: test} whether the obtained object is, at least, a weak GW barycenter.
\end{remark}

\section{Experimental Results}\label{sec: experiments}

In this section, we experimentally assess the performance of our algorithms for solving the GW analysis problem. We also present applications to classification. Additional use cases involving point cloud corruptions are provided in Suppl. Mat. \ref{app: experiments - corr}.

\subsection{Reconstruction Performance of Algorithms \ref{alg: analysis} and \ref{alg: analysis blow up} Under the GW-BCM}\label{sec: alg analysis 1 in 3d}
  In this section, we demonstrate the performance of the proposed GW analysis algorithms by applying them to 3D point cloud data: Figures \ref{fig: analysis_alg} and \ref{fig: 100p}
  illustrate the application of Algorithms \ref{alg: analysis} and \ref{alg: analysis blow up} on the  ModelNet40 - Princeton 3D Object Dataset \cite{wu20153d} consisting of 3D CAD models of 40 common object categories such as chairs, tables, airplanes, lamps, and more. 

  From the object category `airplane' we select $S=3$ templates and sample them to generate 3D point clouds with (varying) numbers of points $N^s$, $s=1,2,3$ (Fig. \ref{fig: analysis_alg}). We compute the associated $N^s\times N^s$ pairwise Euclidean distance matrices $\mathbf X^s$ 
within each point cloud ($s=1,2,3$).
We assign to each point uniform mass,  that is, $\p^s=\frac{1}{N^s}(1,\dots,1)\in \mathbb R^{N^s}$, for $s=1,2,3$. 
We consider either the uniform vector $\lambda=(\frac{1}{3},\frac{1}{3},\frac{1}{3})\in \Delta_2$ or a random vector $\lambda\in \Delta_2$ drawn from a Dirichlet distribution with uniform parameters $\alpha=(1,1,1)$,  which corresponds to the uniform distribution over the 2-simplex. 
For such a $\lambda$ and the selected templates, we synthesize a GW barycenter $M\times M$ matrix $\mathbf{Y}$ by using the function 
\texttt{ot.gromov.gromov\_barycenters} from the POT Library \cite{flamary2021pot} with a predetermined size $M$ and uniform probability $\q=\frac{1}{M}(1,\dots,1)\in \R^{M}$, which is based on a fixed-point algorithm in the spirit of Algorithm \ref{alg: synthesis}. 
Then, we apply Algorithm \ref{alg: analysis} or Algorithm \ref{alg: analysis blow up} to the synthetic GW barycenter, $\q$, and the templates $\{(\mathbf X^1,\p^1),(\mathbf X^2,\p^2),(\mathbf X^3,\p^3)\}$  to get the estimate $\widetilde{\lambda}\approx\lambda$.

In Fig. \ref{fig: analysis_alg} we consider $\lambda=(\frac{1}{3},\frac{1}{3},\frac{1}{3})$, and after estimation, we reconstruct the shape as:
\begin{itemize}[leftmargin=*]
    \item If we use Algorithm \ref{alg: analysis} for recovering the vector $\widetilde \lambda=(\widetilde \lambda_1,\widetilde \lambda_2, \widetilde \lambda_3)$, we compute the matrix

\begin{equation}\label{eq: Y analysis alg}
    \mathbf{Y}_{\mathrm{bary}}:= \frac{1}{\q\q^T}\odot\sum_{s\in [S]}\widetilde \lambda_s (\pi(\mathbf Y, s))^T \mathbf X^s \pi(\mathbf Y,s)\in \R^{M\times M},
\end{equation}
where, for each $s\in[S]$,  $\pi(\mathbf Y, s)$ is an optimal coupling for $GW((\mathbf Y,\q),(\mathbf X^s,\p^s))$ computed using the function \texttt{ot.gromov.gromov\_wasserstein} from the POT Library. 
\item If we use Algorithm \ref{alg: analysis blow up} for recovering a vector $\widetilde \lambda=(\widetilde \lambda_1,\widetilde \lambda_2, \widetilde \lambda_3)$, we compute the matrix 

\begin{equation}\label{eq: Y analysis alg blow up}
    \mathbf{Y}_{\mathrm{bary bu}}:=\sum_{s\in[S]}\widetilde \lambda_s \, \mathbf X^s_b \in \R^{M_b\times M_b},
\end{equation}
where $\{\mathbf{X}^s_b\}_{s=1}^S$ are the blow-up matrices obtained after preprocessing via Algorithm \ref{alg: blow up} applied to the original templates with respect to $(\mathbf Y,\q)$, and where  $\q_b\in\mathcal{P}_{M_b}$ is also derived.
\end{itemize}
  For visualization, we apply the classical multidimensional scaling embedding (MDS) \cite{mead1992review} with 3 components to the synthetic GW barycenter matrix  $\mathbf Y$ and the recovered matrices $\mathbf Y_{\mathrm{bary}}$ and $\mathbf Y_{\mathrm{barybu}}$ given by the two approaches \eqref{eq: Y analysis alg} and \eqref{eq: Y analysis alg blow up}, respectively. For assessing the reconstruction accuracy, we use the function \texttt{ot.gromov.gromov\_wasserstein} from  the POT Library to compute $GW((\mathbf Y,\q),(\mathbf Y_{\mathrm{bary}},\q))$ and $GW((\mathbf Y,\q),(\mathbf Y_{\mathrm{barybu}},\q_b))$.

In Fig. \ref{fig: 100p} we estimate GW barycentric coordinates $\widetilde \lambda\approx \lambda$ for synthetic GW barycenters 
built with randomly generated  $\lambda$. We synthesize with input size $M=30$ and notice that when using templates having the \emph{same number of nodes} (equal to $30$), the blow-up average size (calculated when applying Algorithm \ref{alg: analysis blow up}) is
$M_b=30$; and when using templates having \emph{different number of nodes} (between $25$ and $35$),  the blow-up average size obtained is
 $M_b\approx 107$ (same order as $4\times30$, cf. Remark \ref{remark: blowup_size}).}

\begin{figure}[h!]
\centering
    \begin{subfigure}{\linewidth}

        \centering
        \includegraphics[width=\linewidth]{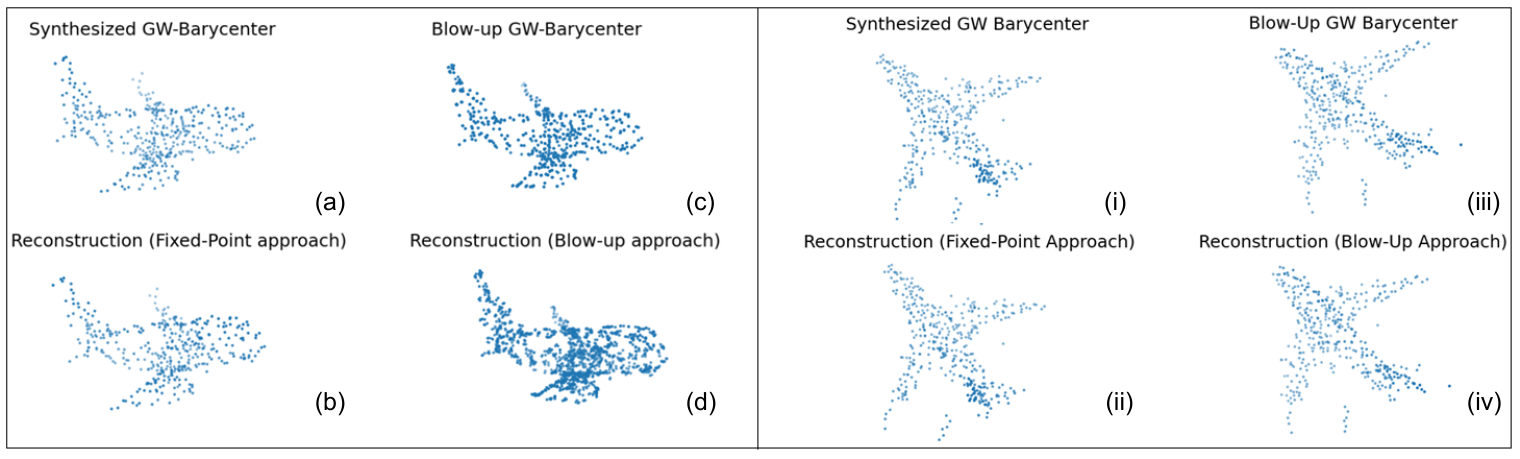}

        \caption*{\small{Reconstructions of synthetic GW barycenters using MDS on distance matrices.}}
    \end{subfigure}
        \begin{subfigure}{0.25\linewidth}\centering
\small

\setlength{\tabcolsep}{4pt}
\renewcommand{\arraystretch}{1.1}
\centering
\begin{tabular}{|c|}
\hline
$\mathrm{GW}(\text{\textbf{(a)}},\text{\textbf{(b)}})\approx 10^{-16}$ \\
$\mathrm{GW}(\text{\textbf{(a)}},\text{\textbf{(c)}})\approx 10^{-16}$ \\
$\mathrm{GW}(\text{\textbf{(a)}},\text{\textbf{(d)}})\approx 10^{-5}$  \\
\hline
\end{tabular}

\caption*{\small{Order of GW Errors.}}
\end{subfigure}
        \begin{subfigure}{0.45\linewidth}
        \centering
        \includegraphics[width=\linewidth]{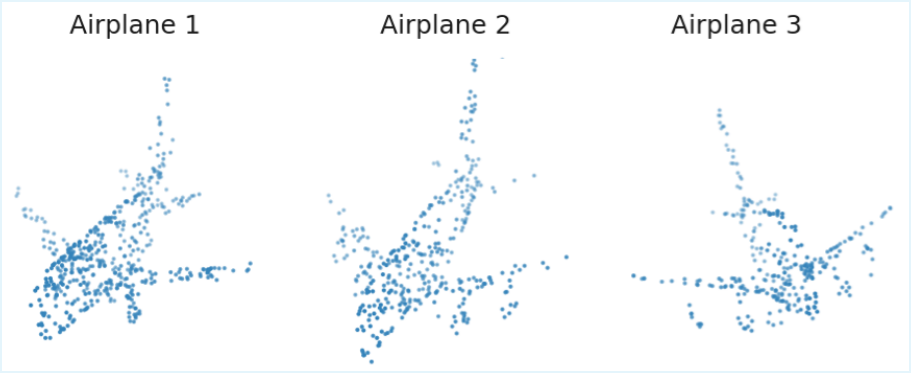}

        \caption*{\small{Templates.}}
    \end{subfigure}
    \begin{subfigure}{0.25\linewidth}
\centering
\small
\setlength{\tabcolsep}{4pt}
\renewcommand{\arraystretch}{1.1}
\centering
\begin{tabular}{|c|}
\hline
$\mathrm{GW}(\text{\textbf{(i)}},\text{\textbf{(ii)}})\approx 10^{-16}$ \\
$\mathrm{GW}(\text{\textbf{(i)}},\text{\textbf{(iii)}})\approx 10^{-17}$ \\
$\mathrm{GW}(\text{\textbf{(i)}},\text{\textbf{(iv)}})\approx 10^{-17}$ \\
\hline
\end{tabular}

\caption*{\small{Order of GW Errors.}}
\end{subfigure}

   \caption{\small{\textbf{(a)} GW barycenter $(\mathbf Y,\q)$ 
    consisting of $M=400$ points and uniform mass $\q$, synthesized with the templates above sampled with \emph{different number of points} between $300$ and $500$.
    \textbf{(b)} MDS of the matrix given by \eqref{eq: Y analysis alg} (i.e., $(\mathbf Y_{\textrm{bary}},\q)$) after applying the fixed-point analysis Algorithm \ref{alg: analysis} to \textbf{(a)} to estimate $\widetilde\lambda\approx \lambda$ (error of order $10^{-14}$). 
    \textbf{(c)} The resulting blow-up of \textbf{(a)} (i.e., $(\mathbf Y_{b},\q_b)$) when Algorithm \ref{alg: blow up} is applied to the templates and \textbf{(a)}. The size $M_b$ was about $1500$, which is of the same order as $400+3N$ where $300 \leq N\leq 500$ (cf. Remark \ref{remark: blowup_size}).
    \textbf{(d)} MDS of the matrix given by  \eqref{eq: Y analysis alg blow up} (i.e., $(\mathbf Y_{\textrm{barybu}},\q_b)$) after applying the blow-up analysis Algorithm \ref{alg: analysis blow up} to \textbf{(a)} to recover $\widetilde\lambda\approx \lambda$ (error of order $10^{-3}$).
    \textbf{(i)--(iv)} Analogous to \textbf{(a)--(d)} but \emph{uniformly sampling} the templates with $M=400$ points. Errors in $\lambda$ estimation for \textbf{(ii)} and \textbf{(iv)} are of order $10^{-15}$ and $10^{-16}$, respectively. Blow-up size $M_b=400$.
    Difference: \textbf{(a)} does not satisfy the test from Remark \ref{remark: test}, while \textbf{(i)} passes the test. 
    }} 
    \label{fig: analysis_alg}
\end{figure}

\begin{figure}[H]
    \centering
   \includegraphics[width=\linewidth]{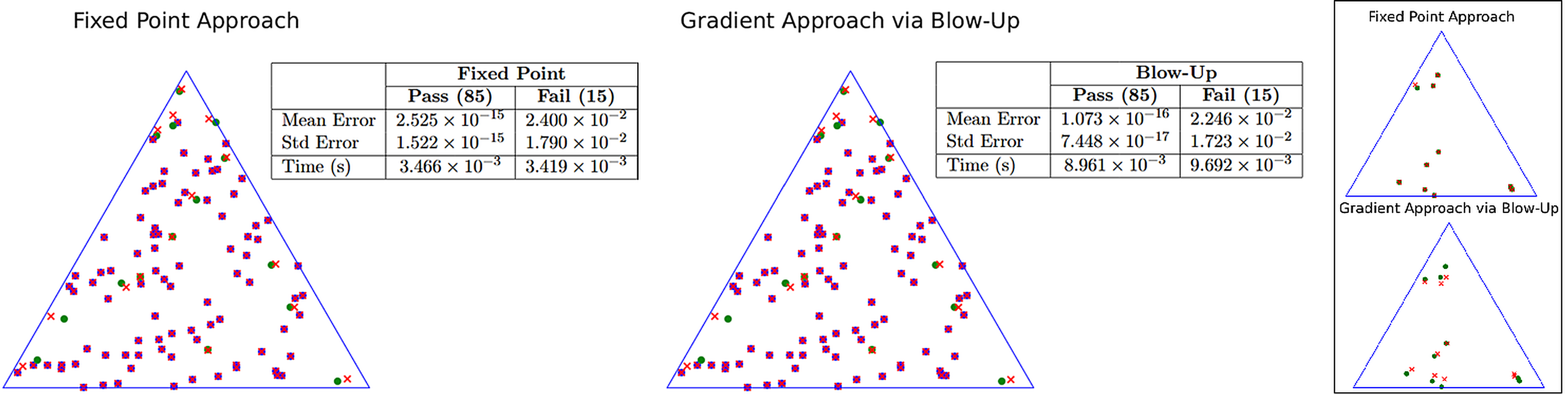}

    \caption{\small{Depiction of 100 different randomly generated $\lambda\in \Delta_2$ (dots) and their estimation $\widetilde \lambda$ (red crosses) from synthesized GW barycenters, using Algorithm \ref{alg: analysis} ({left}) and Algorithm \ref{alg: analysis blow up} ({right}). Each $\lambda$ is represented as a convex combination of the vertices of an equilateral triangle. We compute mean and standard deviation of the reconstruction error $\lVert \lambda - \widetilde\lambda \rVert_2$ and of the average wall-clock time 
    (in seconds) of our GW-analysis methods over $100$ independent runs. For each experiment, we use   
     3 templates, \emph{uniformly sampled} at 30 points, with uniform distribution of mass across nodes. The 100 GW barycenters, with $M=30$ points and $\q=\frac{1}{30}(1,\dots,1)\in \R^{30}$ each, were synthesized using \texttt{ot.gromov.gromov\_barycenters} from the POT Library.  Blow-up average size when using Algorithm \ref{alg: analysis}: $M_b=30$.
     Blue dots (85) represent synthesized output that pass the compatibility test in Remark \ref{remark: test}, whereas green dots do not (15). In the right-hand boxed panel we run 10 experiments with templates \emph{sampled at varying rates} noticing all the synthetic outputs fail to be critical points of the GW synthesis functional (all green dots).      
}}
    \label{fig: 100p}
\end{figure}

\subsection{Clustering, Classification and Visualization}\label{sec: classif main}
In this section, we expand on the motivating visualization in Fig. \ref{fig: baryspace01}, illustrating how GW barycentric coordinates capture meaningful structure in the data: they separate classes while showing slight overlaps when samples shared geometric features.  Generally, suppose a dataset contains $C$ classes, and $u$ templates are randomly selected from each class. For each sample, one can compute GW barycentric coordinates with respect to the templates using either Algorithm \ref{alg: analysis} or Algorithm \ref{alg: analysis blow up}. Each sample is thus represented by a coordinate vector $\lambda \in \mathbb{R}^S$, where $S = u \times C$. These coordinates can be used directly for supervised classification (e.g., k-NN) or for unsupervised analysis such as K-Means clustering. To visualize the coordinates when $S>3$, one can apply another dimensionality reduction such as t-SNE to map $\mathbb{R}^S$ to $\mathbb{R}^2$.

As a proof of concept, we apply this framework to the Point Cloud MNIST 2D dataset \cite{Garcia2023PointCloudMNIST2D}.
We first consider \emph{classification experiments} with $C=2$ and $u=1$; results are illustrated in Fig. \ref{fig: 04_one_template} and reported in Table \ref{tab:classification_results}. Algorithms \ref{alg: analysis} and \ref{alg: analysis blow up} perform comparably in this setting, although Algorithm \ref{alg: analysis blow up} incurs additional computational cost due to the preprocessing step (Algorithm \ref{alg: blow up}). To generate Table \ref{tab:classification_results}, we treat the $\lambda$-vectors as features and apply a k-NN classifier. We compare against two baselines. First, we consider a naive 1-NN classifier based directly on the GW distance: by using only the selected templates as the training set, each sample $\mathbb{Y}$ is assigned the label of the template $\mathbb{X}^s$ minimizing $\mathrm{GW}(\mathbb{X}^s,\mathbb{Y})$. Second, we consider the GW dictionary-learning approach of \cite{D-vincent2021online} (without regularization), implemented in the POT Library \cite{flamary2021pot}. 
Dictionary atoms (of fixed size $20$, learned through 5 epochs) are learned using the full dataset, and each sample is represented by so-called \emph{unmixing weights} obtained by solving a GW linear-like approximation problem. This objective differs from the GW barycenter problem and should be viewed as a computationally tractable relaxation. Additional classification experiments can be found in Suppl. Mat. \ref{sec: additional_classification_app}.

\begin{figure}[H]
    \centering
    \begin{subfigure}{0.4\linewidth}
        \centering
        \includegraphics[width=\linewidth]{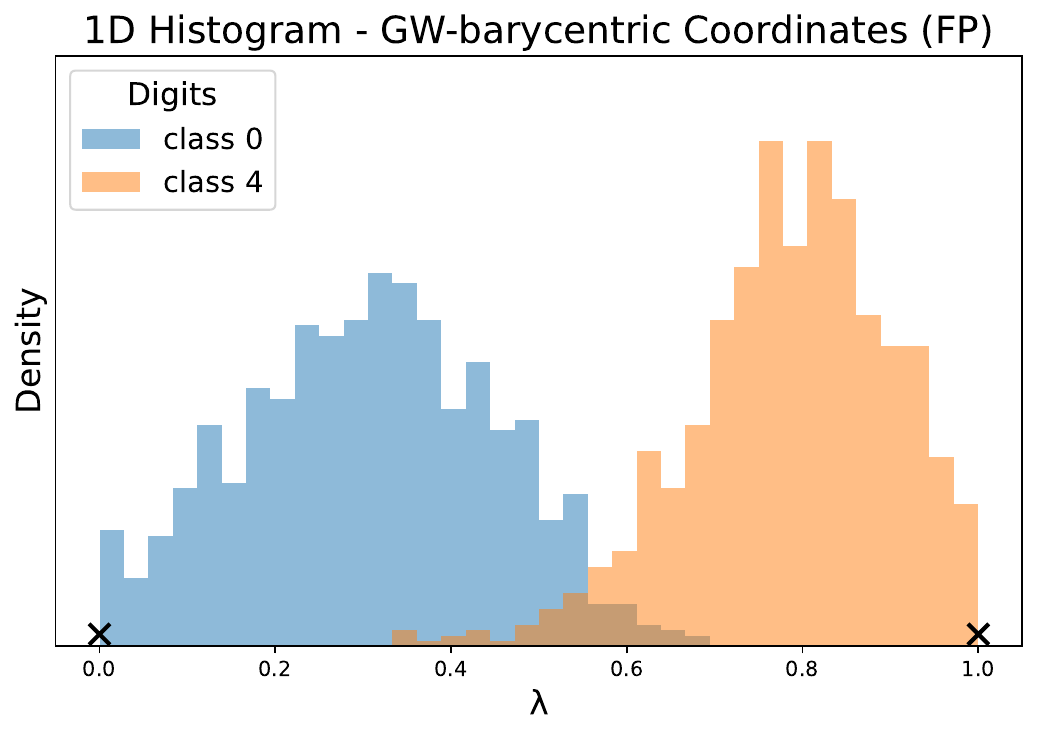}

        \caption{\small{Fixed-Point approach -- Digits $0\&4$.}}
        \label{fig:fphist}
    \end{subfigure}
    \begin{subfigure}{0.4\linewidth}
        \centering
         \includegraphics[width=\linewidth]{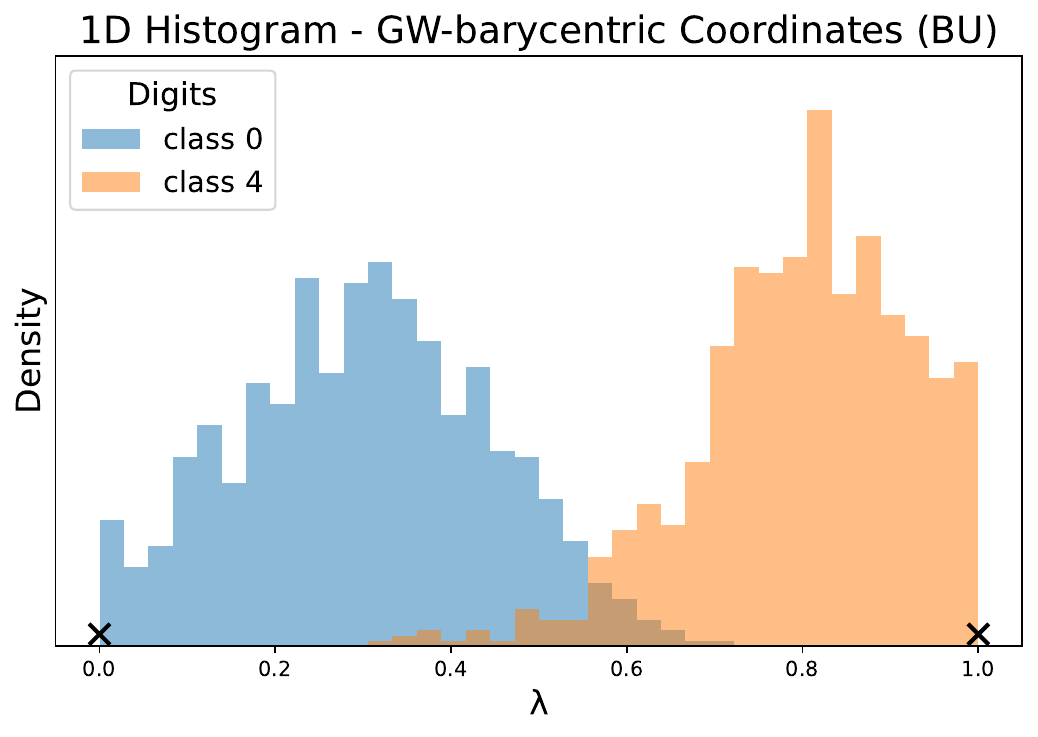}

        \caption{\small{Blow-Up approach -- Digits $0\& 4$.}}
        \label{fig:buhist}
    \end{subfigure}
    \begin{subfigure}{0.18\linewidth}
    \centering 
         \includegraphics[width=0.9\linewidth]{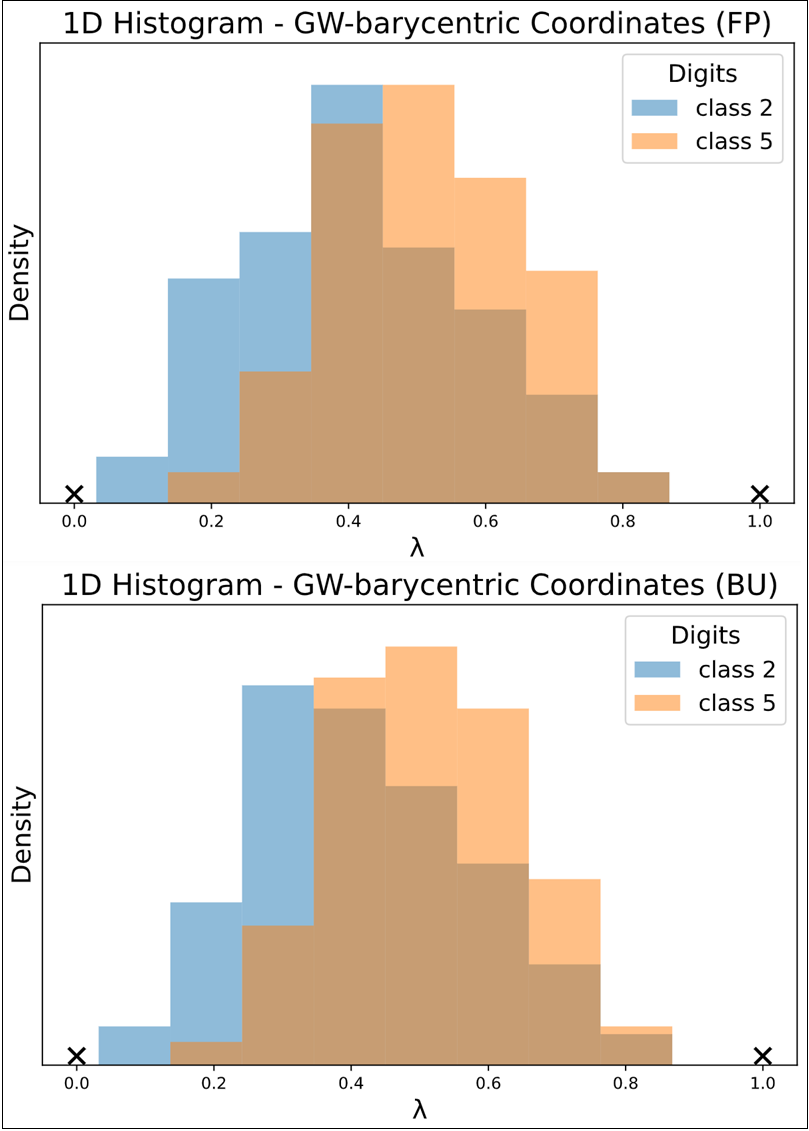}
\vspace{+0.25in}  
        \caption{\small{Digits $2\&5$.}}
        \label{fig:buhist25}
    \end{subfigure}

    \caption{\small{(a)-(b): We considered the classes of digits 0 and 4 from the Point Cloud MNIST 2D dataset \cite{Garcia2023PointCloudMNIST2D}, and selected 900 samples and one template per class. The GW coordinates $(\lambda_1,\lambda_2)\in\Delta_1$ of each data point are computed using Algorithm \ref{alg: analysis} in (a) and Algorithm \ref{alg: analysis blow up} in (b). Since the representation is intrinsically one-dimensional, the plots visualize the empirical distributions of the GW barycentric coordinate $\lambda := \lambda_1 \in [0,1]$. Regions of overlap correspond to geometrically ambiguous samples and explain the observed classification errors of k-NN (see Table \ref{tab:classification_results}). 
   The mass across nodes in each point cloud is not uniformly distributed, and each point cloud in the dataset has a different number of nodes, varying from 50 to 200 points with an average of 120. The blow-up average size $M_b$ from the experiments on the right  is 
   $415.5$
   (which is of the same order as $3\times 120$, cf. remark \ref{remark: blowup_size}). Subplot (c), where 200 samples of the digits 2 and 5 where considered, indicates that the framework captures meaningful structures in the data since those digits share their internal structure and overlap is expected. }}
    \label{fig: 04_one_template}
\end{figure}

\begin{table}[ht!]
\centering
\small
\setlength{\tabcolsep}{4pt}
\resizebox{\linewidth}{!}{\begin{tabular}{lcccc}
\toprule
 & \textbf{$\lambda$-GW (FP) + k-NN} & \textbf{$\lambda$-GW (BU) + k-NN} &\textbf{ GW 1-NN}  & \textbf{GW dict + k-NN} \\
\midrule
\textbf{Accuracy (mean $\pm$ std) }
& 0.94 $\pm$ 0.009 
& 0.95 $\pm$ 0.012
& 0.89 $\pm$ 0.016
& 0.93 $\pm$ 0.016 \\
\textbf{Time (s)}
& 92.67 ($\lambda$-GW coord)
& 147.23 ($\lambda$-GW coord)
& 97.17 (avg run)
& 50.21 (learn atoms) + 10.05 (unmix weights) \\
\bottomrule
\end{tabular}}
\caption{\small{Classification performance and runtime comparison on the Point Cloud MNIST 2D dataset \cite{Garcia2023PointCloudMNIST2D} for digits 0 and 4 with $200$ samples each.
For columns 2, 3, and 5, each sample is represented either by GW barycentric coordinates $\lambda$ computed via the fixed-point (FP) approach (Algorithm \ref{alg: analysis}), the blow-up (BU) approach (Algorithm \ref{alg: analysis blow up}), or by GW dictionary unmixing weights \cite{D-vincent2021online}, respectively, and then classified using k-NN in the resulting feature space.
For k-NN methods, we use repeated stratified 70/30 train/test splits (20 outer repetitions) and select $k$ by 5-fold stratified cross-validation (CV) performed on the training split within each outer repetition. Specifically, $k$ is chosen across $\{1,3,5,7,9\}$ to maximize the mean CV accuracy; the classifier is then retrained on the training set and evaluated on the test set from the corresponding outer split.
In column 4 (GW 1-NN), each sample is assigned the label of the template that minimizes the GW distance. 
}}
\label{tab:classification_results}
\end{table}

We then consider \emph{unsupervised analysis} and \emph{embedding visualization} using multiple templates per class. For $u=3,6$, we perform K-Means on the GW coordinate space and apply t-SNE for visualization; the resulting embeddings are shown in Fig. \ref{fig:combined_templates_tsne}. As illustrated in Figures \ref{fig:templates_side1} and \ref{fig:tsne_side}, increasing the number of templates per class leads to clearer separation between classes in the embedded space.   
Additional visualizations with three classes ($C=3$) are provided in Fig. \ref{fig:decision_regions_simplex} and Fig. \ref{fig: Delta_2} in Suppl. Mat. (with $u=1$). 

\begin{figure}[ht!]
    \centering
    \begin{subfigure}[t]{0.33\linewidth}
        \centering
        \includegraphics[width=\linewidth]{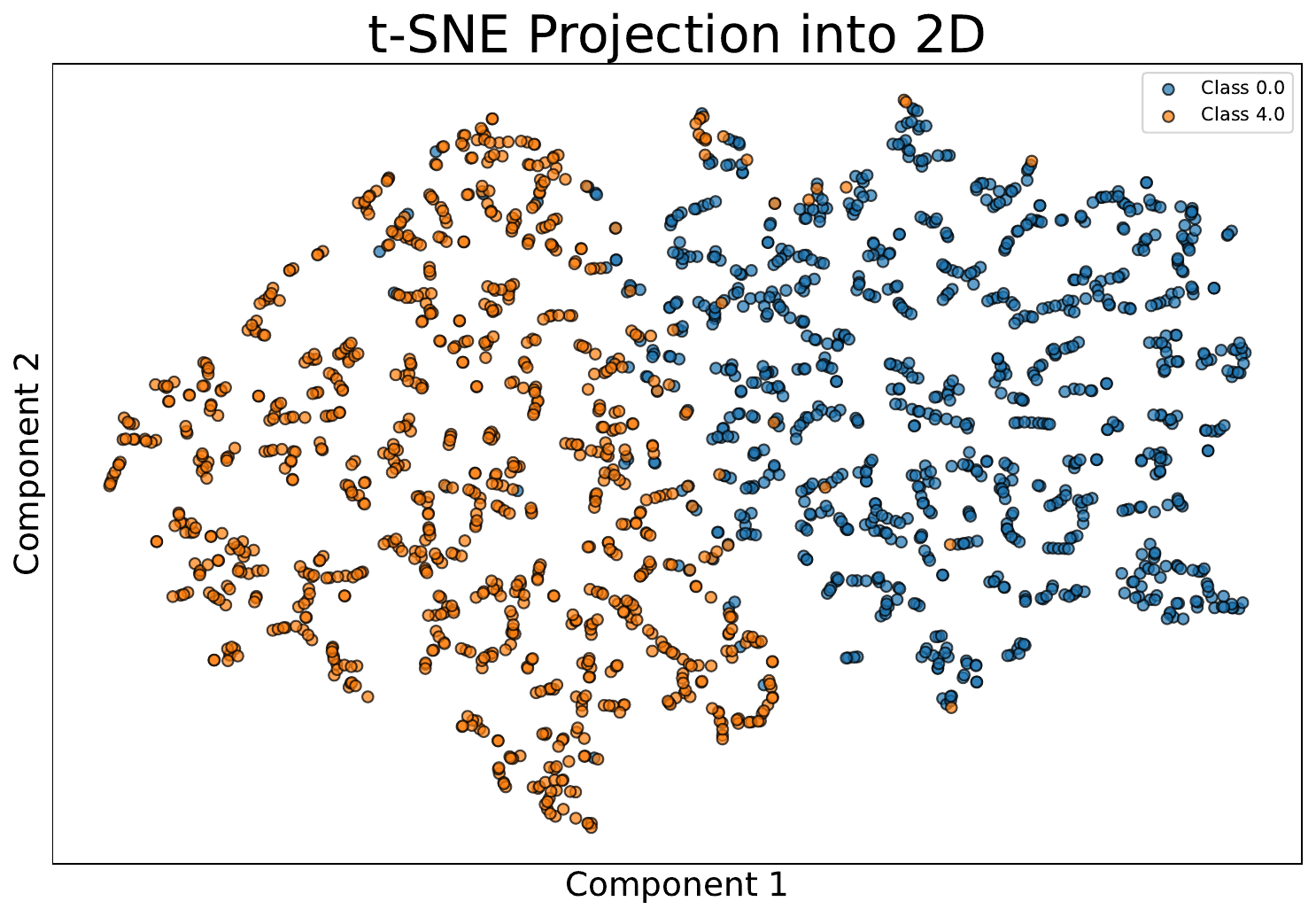}

        \caption{\small{3 templates per class--FP.}}
        \label{fig:templates_side1}
    \end{subfigure}%
    \begin{subfigure}[t]{0.33\linewidth}
        \centering
        \includegraphics[width=\linewidth]{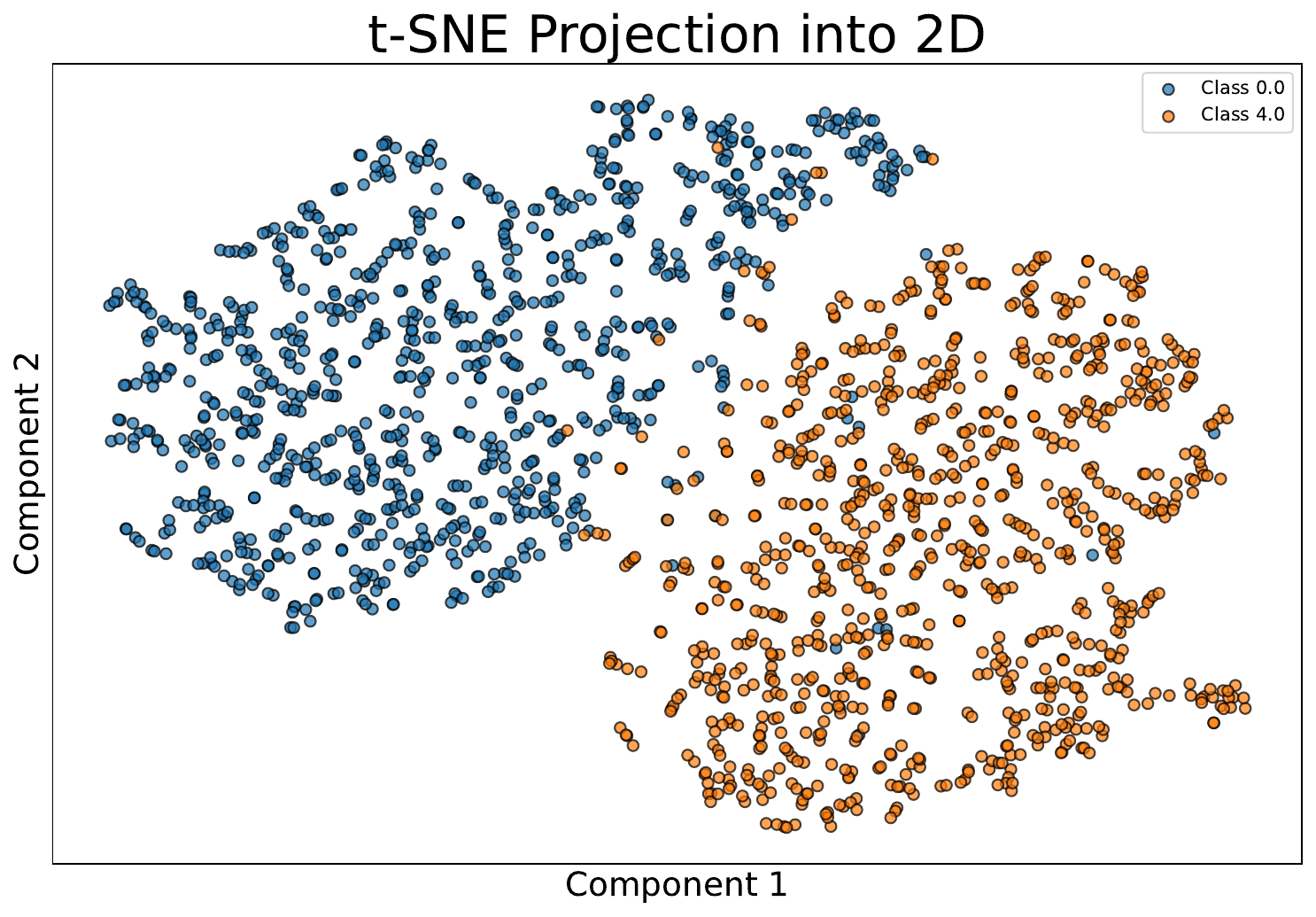}

        \caption{6 templates per class--FP.}
        \label{fig:tsne_side}
    \end{subfigure}
            \begin{subfigure}[t]{0.33\linewidth}
        \centering
        \includegraphics[width=\linewidth]{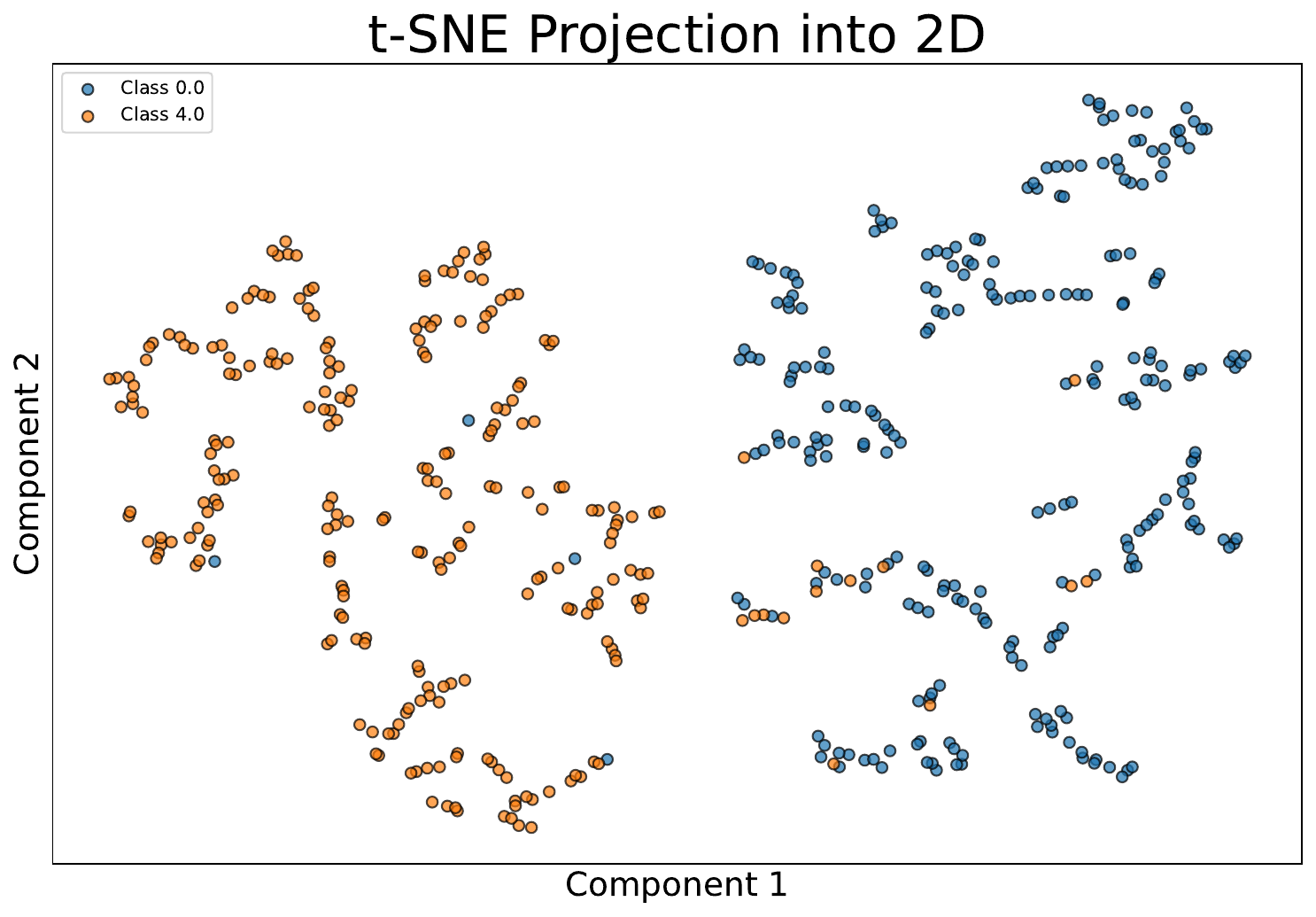}

        \caption{\small{6 templates per class--BU.}}
    \end{subfigure}

    \caption{\small{Visualization through t-SNE of estimated $\lambda$ coordinates via Algorithm \ref{alg: analysis} (a)-(b) and Algorithm \ref{alg: analysis blow up} (c). For (a) and (b) we consider 1800 total samples, and 500 for (c).
    Multiple templates per class: (a) 3; (b) 6; (c) 6.  
    Clustering performance via K-Means in $\lambda$-space: Adjusted Rand Index for FP analysis, (a) 0.76,   (b) 0.81; for BU analysis, 0.86 when using 3 templates per class.
    }}
    \label{fig:combined_templates_tsne}
\end{figure}

\begin{figure}[ht!]
    \centering
    \begin{subfigure}{0.4\linewidth}
        \centering
        \includegraphics[width=\linewidth]{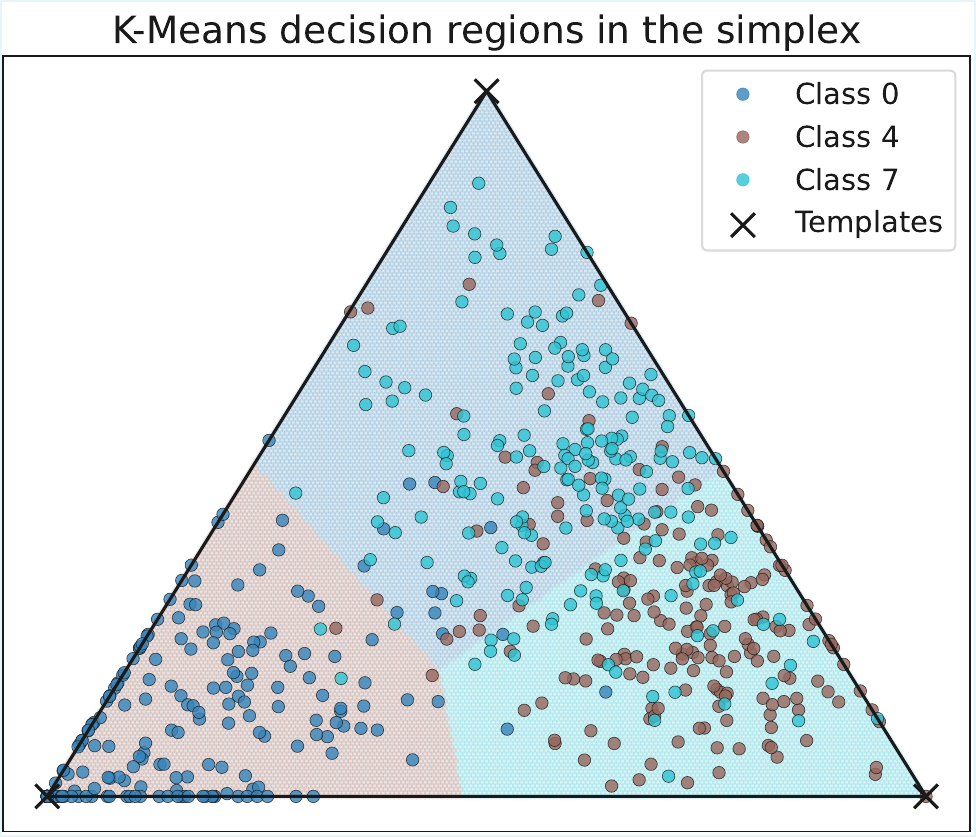}

        \caption{\small{Fixed-Point approach.}}
    \end{subfigure}
    \hfill
    \begin{subfigure}{0.4\linewidth}
        \centering
        \includegraphics[width=\linewidth]{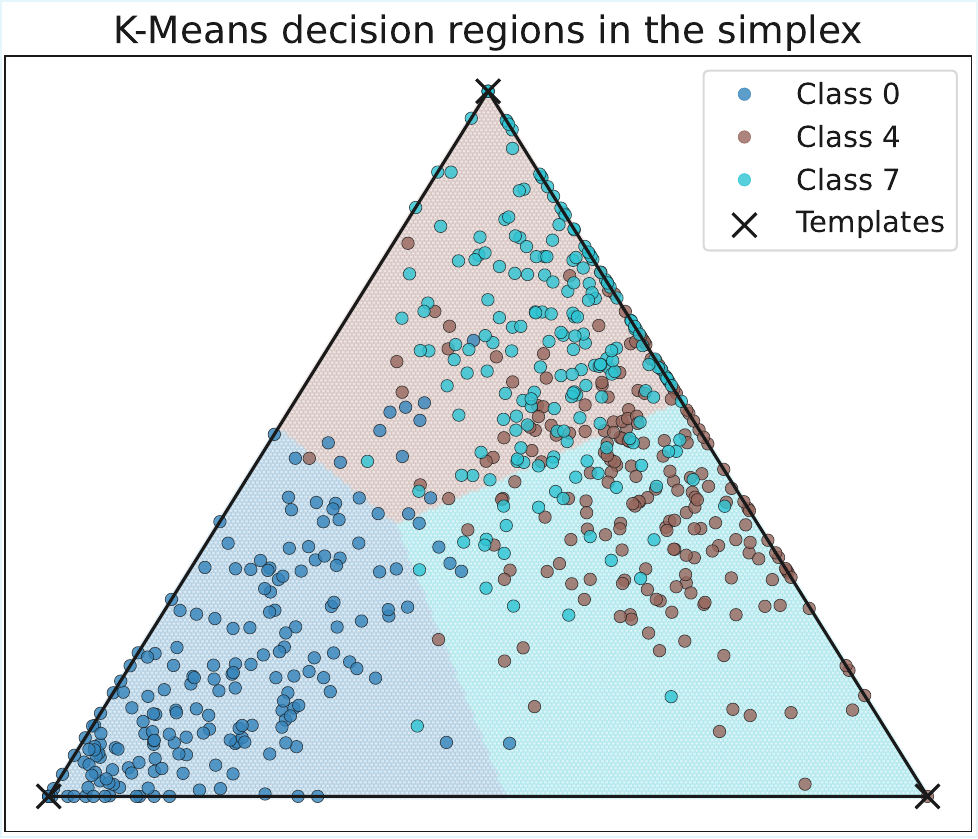}

        \caption{\small{Blow-Up approach.}}
    \end{subfigure}

\caption{\small{K-means clustering partitions: Comparison of decision regions on the simplex for three MNIST classes (digits 0, 4, and 7), using 200 sample point clouds per class. 
Algorithm \ref{alg: analysis} achieves Adjusted Rand Index $\mathrm{ARI}\approx 0.6$
and requires $\approx 405$ seconds to compute all GW-barycentric coordinates.
Algorithm \ref{alg: analysis blow up} yields $\mathrm{ARI}\approx0.55$,
requires $\approx 1293$ seconds to compute all coordinates, and has an average blow-up size ($M_b$): 661.8.  Clear clustering for digit class 0. In contrast, digit classes 4 and 7 exhibit
overlap, suggesting the presence of shared geometric shapes between them.}}
    \label{fig:decision_regions_simplex}
\end{figure}

\section{Conclusions}\label{sec: conclusions}

This paper addresses the problem of recovering barycentric coordinates in the context of finite metric spaces, or more generally, networks with weighted dissimilarities between nodes, embedded via the Gromov-Wasserstein (GW) distance, referred to as the \emph{GW analysis problem}. While the literature offers at least two established methods for computing GW barycenters, this work builds 
upon them and proposes two corresponding solutions to the GW analysis problem, each with its own algorithm. Moreover, we show that these two perspectives are not entirely independent: under suitable 
parameter choices,   they share common ground. Notably, both proposed solutions correspond to critical points of convex functions (Theorems \ref{thm: analysis} and \ref{prop: gradient method}). Our results pertain to recovery of optimal weights when a network lies exactly in the GW barycentric span. Nevertheless, our strategies are applicable, and can still be useful, when target objects are not exact barycenters. Obtaining theoretical guarantees to learn optimal weights for networks \emph{outside} the span is an interesting problem. In addition to Section \ref{sec: classif main}, where experiments are conducted on non-synthetic data, we present preliminary empirical work and a discussion of the challenges of this problem in Suppl. Mat.  \ref{sec supp: projection}. 

This paper focuses on the case of discrete metric measure spaces (or discrete networks), although the GW framework also applies in the continuum setting. Extensions of our fixed-point and gradient-based approaches to a broader context could be of interest. This paper does not establish a full equivalence between Karcher means and GW barycenters. In the case of classical OT, such an equivalence can be shown under certain assumptions \cite{werenski2022measure}. Developing analogous conditions in the GW setting remains an open problem for future research. 

There are also several experimental directions to explore. For instance, efficient implementations of the blow-up technique, so that, when applied to a collection of networks, they could identify a smallest equivalent representative of each network while preserving global alignment (enabling compressed storage, and faster computations, by avoiding repeated values). Also, further experiments in settings with alternative node dissimilarities (e.g., replacing the Euclidean distance with intrinsic notions such as geodesic distances  on manifolds) would also help solidify the applicability of the GW-BCM to real world problems.  Initial experimental results with  $p$-Fermat \cite{james-little2022balancing} and graph  distances in the GW-BCM framework are in Suppl. Mat. \ref{sec:NonEuc} and
\ref{sec: additional_classification_app}.

\section*{Acknowledgments} RDM and JMM were partially supported by DMS-2318894.  JMM was partially supported by DMS-2309519.  IVM was funded by NIH award GM130825 and ONR award N000142212505.

\appendix

\section{Background: The Barycentric Coding Model (BCM)}\label{sec: BCM}

Suppose $\mathrm B$ is a convex set in a vector space having finitely many \emph{extreme points} $\mathbf x^1,\dots,\mathbf x^S$, for some $S\in \N$. That is, $\mathrm B$ is the \emph{convex hull} of $\{\mathbf x^1,\dots,\mathbf x^S\}$: every point $\mathbf x\in \mathrm B$ can be expressed as
\begin{equation}\label{eq: BC}
    \mathbf x=\sum_{s\in [S]}\lambda_s\mathbf x^s, \qquad \text{ for some } \lambda=(\lambda_1,\dots,\lambda_S)\in \Delta_{S-1}.
\end{equation}
Reciprocally, since we are assuming that $\mathrm B$ is convex, we have that every convex combination of the points $\mathbf x^1,\dots,\mathbf x^S$ belongs to $\mathrm B$. Hence, we have a correlation between points in $\mathrm B$ and its \emph{barycentric coordinates} in $\Delta_{S-1}$ given by the relation \eqref{eq: BC}: $\mathrm B\ni\mathbf x \longleftrightarrow \lambda\in \Delta_{S-1}.$
Thus, given a vector of weights $\lambda\in \Delta_{S-1}$, one can \emph{synthesize} a point in $\mathrm B$ by:
\begin{equation}\label{eq: bary ad convex comb}
    \Delta_{S-1}\ni \lambda\longmapsto \mathbf x_\lambda:=\sum_{s\in [S]}\lambda_s\mathbf x^s\in \mathrm B.
\end{equation}
In Euclidean spaces, \eqref{eq: bary ad convex comb} coincides with solving the barycenter problem
\begin{equation}\label{eq: bary ad convex comb2}
    \Delta_{S-1}\ni \lambda\longmapsto \mathbf x_\lambda=\argmin_{\mathbf x\in \R^n}\sum_{s\in [S]}{\lambda_s}\|\mathbf x^s-\mathbf x\|^2\in \mathrm B.
\end{equation}
On the other hand, one can \emph{analyze} an element $\mathbf x \in \mathrm B$, by decomposing it into the set $\{\mathbf{x^{1}},\dots,\mathbf{x^{S}}\}$ in terms of its \emph{barycentric coordinates}. For example, if $\mathrm B\subset \mathbb R^n$, this corresponds to solving the following optimization problem:
\begin{equation}\label{eq: BCanalysis}
    \mathbf x\longmapsto\lambda_{\mathbf x}:=\argmin_{\lambda\in \Delta_{S-1}}\|\mathbf x -\sum_{s\in [S]} \lambda_s \mathbf x^s\|^2 .
\end{equation}
Moreover, one can extend \eqref{eq: BCanalysis} by considering $\mathbf{x} \in \mathbb{R}^n$ instead of restricting to $\mathbf{x} \in \mathrm{B}$. 
This amounts to projecting $\mathbf{x}$ onto $\mathrm{B}$ and retaining only the coordinates $\lambda_{\mathbf{x}}$ of the projected point\footnote{
Let $\Phi : \mathbb{R}^S \to \mathbb{R}^n$ be the matrix whose columns are the vectors $\mathbf{x}^1, \dots, \mathbf{x}^S$, and consider the quadratic function  $f(\lambda) := \|\mathbf x -\sum_{s\in [S]} \lambda_s \mathbf x^s\|^2=\|\mathbf{x} - \Phi \lambda\|^2$ for all $\lambda \in \Delta_{S-1}$. Since $f$ is continuous and the simplex is compact, a minimizer $\lambda_{\mathbf x}$ for problem \eqref{eq: BCanalysis} exists.  Define $\mathbf{z} := \Phi \lambda_{\mathbf{x}} \in \mathrm{B}$. By construction and since $\mathrm B=\Phi(\Delta_{S-1})$, we have
$ \|\mathbf x-\mathbf z\|=\min_{\lambda \in \Delta_{S-1}}\|\mathbf x-\Phi \lambda\| =\min_{\mathbf y\in \mathrm B}\|\mathbf x-\mathbf y\|$.
Thus, by convex analysis, this last property characterizes  $\mathbf{z}$ as the unique orthogonal projection of $\mathbf{x}$ onto $\mathrm{B}$ (even if $\lambda_{\mathbf x}$ is not unique).
}; consequently, the orthogonal projection of $\mathbf{x} \in \mathbb{R}^n$ onto the closed convex set $\mathrm{B}$ is given by $\sum_{s \in [S]} (\lambda_{\mathbf{x}})_s \, \mathbf{x}^s$ \cite{tyrrell1970convex}. Moreover, problem \eqref{eq: BCanalysis} can be rewritten as
\begin{equation}\label{eq: BCanalysis2}
    \mathbf x\longmapsto\lambda_{\mathbf x}:=\argmin_{\lambda\in \Delta_{S-1}}\|\mathbf x -\mathbf x_\lambda\|^2 ,
\end{equation}
where $\mathbf x_\lambda$ is a synthesized element as in \eqref{eq: bary ad convex comb} or \eqref{eq: bary ad convex comb2}.

With more generality, let us consider an ambient space $\mathrm G$ that is not necessarily convex, nor contained in a vector space, but endowed with a metric $dist(\cdot,\cdot)$. Consider also a \emph{divergence} $d_a(\cdot, \cdot)$, that is, a non-negative function defined on $\mathrm G\times \mathrm G$ satisfying the identity of indiscernibles (i.e., $d_a(\mathbf{x},\mathbf{y})=0$ if and only if $\mathbf{x}=\mathbf{y}$), but not necessarily symmetric and/or not necessarily satisfying the triangle inequality (for readability, it can be assumed $d_a(\cdot,\cdot)=dist(\cdot,\cdot)$).  Given finitely many
template points $\mathbf{x}^1, \dots, \mathbf{x}^S \in \mathrm G$, for some $S \in \mathbb{N}$, we can formulate the following \emph{synthesis} and \emph{analysis} problems, borrowing terminology from harmonic analysis. 
 The \emph{barycenter synthesis problem} is 
    \begin{gather}\label{eq: synthesis}
        \Delta_{S-1}\ni\lambda\longmapsto \mathbf{x}_\lambda\in\argmin_{\mathbf x\in \mathrm G}\sum_{s\in [S]}\lambda_s \, dist^2(\mathbf{x}^s,\mathbf{x}).
    \end{gather}
The \emph{barycenter analysis problem} is
    \begin{gather}\label{eq: analysis}
        \mathrm G \ni \mathbf{x}\longmapsto \lambda_\mathbf{x}\in\argmin_{\lambda\in \Delta_{S-1}}d^2(\mathbf{x},\mathbf{x}_\lambda),
    \end{gather}
    where each $\mathbf{x}_\lambda$ is any point defined in \eqref{eq: synthesis}. 
 Note that \eqref{eq: synthesis} and \eqref{eq: analysis} are the generalizations of problem \eqref{eq: bary ad convex comb2} and problem \eqref{eq: BCanalysis} (or \eqref{eq: BCanalysis2}), respectively.

\begin{definition}\label{def: bary general}
    The \emph{barycentric space} generated by the reference points $\{\mathbf{x}^s\}_{s\in [S]}$ is defined by
          $\mathrm{Bary}(\{\mathbf{x}^s\}_{s\in [S]}) := \{\mathbf{x}_\lambda \text{ solution of } \eqref{eq: synthesis}\mid \, \lambda\in \Delta_{S-1}\}.$
\end{definition}

\begin{remark}\label{rem: for any div}
    Notice that $\mathbf{x}\in \mathrm{Bary}(\{\mathbf{x}^s\}_{s\in [S]})$ if and only if $\mathbf{x}=\mathbf{x}_{\mathbf{\lambda}}$ for some $\mathbf{\lambda}\in \Delta_{S-1}$; equivalently, if and only if $d_a(\mathbf{x},\mathbf{x}_{\mathbf{\lambda}})=0$ for some $\mathbf{\lambda}\in \Delta_{S-1}$, \emph{for any} choice of divergence $d_a(\cdot,\cdot)$. However, in general, $\mathrm{Bary}(\{\mathbf{x}^s\}_{s\in [S]})$ is not necessarily a convex set, in particular is not necessarily the convex hull of the points $\{\mathbf{x}^s\}_{s\in [S]}$. 
\end{remark}

In this general context, the problem of estimating an unknown point $\mathbf{x}$ under the \emph{barycentric coding model}
(BCM) consists of assuming that $\mathbf{x}\in\mathrm{Bary}(\{\mathbf{x}^s\}_{s\in [S]})$ (\emph{model}), for a fixed set of points $\{\mathbf{x}^s\}_{s\in [S]}$, and estimating the unknown barycentric coordinates $\lambda\in\Delta_{S-1}$ such that a solution $\mathbf{x}_\lambda$ of the optimization problem \eqref{eq: synthesis} approximates the original point $\mathbf{x}$.  

\subsection{BCM: Classical OT vs. Euclidean framework}\label{sec: OT}

To illustrate the qualitative differences between geometric models, we consider the problem of interpolating between two weighted point clouds representing handwritten digits. This corresponds to visualizing the barycentric space generated by two template point clouds. Using the Point Cloud MNIST 2D dataset \cite{Garcia2023PointCloudMNIST2D}, we examine how the choice of geometry influences the resulting intermediate shapes. Fig. \ref{fig:OT} compares the interpolations obtained using Wasserstein geometry versus using standard Euclidean geometry. The reader can also compare Figure \ref{fig:OT} with Figure \ref{fig: GW bary space} in the main text. 

\begin{figure}[h!]
    \centering
    \includegraphics[width=0.8\linewidth]{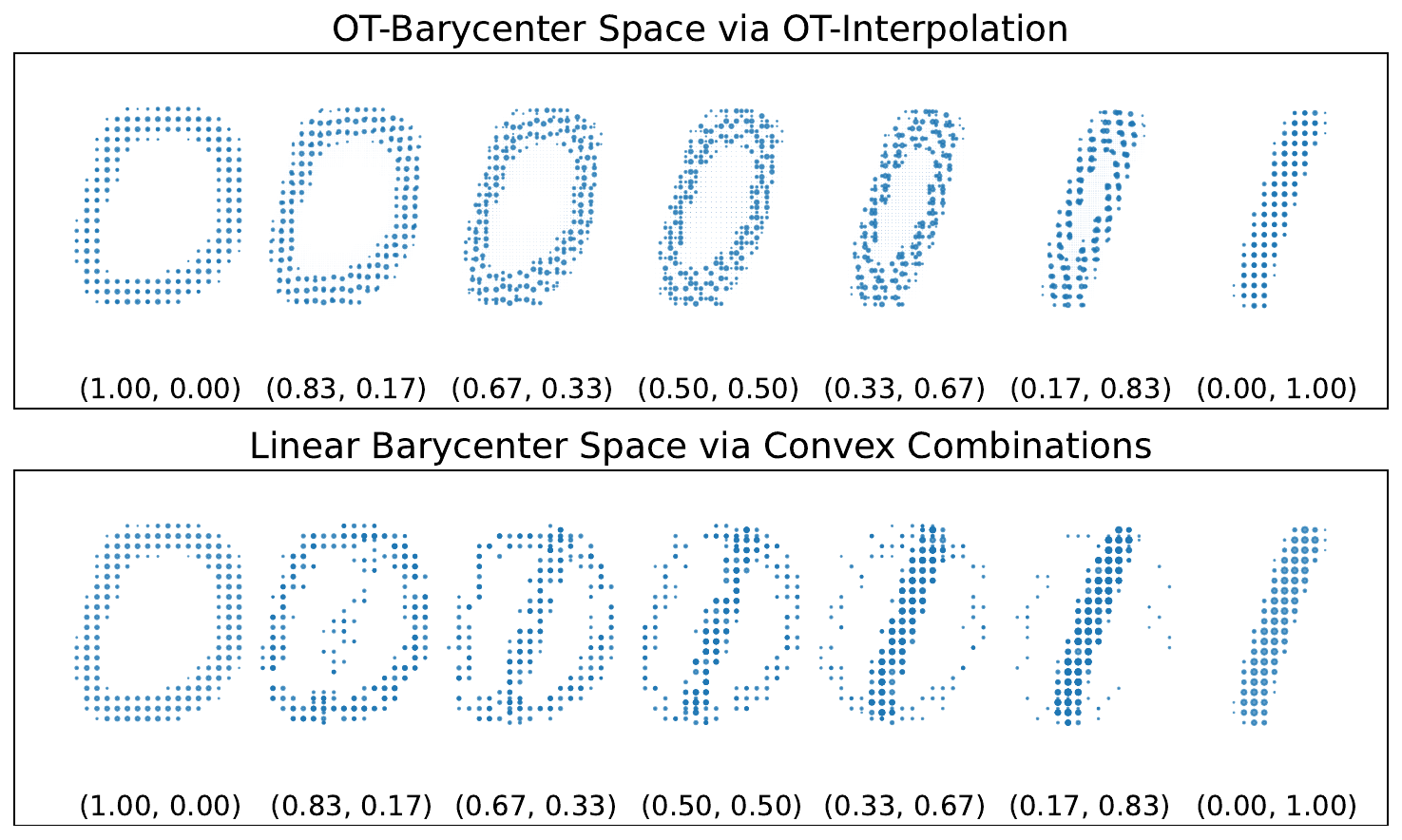}
    \caption{\small{Illustration of different barycenter spaces between two point clouds selected from the Point Cloud MNIST 2D dataset \cite{Garcia2023PointCloudMNIST2D}, shown  for various $t \in [0,1]$ with interpolation coordinates $(1-t, t)$ on the horizontal axis. Top:  OT interpolations (with respect to the 2-Wasserstein geometry). Bottom: Linear interpolations (with respect to the Euclidean geometry).}}
    \label{fig:OT}
\end{figure}

Analytically, consider the space of discrete probability measures with finite support on $\mathbb{R}$, equipped with the 2-Wasserstein metric $W_2(\cdot,\cdot)$. Let the template measures be $\delta_0$ and $\delta_1$, where $\delta_x$ denotes the Dirac measure centered at $x$. The corresponding barycenter space consists of the set $\{\delta_t :\, t \in [0,1]\}$, which is obtained by solving the synthesis problem \eqref{eq: syn} (in the main text) in this specific setting. In other words, each measure $\mu = \delta_t$ minimizes the functional $\mu \mapsto t\, W_2^2(\delta_0, \mu) + (1 - t)\, W_2^2(\delta_1, \mu)$.  Hence, we can interpret the coordinates of the measure $\delta_t$ with respect to the templates $\delta_0$ and $\delta_1$ as the vector $(t, 1 - t)$ in $\Delta_1$. In this simple example, barycenters correspond to Dirac masses supported on convex combinations of the points $0$ and $1$--that is, Euclidean barycenters. In more general settings, however, measures or distributions have the ability to represent more complex data than individual points in space. 
For instance, interpolating between two templates can be understood as constructing the barycentric space generated by those reference templates. Specifically, if $\mu_0$ and $\mu_1$ are the probability distributions corresponding to two templates, and $\pi$ is an OT plan between them in the 2-Wasserstein sense, then the interpolated measure $\mu_t$ defined as the pushforward of $\pi$ under the map $(x, y) \mapsto (1 - t)x + t y$, that is, $\mu_t := ((1 - t)x + t y)_{\#} \pi$ for $t \in [0, 1]$, coincides with the {Wasserstein barycenter} of $\mu_0$ and $\mu_1$ with respective weights $(1 - t)$ and $t$:
$\mu_t = \argmin_{\mu} \left\{ (1 - t) W_2^2(\mu, \mu_0) + t W_2^2(\mu, \mu_1) \right\}$.

Finally, we notice that one of the main differences between the 2-Wasserstein distance and the GW distance considered in this work, lies in how they support barycenter constructions. In the $W_2$ setting, barycenters between two or more measures can be constructed via displacement interpolation using OT maps under regularity assumptions \cite{mccann1997convexity, agueh2011barycenters}. This interpolation aligns mass in a common ambient space and is underpinned by the pseudo-Riemannian structure of Wasserstein space \cite{zemel2019frechet}.
In contrast, GW distance compares measures on different metric spaces by optimizing over structural correspondences rather than spatial displacements. While geodesics between two metric measure spaces exist in GW space and can yield meaningful barycenters (see Figure \ref{fig: GW bary space} in the main text), there is no analogous displacement interpolation for more than two templates. This reflects the absence of a pseudo-Riemannian framework (in the classical sense \cite{otto2001geometry, ambrosio2008gradient}), limiting the direct geometric interpretation of GW barycenters.

\section{Generalities on the Gromov-Wasserstein Problem}\label{app: generalities}

\subsection{GW Distance}\label{app: GW dist}

\begin{proposition}\label{prop: convex}   
The  minimization problem \eqref{eq: rest_gw}  is convex (respectively, strictly convex) if the input matrices $\mathbf{X}$ and $\mathbf{Y}$ are symmetric positive definite (respectively, symmetric positive definite).
\end{proposition} 

\begin{proof}
On the one hand, the set of couplings $\Pi(\p,\q)$ is convex.  On the other hand, we first note that in general, by using the constraint $\pi\in\Pi(\p,\q)$, the cost function in the minimization problem \eqref{eq: rest_gw} can be rewritten as
    \begin{align*}
        &\sum_{i\in[N],j\in[M]}\sum_{k\in[N],l\in[M]}|\mathbf X_{ik}- \mathbf Y_{jl}|^2\pi_{ij}\pi_{kl}\\
        &= \sum_{i,k\in[N]}(\mathbf X_{ik})^2\p_i\p_k
        -2\underbrace{\sum_{i,k\in[N],j,l\in[M]}\mathbf X_{ik} \mathbf Y_{jl}\pi_{ij}\pi_{kl}}_{\mathrm{tr}(\pi^T\mathbf X^T\pi \mathbf Y)}+
        \sum_{j,l\in[M]}(\mathbf Y_{jl})^2\q_j\q_l,
    \end{align*}
    where only the second term depends on the coupling $\pi$.
    It is easy to check that, for each $i_0\in[N], j_0\in[M]$,
    \begin{equation*}
        \frac{\partial}{\partial \pi_{i_0j_0}}\mathrm{tr}(\pi^T\mathbf X^T\pi \mathbf Y)=2 \, \sum_{k\in[N],l\in[M]}\mathbf X_{i_0k}\pi_{kl} \mathbf Y_{j_0l}, 
    \end{equation*}
    and 
    \begin{equation*}
        \frac{\partial^2}{\partial \pi_{k_0l_0}\partial \pi_{i_0j_0}}\mathrm{tr}(\pi^T\mathbf X^T\pi \mathbf Y)=2\mathbf X_{i_0k_0}\mathbf Y_{j_0l_0}.
    \end{equation*}
    Thus, the first-order derivative of the functional $\pi\mapsto \mathrm{tr}(\pi^T\mathbf X^T\pi \mathbf Y)$ with respect to $\pi$ is 
    $ 2 \mathbf{X} \pi \mathbf{Y}$
    and the Hessian is $2 (\mathbf{X} \otimes \mathbf{Y}),$
    where \(\otimes\) denotes the Kronecker product.

   If both $\mathbf{X}$ and $\mathbf{Y}$ are symmetric positive 
    semidefinite matrices, then so it is $\mathbf{X} \otimes \mathbf{Y}$ because the eigenvalues of the $MN\times MN$ matrix $ \mathbf{X} \otimes \mathbf{Y}$ are given by the products of the eigenvalues of $\mathbf{X}$ and $\mathbf{Y}$. Thus, in this case the minimization problem \eqref{eq: rest_gw} is convex (it is strictly convex if the input matrices $\mathbf{X}$ and $\mathbf{Y}$ are symmetric positive definite). 
\end{proof}

\begin{remark}
See also \cite{chowdhury2020gromov}, where the above result is proven in the case where $\mathbf{X}$ and $\mathbf Y $ are symmetric positive definite matrices, using the Cholesky decomposition of the given matrices.    
\end{remark}

\subsection{GW Synthesis Problem -- Fixed-Point Iteration Scheme}\label{app: gw fp synth}

In this section we consider the GW synthesis problem \eqref{eq: solomon_barycenter}. First, for completeness, we provide a proof of Proposition \ref{remark: bary_fix_pi} (see also \cite[Prop. 3]{peyre2016gromov}), followed by several remarks. 

\begin{proof}[Proof of Proposition \ref{remark: bary_fix_pi}]
Notice that, as in the proof of Proposition \ref{prop: convex}, if we use the marginal constraints $\pi^s\in \Pi(\p^s,\q)$, that is, 
\begin{equation*}
  \sum_{j\in[M]}{\pi_{ij}^s}=\p_i^s \qquad  \forall i\in [N^s] \qquad  \text{ and } \qquad \sum_{i\in [N^s]} \pi_{ij}^s=\q_j \qquad  \forall j\in [M],  
\end{equation*}
then the objective function $J_\lambda$ 
can be rewritten as
\begin{align}
   0\leq J_\lambda (\mathbf Y) &=
       \sum_{s\in [S]}\lambda_s    \sum_{i\in[N^s],j\in[M]}\sum_{k\in[N^s],l\in[M]}|\mathbf X_{ik}^s- \mathbf Y_{jl}|^2\pi_{ij}^s\pi_{kl}^s \\
       &=
        \sum_{s\in [S]}
        \sum_{i,k\in[N]}\lambda_s(\mathbf X_{ik}^s)^2\p^s_i\p^s_k+
        \sum_{j,l\in[M]}( \mathbf Y_{jl})^2\q_j\q_l \notag\\
        &\qquad -2\sum_{s\in [S]}\sum_{i\in[N^s],j\in[M]}\sum_{k\in[N^s],l\in[M]}\lambda_s\mathbf X_{ik}^s \mathbf Y_{jl}\pi_{ij}^s\pi_{kl}^s. \label{eq: calc}
\end{align}
Thus, $J_\lambda$ is continuous, convex, and nonnegative, (i.e., $J_\lambda (\mathbf Y)\geq 0$ for all $\mathbf Y\in \R^{M\times M}$).
It is easy to check that if $\{\pi^s\}_{s\in [S]}$ are fixed,  minimizing $J_\lambda$ over $ \mathbf Y\in\R^{M\times M}$ yields the gradient expression:
\begin{align}\label{eq: gradient with fix pi}
    \frac{\partial \, J_\lambda}{\partial  \mathbf Y_{j_0\, l_0}}=2 \, \q_{j_0}\q_{l_0} \mathbf Y_{j_0 \, l_0}-2\sum_{s\in [S]}\lambda_s\sum_{i,k\in[N^s]}\mathbf X_{ik}^s\pi_{ij_0}^s\pi_{kl_0}^s.
\end{align}
The Hessian of 
$J_\lambda$ is a $M^2\times M^2$ diagonal matrix with diagonal entries given by:
\begin{equation*}
    \frac{\partial^2 \, J_\lambda}{\partial  \mathbf Y_{j_0\, l_0}^2}= 2 \, \q_{j_0}\q_{l_0}, \qquad \forall l_0,j_0\in [M].    
\end{equation*}
In general, since $\q_j\geq 0$ for all $j$, the Hessian is positive semidefinite, and therefore, by convexity and the fact that the objective function is nonnegative, a minimizer exists.
Moreover, if $\q_j>0$ for all $j$, the Hessian is positive definite, implying that the problem admits a unique global minimum. The entries of the minimum
$\mathbf Y^*\in \mathbb R^{M\times M}$ are characterized by setting the gradient expression \eqref{eq: gradient with fix pi} to zero, yielding the closed-form solution \eqref{eq: bary_fix_pi}.
\end{proof}

\begin{remark}
Notice that if $\{\mathbf X^s\}_{s\in [S]}$ are symmetric, then formula \eqref{eq: bary_fix_pi} defines a symmetric matrix $\mathbf Y^*$.
Similarly, if the matrices $\{\mathbf X^s\}_{s\in [S]}$ have non-negative entries or are positive semidefinite, then $\mathbf Y^*$ will also have non-negative entries or be positive semidefinite \cite[Prop. 3]{peyre2016gromov}. 
However, in general, $\mathbf Y^*$ does not have zero-trace (or zero-diagonal), even if the template matrices $\{\mathbf X^s\}_{s\in[S]}$ have zero trace. Therefore, even if the templates are distance matrices, $\mathbf Y^*$ given by \eqref{eq: bary_fix_pi} does not necessarily define a distance matrix. 
\end{remark}

\begin{proposition}[Existence of a Solution for the GW Synthesis Problem \eqref{eq: solomon_barycenter}]\label{remark: existence min synth}
Consider a set of finite templates $\{( \mathbf X^s, \p^s)\in \R^{N^s\times N^s} \times \mathcal{P}_{N^s}\}_{s\in[S]}$. Fix $\lambda\in \Delta_{S-1}$, a size $M\in \N$, and a probability vector $\q\in \mathcal{P}_M$. Consider the functional $H:\R^{M\times M}\times \Pi(\p^1,\q)\times\dots \times\Pi(\p^S,\q)\to \R$ given by
\begin{equation*}
    H(\mathbf Y,\pi^1,\dots, \pi^S):=\sum_{s\in[S]}\lambda_s    \sum_{i\in[N^s],j\in[M]}\sum_{k\in[N^s],l\in[M]}|\mathbf X_{ik}^s- \mathbf Y_{jl}|^2\pi_{ij}^s\pi_{kl}^s .
\end{equation*}
If we assume the extra hypothesis that all fixed matrices $\mathbf X^1,\dots, \mathbf X^S$ and the variable matrix $\mathbf Y$ have non-negative entries, then there exists a minimizer $(\mathbf Y^*,(\pi^1)^*,\dots,(\pi^S)^*)$ of $H$ and problem \eqref{eq: solomon_barycenter} has at least one solution $\mathbf Y_\lambda$.
\end{proposition}  

Note that $\Pi(\p^1,\q)\times\dots \times\Pi(\p^S,\q)$ is compact, since it is the product of compact sets, and  
$H$ is continuous and non-negative.
However, $H$ is not convex and the variable  $\mathbf Y$ does not belong to a compact set. For this reason, we included the additional hypothesis in the above statement.

\begin{proof}[Proof of Proposition \ref{remark: existence min synth}]

Since $H$ is non-negative,
let
\begin{equation*}
    a^*:=\inf\{    H(Y,\pi^1,\dots, \pi^S)\mid \, \mathbf Y\in \R^{M\times M}_{\geq 0}, \, \pi^s\in \Pi(\p^s,\q) \, \forall s\in [S]\},
\end{equation*}
where $\R^{M\times M}_{\geq 0}$ denotes the subset of $\R^{M\times M}$ of all $M\times M$ matrices with non-negative entries.
Notice that since the vectors $\p^1,\dots,\p^S$, and $\q$ are fixed,  we can rewrite the functional $H$ as 
\begin{align*}
    H(\mathbf Y,\pi^1,\dots, \pi^S)&=
       \sum_{s\in [S]}\lambda_s    \sum_{i\in[N^s],j\in[M]}\sum_{k\in[N^s],l\in[M]}|\mathbf X_{ik}^s- \mathbf Y_{jl}|^2\pi_{ij}^s\pi_{kl}^s \\
       &=
        \sum_{s\in [S]}
        \sum_{i,k\in[N]}\lambda_s(\mathbf X_{ik}^s)^2\p^s_i\p^s_k+
        \sum_{j,l\in[M]}( \mathbf Y_{jl})^2\q_j\q_l\\
        &\qquad -2\sum_{s\in [S]}\sum_{i\in[N^s],j\in[M]}\sum_{k\in[N^s],l\in[M]}\lambda_s\mathbf X_{ik}^s \mathbf Y_{jl}\pi_{ij}^s\pi_{kl}^s
\end{align*}
(as in \eqref{eq: calc} in the previous proof). Thus, when working under the hypothesis that all the matrices have non-negative entries, we have
\begin{align}
    H(\mathbf Y,\pi^1,\dots, \pi^S)&\geq \sum_{j,l\in[M]}( \mathbf Y_{jl})^2\q_j\q_l-2\sum_{s\in [S]}\sum_{i\in[N^s],j\in[M]}\sum_{k\in[N^s],l\in[M]}\lambda_s\mathbf X_{ik}^s \mathbf Y_{jl}\pi_{ij}^s\pi_{kl}^s \notag \\
    &\geq \sum_{j,l\in[M]}( \mathbf Y_{jl})^2\q_j\q_l-\left(2\sum_{s\in [S]}\sum_{i,k\in[N^s]}\lambda_s\mathbf X_{ik}^s \right)\sum_{j,l\in[M]}\mathbf Y_{jl} \notag\\
    &\geq k_1 \|\mathbf Y\|_{\mathrm{Frob}}^2-k_2\|\mathbf Y\|_{\mathrm{Frob}}\label{eq: aux coer}
\end{align}
for some positive constants $k_1,k_2$ (which do not depend on $\pi^s$), 
where we used that all the norms in $\R^{M\times M}$ are equivalent, and previously we used that $0\leq \pi_{ij}^s\leq 1$. Thus, \eqref{eq: aux coer} shows that if $\|\mathbf Y\|_{\mathrm{Frob}}\to \infty$, then
$H(\mathbf Y,\pi^1,\dots, \pi^S)\to \infty$.
This implies that for all $a> a^*$, there exists $R>0$ such that
\begin{equation*}
    a^*<a\leq H(\mathbf Y, \pi^1,\dots,\pi^s) \qquad \forall \mathbf Y\in \R^{M\times M}_{\geq 0}\setminus B_R,
\end{equation*}
where 
\begin{equation*}
    B_R:=\left\{\mathbf Y\in\R^{M\times M}_{\geq 0} \mid \, \|\mathbf Y\|_{\mathrm{Frob}}\leq R \right\}.
\end{equation*}
Thus, 
\begin{align*}
    a^* 
    &=\inf\{    H(Y,\pi^1,\dots, \pi^S)\mid \, \mathbf Y\in B_R, \, \pi^s\in \Pi(\p^s,\q) \, \forall s\in [S]\}\\
    &=\min\{    H(Y,\pi^1,\dots, \pi^S)\mid \, \mathbf Y\in B_R, \, \pi^s\in \Pi(\p^s,\q) \, \forall s\in [S]\},
\end{align*}
where the last expression holds since $H$ is continuous and we are minimizing it over a compact set. In particular,
\begin{align*}
    &\inf\{    H(Y,\pi^1,\dots, \pi^S)\mid \, \mathbf Y\in \R^{M\times M}_{\geq 0}, \, \pi^s\in \Pi(\p^s,\q) \, \forall s\in [S]\}\\
    &=    \min\{    H(Y,\pi^1,\dots, \pi^S)\mid \, \mathbf Y\in \R^{M\times M}_{\geq 0}, \, \pi^s\in \Pi(\p^s,\q) \, \forall s\in [S]\}.
 \end{align*}
Finally, from the general fact that given a function $f$, if $f(x_0,y_0)=\min_{x,y}f(x,y)$ for some point $(x_0,y_0)$, then $\min_{x,y}f(x,y)=\min_x\min_yf(x,y)$, we obtain  
  \begin{align*}  
   & \min\{    H(Y,\pi^1,\dots, \pi^S)\mid \, \mathbf Y\in \R^{M\times M}_{\geq 0}, \, \pi^s\in \Pi(\p^s,\q) \, \forall s\in [S]\}\\
    &=\min_{\mathbf Y\in \R^{M\times M}_{\geq 0}}\min_{\{\pi^s\in \Pi(\pi^s,\q)\}_{s=1}^S} H(Y,\pi^1,\dots, \pi^S)\\
    &=\min_{\mathbf Y\in \R^{M\times M}_{\geq 0}}\sum_{s\in[S]}\lambda_s GW(( \mathbf X^s, \p^s),( \mathbf Y, \q))^2,
\end{align*}
that is, the last minimum is well-defined, meaning that there exists  at least one solution $\mathbf Y_\lambda \in \R^{M\times M}_{\geq 0}$ of problem \eqref{eq: solomon_barycenter} under the assumption that the template matrices have non-negative entries. 
\end{proof}

\begin{example}[Illustrative Example Related to Remark \ref{remark: does not depend on the plan}]\label{example: simple}
To illustrate the point made in Remark \ref{remark: does not depend on the plan} more clearly, let us analyze the following simple example. For a single template $(\mathbf X,\p)\in \R^{N\times N}\times \mathcal{P}_N$ (i.e., $S=1$), consider the particular case of problem \eqref{eq: solomon_barycenter} for $M=N$ and $\q=\p$, that is, 
\begin{equation*}
    \min_{ \mathbf Y\in \R^{N\times N}} GW((\mathbf X,\p),(\mathbf Y, \p))^2.
\end{equation*} 
Clearly, $\mathbf X\in  \argmin_{ \mathbf Y\in \R^{N\times N}} GW((\mathbf X,\p),(\mathbf Y, \p))^2$.  Then, Theorem \ref{prop: formula_bary} implies that
\begin{equation}\label{eq: x = x}
    \mathbf X = \frac{1}{\p\p^T}\odot(\pi^*)^T\mathbf X\pi^* \qquad \forall \pi^*\in \Pi(\p,\p) \text{ optimal for } GW((\mathbf X,\p),(\mathbf X, \p)).
\end{equation}
To reinforce \eqref{eq: x = x}, let us verify it by exhaustion in the case $N=M=2$, $\mathbf X$ anti-diagonal matrix and $\p$ uniform. That is, consider $\mathbf X=\begin{pmatrix}
    0&c\\
    c&0
\end{pmatrix}$ for some $c>0$ and $\p=(.5,.5)$. Then 
\begin{equation*}
    \pi_1=\begin{pmatrix}
    .5&0\\
    0&.5
\end{pmatrix}\qquad \text{ and } \qquad 
    \pi_2=\begin{pmatrix}
    0&.5\\
    .5&0
\end{pmatrix}
\end{equation*}
are all the possible \emph{optimal} plans for $ GW((\mathbf X,\p),(\mathbf X, \p))$.  
It is easy to check that in fact
\begin{equation*}
    {\mathbf X}_1^*:=\frac{1}{\p\p^T}\odot\pi_1^T\mathbf X\pi_1=\mathbf X \qquad \text{ and } \qquad {\mathbf X}_2^*:=\frac{1}{\p\p^T}\odot\pi_2^T\mathbf X\pi_2=\mathbf X^T=\mathbf X,
\end{equation*}    
showing that ${\mathbf X}_1^*={\mathbf X}_2^*$ and so further supporting that \eqref{eq: bary_fix_pi} does not depend on the choice of the GW OT plans.
\end{example}

\subsection{Weak Isomorphisms}\label{app: weak}\,

Let $\mathbb X, \mathbb Y$ be two mm-spaces. By \cite[Thm. 5.1]{memoli2011gromov}  
 $GW(\mathbb X, \mathbb Y)=0$ if and only if $\mathbb X$ and $\mathbb Y$ are strongly isomorphic. In more generality, if $\mathbb X, \mathbb Y$ are two networks, then $GW(\mathbb X, \mathbb Y)=0$ if and only if $\mathbb X$ and $\mathbb Y$ are weakly isomorphic \cite[Thms. 2.3 and 2.4]{chowdhury2017distances}. 
 As a consequence,  two mm-spaces are strongly isomorphic if and only if they are weakly isomorphic. To gain some intuition we provide a proof for the finite case. The proof of the general case follows the same ideas but requires some technical results such as \cite[Lem. 2.2 or Lem. 10.4]{memoli2011gromov}.

\begin{proposition}\label{prop: weak+ metric structure implies strong}
    If two finite metric measure spaces are weakly isomorphic, then they are strongly isomorphic.
\end{proposition}
\begin{proof}
    Let $\mathbb X=(X,d_X,\mu_X)$, $\mathbb Y=(Y,d_Y,\mu_Y)$
    be two metric measure spaces over finite sets
    $X=\{x_i\}_{i=1}^N$ and $Y=\{y_j\}_{j=1}^M$
    such that 
    $\mathbb X\sim^w \mathbb Y$. Assume, without loss of generality that the support of the measures $\mu_X$ and $\mu_Y$ are exactly $X$ and $Y$, respectively. Let $(Z,\mu_Z)$ be a `pivoting' finite measure space and maps $\varphi_X:Z\to X$, $\varphi_Y:Z\to Y$ as in Definition \ref{def: weak iso}. In fact, since the spaces are finite, the second property in the definition of weak isomorphism (Definition \ref{def: weak iso}) takes the form    
    \begin{equation}\label{eq: pull back equality}        d_X(\varphi_X(z),\varphi_X(z'))=d_Y(\varphi_Y(z),\varphi_Y(z')) \qquad \forall z,z'\in Z.
    \end{equation}
    Moreover, $\varphi_X$ and $\varphi_Y$ are surjective functions onto $X$ and $Y$, respectively. This is a consequence of  the first property in the definition of weak isomorphism (Definition \ref{def: weak iso}), because for each $x\in X$ and each $y\in Y$,
    \begin{align*}
        &0\not=\mu_X(\{x\})=(\varphi_X)_\#\mu_Z(\{x\})=\mu_Z\left(\{z\mid \, \varphi_X(z)=x\}\right),\\
         &0\not=\mu_Y(\{y\})=(\varphi_Y)_\#\mu_Z(\{y\})=\mu_Z\left(\{z\mid \, \varphi_Y(z)=y\}\right),
    \end{align*}
    that is, for each $x\in X$ there exists at least one point $z\in Z$ such that $\varphi_X(z)=x$; similarly, for each $y\in Y$ there exists at least one point $z\in Z$ such that $\varphi_Y(z)=y$. 

    Let us define the function $\varphi:X\to Y$ as follows: for  $x\in X$, set $\varphi(x):=\varphi_Y(z_x)$, where $z_x$ is some point in $Z$ such that $\varphi_X(z_x)=x$. 
    Let us show the following properties of $\varphi$:
    \begin{enumerate}[leftmargin=*]
        \item It is a well-defined function: It does not depend on the representative $z_x$, that is, if $z_x, z_x'\in Z$ are such that $\varphi_X(z_x)=x=\varphi_X(z_x')$, then 
    \begin{equation*}
        d_Y(\varphi_Y(z_x),\varphi_Y(z_x'))=d_X(\varphi_X(z_x),\varphi_X(z_x'))=d_X(x,x)=0,
    \end{equation*}
    thus $\varphi_Y(z_x)=\varphi_Y(z_x')$.
    Notice that we have applied the identity of indiscernibles for distances. 
    \item It is an isometry: For every $x,x'\in X$ we have, 
    \begin{equation*}
        d_Y(\varphi(x),\varphi(x'))=d_Y(\varphi_Y(z_x),\varphi_Y(z_{x'}))=d_X(\varphi_X(z_x),\varphi_X(z_{x'}))=d_X(x,x').
    \end{equation*}
    \item It is a bijection: In analogy to the definition of $\varphi$, we can define a new map $\phi:Y\to X$ by $\phi(y):=\varphi_X(z_y)$ for any $z_y\in Z$ such that $\varphi_Y(z_y)=y$. It is a well-defined function by the same argument as in the first item. Moreover, $\phi=\varphi^{-1}$, since, by definition 
    \begin{align*}
        &\varphi(\phi(y))=\varphi(\varphi_X(z_y))=\varphi_Y(z_y)=y, \qquad \forall y\in Y,\\
        &\phi(\varphi(x))=\phi(\varphi_Y(z_x))=\varphi_X(z_x)=x, \qquad \forall x\in X.
    \end{align*}
    \item $\varphi_\#\mu_X=\mu_Y$: For every $y\in Y$,
    \begin{align*}
        \varphi_\#\mu_X(\{y\})&=\mu_X(\{x\mid \, \varphi(x)=y\})=(\varphi_X)_\#\mu_Z(\{x=\phi(y)\})\\
        &=\mu_Z(\{z\mid \, \varphi_X(z)=\phi(y) \})=\mu_Z(\{z\mid \, \varphi(\varphi_X(z))=y \})\\
        &=\mu_Z(\{z\mid \, \varphi_Y(z)=y\})=(\varphi_Y)_\#\mu_Z(\{y\})=\mu_Y(\{y\}).
    \end{align*}
    \end{enumerate}

\end{proof}

\begin{example}[Weak Isomorphism]\label{example: weakiso}
    Figure \ref{fig: weakiso_2} is an illustration of a weak isomorphism between two networks $\mathbb X$ and $\mathbb Y$.  
    The network $\mathbb X$ consists of a measure space $(X,\mu_X)$ with two points having masses $0.6$ and $0.4$, and $\omega_X$ represented by the $2\times 2$ matrix 
    \begin{equation*}
        \mathbf{X}=\begin{pmatrix}
        3 & 2\\
        2 & 0
    \end{pmatrix}.
    \end{equation*}
    The network $\mathbb Y$ is given by the measure space $(Y,\mu_Y)$ consisting of three points with mass $0.3$, $0.3$, and $0.4$, and $\omega_Y$ represented by the $3\times 3$ matrix 
    \begin{equation*}
        \mathbf{Y}=\begin{pmatrix}
        3 & 3 & 2\\
        3 & 3 & 2\\
        2 & 2 & 0\\
    \end{pmatrix}.
    \end{equation*} 
    Similarly as in Figure \ref{fig: weakiso}, the pivoting space $(Z,\mu_Z)$ in Definition \ref{def: weak iso} can be taken as $(Y,\mu_Y)$, with $\varphi_X=\varphi$ (red assignment) and $\varphi_Y$ the identity map, getting $\mathbb X\sim^w\mathbb Y$.

As an informal interpretation, under weak isomorphism two nodes
$y, y'\in Y$ are  `\emph{the same}' if they have the same `internal connectivity', i.e., $\omega_Y(y,y)=\omega_Y(y,y')=\omega_X(y',y)=\omega_Y(y',y')$, and the same `external connectivity', i.e., $\omega_Y(y,a)=\omega_Y(y',a)$ and $\omega_Y(a,y)=\omega_Y(a,y')$ for all $a\in Y$ \cite{chowdhury2019gromov}.
In our example, the two nodes in the space $Y$ with mass $0.3$ are `\emph{the same}'.

\begin{figure}[ht!]
    \centering
    \includegraphics[width=0.4\linewidth]{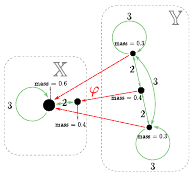}
    \caption{\small{Illustration of two weakly isomorphic spaces $\mathbb X\sim^w\mathbb Y$.}}
    \label{fig: weakiso_2}
\end{figure}

\end{example}

\begin{remark}[Log Map]
By using the blow-up technique (i.e., by using Algorithm \ref{alg: blow up} with an optimal coupling for $GW(\mathbb X, \mathbb Y)$), the \emph{logarithmic map} $Log_{[\mathbb{Y}]}:\mathcal{GM}\longrightarrow \mathrm{Tan}_{[\mathbb{Y}]}$, that is, the local inverse of the exponential map \eqref{eq: exp}, can be written as: $ Log_{[\mathbb Y]}([\mathbb X]):=[\mathbf X_b-\mathbf Y_b]$. We refer the reader to \cite[Sec. 3.4]{chowdhury2020gromov}. 
\end{remark}

\begin{remark}[Geodesics]\label{rem: geod}
 By using the blow-up technique, the curve  (for $t\in [0,1]$)
 $$\gamma(t):=[(X_b\times Y_b, (1-t)\omega_{X_b}+t\omega_{Y_b}, \pi_b)] \text{ or, equivalently, } \gamma(t):= [((1-t)\mathbf{X}_b+t\mathbf Y_{b}, \q_b)],$$
 is a GW \emph{geodesic} between $[\mathbb X]$ and $[\mathbb Y]$,
    in the sense that 
    $GW(\gamma(t), \gamma(s))=|t-s|GW(\mathbb X, \mathbb Y )$ for all $0\leq s,t\leq 1$.
    In particular,  ${\omega_{X_b}}-{\omega_{Y_b}}$ (represented by the matrix $\mathbf X_b-\mathbf Y_b$) can be viewed as its \emph{velocity}. We refer the reader to \cite{sturm2023space, chowdhury2020gromov}.
    As a consequence, the GW barycenter between two references $[\mathbb X]$, $[\mathbb Y]$ with uniform weights $\lambda=(1/2,1/2)$ is given by the network class $\gamma(1/2)= [(\frac{1}{2}\mathbf{X}_b+\frac{1}{2}\mathbf Y_{b}, \q_b)]$. Indeed, by using the triangle inequality, for every network class $[\mathbb Z]$ we have: 
    \begin{align*}
        \frac{1}{2}GW^2(\mathbb X, \gamma(1/2))+\frac{1}{2}GW^2(\mathbb Y, \gamma(1/2))&=\frac{1}{2}GW^2(\gamma(0), \gamma(1/2))+\frac{1}{2}GW^2(\gamma(1), \gamma(1/2))\\
        &=\frac{1}{2}\frac{1}{4}GW^2(\mathbb X, \mathbb Y )+\frac{1}{2}\frac{1}{4}GW^2(\mathbb X, \mathbb Y )\\
        &=\frac{1}{4}GW^2(\mathbb X, \mathbb Y )\\
        &\leq \frac{1}{4}\left(GW(\mathbb X, \mathbb Z )+GW(\mathbb X, \mathbb Z )\right)^2\\
        &=\left(\frac{1}{2}GW(\mathbb X, \mathbb Z )+\frac{1}{2}GW(\mathbb X, \mathbb Z )\right)^2\\
        &\leq \frac{1}{2}GW^2(\mathbb X, \mathbb Z )+\frac{1}{2}GW^2(\mathbb X, \mathbb Z ) .
    \end{align*}
    In general, $\gamma(t)$ is a GW barycenter between $[\mathbb X]$ and $[\mathbb Y]$ with weights $\lambda=(1-t,t)$ because for every network class $[\mathbb Z]$ we have:
    \begin{align*}
    (1 - t)\, GW^2(\mathbb{X}, \gamma(t)) + t\, GW^2(\mathbb{Y}, \gamma(t)) 
    &= (1 - t)\, t^2\, GW^2(\mathbb{X}, \mathbb{Y}) + t\, (1 - t)^2\, GW^2(\mathbb{X}, \mathbb{Y}) \\
    &= t(1 - t)\, GW^2(\mathbb{X}, \mathbb{Y}) \\
    &\leq t(1 - t)\, \left(GW(\mathbb{X}, \mathbb{Z}) +  GW(\mathbb{Y}, \mathbb{Z})\right)^2\\
    &\leq (1 - t)\, GW^2(\mathbb{X}, \mathbb{Z}) + t\, GW^2(\mathbb{Y}, \mathbb{Z}),
\end{align*}
where in the last step we used the inequality $t(1-t)(a+b)^2\leq (1-t)a^2+tb^2$ (which arises from expanding squares or completing the square $((1-t)a-tb)^2\geq 0$).
However, this procedure cannot be adapted when considering more than 2 templates. 
\end{remark}    

\begin{remark}[First-Order Optimality in GW Space with the Formal Riemannian-Like Structure Defined in \cite{chowdhury2020gromov} and Used in This Work]\label{rem: first order calc}
    Let us show that a  minimizer $[\mathbb Y^*]$ of a differentiable functional in GW space, $G:\mathcal{GW}\to \R$,  must satisfy the \emph{first-order optimality condition}, that is, $\nabla G([\mathbb Y^*])=0$.
    First, by assumption we have that $G([\mathbb Y^*])\leq G([\mathbb Y])$ for all $[\mathbb Y]\in \mathcal{GW}$. In particular, $G([\mathbb Y^*])\leq G(Exp_{[\mathbb Y^*]}(tv))$ for all $[v]\in \mathrm{Tan}_{[\mathbb Y^*]}$. Thus,
    \begin{equation*}
        \langle [v], \nabla G ([\mathbb Y^*])\rangle_{\mathrm{Tan}_{[\mathbb Y^*]}}=D_vG([\mathbb Y])=\lim_{t\searrow 0}\frac{G(Exp_{[\mathbb Y^*]}(tv))-G([\mathbb Y^*])}{t}\geq 0, \quad \forall [v]\in \mathrm{Tan}_{[\mathbb Y^*]}.
    \end{equation*}
    In particular, for the direction $[v]=-\nabla G ([\mathbb Y^*])$ (which exists because tangent vectors are modeled as elements of $L^2$-spaces) we have 
    \begin{equation*}
     0 \leq \langle -\nabla G ([\mathbb Y^*]), \nabla G ([\mathbb Y^*])\rangle_{\mathrm{Tan}_{[\mathbb Y^*]}}= -\|\nabla G ([\mathbb Y^*])\|_{\mathrm{Tan}_{[\mathbb Y^*]}}^2\leq 0,
    \end{equation*}
    which is only possible if $\nabla G ([\mathbb Y^*])=0$.

Thus, the authors in \cite{chowdhury2020gromov} develop a \emph{first-order calculus} on GW space--comprising directional derivatives of functionals, gradients as linear functionals (or vectors), first-order optimality conditions for minimizers, and linear approximations of functions around a point--by equipping it with a formal tangent space at each point $[\mathbb Y]$, based on a Riemannian-like framework. This ends the remark. 

\end{remark}

\bigskip

For completeness, we revisit the proof of Proposition \ref{prop: grad via blow up}, originally presented in \cite[Prop. 8]{chowdhury2020gromov}. Before doing so, we state the following lemma and a subsequent observation.

\begin{lemma}\label{lem: aux lem}
    Given a finite network $\mathbb Y=(Y,\omega_Y,\mu_Y)$ and a measure space $(Z,\mu_Z)$ with an associated function $\varphi:Z\to Y$ such that $\varphi_\#\mu_Z = \mu_Y$. Then 
    \begin{enumerate}
        \item $(Z,\varphi^*\omega_Y,\mu_Z) \in [\mathbb Y]$.
        \item For any $v \in L^2(Y\times Y, \mu_Y \otimes \mu_Y)$, $v \simeq \varphi^*v \in L^2(Z\times Z, \mu_Z \otimes \mu_Z)$ (i.e., $[v]=[\varphi^*v] \in \mathrm{Tan}_{[\mathbb Y]}$).  
    \end{enumerate}
\end{lemma}
\begin{proof}
It follows from the fact that the plan in $\Pi(\mu_Z,\mu_Y)$ induced by the map $\varphi$ (i.e., $\pi(\{z,y\})=\mu_Z(\{z\})\delta_{y=\varphi(z)}$) achieves    
    $GW( (Z,\varphi^*\omega_Y,\mu_Z), \mathbb Y)=0$.
\end{proof}

\begin{remark}\label{rem: aux}
    Let  $\mathbb Y=(Y,\omega_Y,\mu_Y)=(\mathbf Y,\q)$ and $\mathbb Y'=(Y',\omega_{Y'},\mu_{Y'})=(\mathbf Y',\q')$ be finite networks such that $\mathbb Y\sim^w \mathbb Y'$ (i.e., both belong to the same class $[\mathbb Y]$).
    Let $\pi\in \Pi(\q,\q')$ be a GW OT plan between $\mathbb Y$ and $\mathbb Y'$, i.e., it achieves  $GW((\mathbf Y,\q),(\mathbf Y',\q') )=0$.
    Apply the blow-up procedure between $\mathbb Y$ and $\mathbb Y'$ with respect to $\pi$ (see Definition \ref{def: blowup} and the follow up properties). Therefore, after relabeling the nodes if necessary, one obtains a new set of nodes $Y_b=Z=Y_b'$, with cardinality equal to that of the support of $\pi$, and new distribution of mass $\mu_{Y_b}=\mu_Z=\mu_{Y_b'}$ on such nodes. Moreover, there exist functions $\varphi_Y:Z\to Y$, $\varphi_{Y'}:Z\to Y'$ (defined implicitly in parts  2 and 3 of Definition \ref{def: blowup}) such that $(\varphi_Y)_\#\mu_Z=\mu_Y$ $(\varphi_{Y'})_\#\mu_Z=\mu_{Y'}$ and  $(\varphi_Y)^*\omega_Y=(\varphi_{Y'})^*\omega_{Y'}$. 

    Now, consider $[v],[v']\in \mathrm{Tan}_{[\mathbb Y]}$ with representatives $v\in L^2(Y\times Y,\mu_Y\otimes \mu_Y)$ and $v'\in L^2(Y'\times Y',\mu_{Y'}\otimes \mu_{Y'})$. 
    As a consequence of Lemma \ref{lem: aux lem}, 
    we can obtain new representatives  $(\varphi_Y)^*v$ and $(\varphi_{Y'})^*v'$ defined over the same underlying space $Z\times Z$, i.e.,  $(\varphi_Y)^*v,(\varphi_{Y'})^*v'\in L^2(Z\times Z, \mu_Z\otimes \mu_Z)$, satisfying   $(\varphi_Y)^*v\simeq v$ and $(\varphi_{Y'})^*v'\simeq v$.  
\end{remark}

\begin{proof}[Proof of Proposition \ref{prop: grad via blow up}]\,

\textbf{Case $S=1$:} For simplicity, we start by fixing only one finite template $[\mathbb X]$, and considering the functional $[\mathbb Y]\longmapsto \frac{1}{2}GW^2(\mathbb X,\mathbb Y)$ defined over classes of finite networks, which we denote by $G$. Given a finite network $[\mathbb Y]$, we want to compute $\nabla G([\mathbb Y])$ according to Definition \ref{def: grad}.

Let $(\mathbf X,\p)$ and $(\mathbf Y,\q)$ be representatives of the weak classes of equivalence of $[\mathbb X]$ and $[\mathbb Y]$
after blow-up (for simplicity in the notation we avoid the subindex `$b$' from Definition \ref{def: blowup}). In particular, both $\mathbf X$ and $\mathbf Y$ are matrices of the same size, $M\times M$. Moreover, we can assume that a GW OT plan in $\Pi(\p,\q)$ for $GW(\mathbb X, \mathbb Y)=GW((\mathbf X,\p),(\mathbf Y, \q))$
is induced by the identity map, after permuting and relabeling the nodes if necessary, and so we can also assume  $\p=\q$. Thus,
\begin{equation}\label{eq: gw aux weak}
    GW^2(\mathbb X,\mathbb Y)
    =\sum_{i,j\in [M]}|\mathbf X_{ij}-\mathbf Y_{ij}|^2 \q_i\q_j,
\end{equation}
that is, a GW OT plan is a diagonal matrix whose diagonal is given by $\q$, and we denote it as $\mathrm{diag}(\q)$.

Consider the $M\times M$ matrix $\mathbf X-\mathbf Y$.   The candidate for $\nabla G([\mathbb Y])$ is $[\mathbf X-\mathbf Y]$.
Let $[v]\in\mathrm{Tan}_{[\mathbb Y]}$ be an arbitrary tangent vector at $[\mathbb Y]$. We want to show that 
\begin{equation*}
D_vG([\mathbb Y])=\langle [v], [\mathbf X-\mathbf Y]\rangle_{\mathrm{Tan}_{[\mathbb Y]}}.
\end{equation*}
From Remark \ref{rem: aux} applied to $[v]$ and $[\mathbf X-\mathbf Y]$ in $\mathrm{Tan}_{[\mathbb Y]}$, we can consider representatives that are matrices of the same size. Therefore, 
without loss of generality, and by abuse of notation, we can assume that $v$ and $\mathbf X-\mathbf Y$ are representatives of $[v]$ and $[\mathbf X-\mathbf Y]$, respectively, of the same size $M\times M$ (after a new blow-up, possible enlargement, realignment, and relabeling). Thus,
\begin{equation}\label{eq: tan auxxx}
  \langle [\mathbf X - \mathbf Y],[v] \rangle_{\mathrm{Tan}_{[\mathbb Y]}}=\sum_{i,j\in [M]}(\mathbf X_{ij}-\mathbf Y_{ij})v_{ij}\q_i\q_j,  
\end{equation}
where the inner product is given by the weighted trace inner product with weight  $\q\otimes \q$.

For $t\geq 0$, consider $\pi^t$  a GW optimal  $M\times M$ coupling in $\Pi(\q,\q)$ between $(\mathbf X,\q)$ (representative of $[\mathbb X]$) and $(\mathbf Y+tv,\q)$ (representative of $Exp_{[\mathbb Y]}(tv)$) such that $\lim_{t\to 0}\pi^t=\mathrm{diag}(\q)$ (notice that $Exp_{[\mathbb Y]}(0)=[\mathbb Y]$). 
On the one hand, we have:
\begin{align}
    &GW^2(\mathbb X,Exp_{[\mathbb Y]}(tv))
        =\sum_{i,j\in [M]} \sum_{k,l\in[M]} |\mathbf X_{ik}-\mathbf Y_{jl}-tv_{jl}|^2 \pi^t_{ij}\pi^t_{kl} \notag\\ 
        &=\sum_{i,j\in [M]} \sum_{k,l\in[M]} |\mathbf X_{ik}-\mathbf Y_{jl}|^2 \pi^t_{ij}\pi^t_{kl}+ t^2\sum_{i,j\in [M]} \sum_{k,l\in[M]}|v_{jl}|^2 \pi^t_{ij}\pi^t_{kl}\notag\\
        &\qquad +2t\sum_{i,j\in [M]} \sum_{k,l\in[M]} (\mathbf X_{ik}-\mathbf Y_{jl})v_{jl}\pi_{ij}^t\pi^t_{kl}\notag\\
        &=\sum_{i,j\in [M]} \sum_{k,l\in[M]} |\mathbf X_{ik}-\mathbf Y_{jl}|^2 \pi^t_{ij}\pi^t_{kl}+ t^2\sum_{i,j\in [M]} |v_{ij}|^2 \q_i\q_j\notag\\
        &\qquad +2t\sum_{i,j\in [M]} \sum_{k,l\in[M]} (\mathbf X_{ik}-\mathbf Y_{jl})v_{jl}\pi_{ij}^t\pi^t_{kl}\notag\\
        &\geq GW^2(\mathbb X,\mathbb Y)+t^2\sum_{i,j\in [M]}|v_{ij}|^2 \q_i\q_j+2t\sum_{i,j\in [M]} \sum_{k,l\in[M]} (\mathbf X_{ik}-\mathbf Y_{jl})v_{jl}\pi_{ij}^t\pi^t_{kl}.\label{eq: bound 1}
\end{align}
On the other hand, using \eqref{eq: gw aux weak} and since $\mathrm{diag}(\q)$ is an admissible coupling but not necessarily optimal for $GW^2(\mathbb X,Exp_{[\mathbb Y]}(tv))$, we have:
\begin{align}
    &GW^2(\mathbb X,Exp_{[\mathbb Y]}(tv))
        =\sum_{i,j\in [M]} \sum_{k,l\in[M]} |\mathbf X_{ik}-\mathbf Y_{jl}-tv_{jl}|^2 \pi^t_{ij}\pi^t_{kl}\notag\\ 
        &\leq
        \sum_{i,j\in [M]}|\mathbf X_{ij}-\mathbf Y_{ij}-tv_{ij}|^2 \q_i\q_j \notag\\
        &=\sum_{i,j\in [M]}|\mathbf X_{ij}-\mathbf Y_{ij}|^2 \q_i\q_j+ t^2\sum_{i,j\in [M]}|v_{ij}|^2 \q_i\q_j+2t\sum_{i,j\in [M]}(\mathbf X_{ij}-\mathbf Y_{ij})v_{ij}\q_i\q_j\notag\\
        &= GW^2(\mathbb X,\mathbb Y)+t^2\sum_{i,j\in [M]}|v_{ij}|^2 \q_i\q_j+2t\sum_{i,j\in [M]}(\mathbf X_{ij}-\mathbf Y_{ij})v_{ij}\q_i\q_j. \label{eq: bound 2}
\end{align}
In what follows we use \eqref{eq: bound 1} and \eqref{eq: bound 2}  to provide upper and lower bounds for the directional derivative
\begin{equation*}
    2\, D_vG([\mathbb Y])=\lim_{t\to 0}\frac{GW^2(\mathbb X,Exp_{[\mathbb Y]}(tv))-GW^2(\mathbb X,\mathbb Y)}{t}.
\end{equation*}
\begin{itemize}
    \item \emph{Upper bound:} 
\begin{align*}
    &\frac{GW^2(\mathbb X,Exp_{[\mathbb Y]}(tv))-GW^2(\mathbb X,\mathbb Y)}{t}\leq \sum_{i,j\in [M]}\left(t|v_{ij}|^2 +2(\mathbf X_{ij}-\mathbf Y_{ij})v_{ij}\right)\q_i\q_j.
\end{align*}
Thus,
\begin{align}\label{eq: prop 8 b}
    \lim_{t\to 0}&\frac{GW^2(\mathbb X,Exp_{[\mathbb Y]}(tv))-GW^2(\mathbb X,\mathbb Y)}{t}\leq 2\sum_{i,j\in [M]}(\mathbf X_{ij}-\mathbf Y_{ij})v_{ij}\q_i\q_j.
\end{align}
    \item \emph{Lower bound:}    
\begin{align*}
    &\frac{GW^2(\mathbb X,Exp_{[\mathbb Y]}(tv))-GW^2(\mathbb X,\mathbb Y)}{t}\\
    &\geq t\sum_{i,j\in [M]}|v_{ij}|^2 \q_i\q_j+2\sum_{i,j\in [M]} \sum_{k,l\in[M]} (\mathbf X_{ik}-\mathbf Y_{jl})v_{jl}\pi_{ij}^t\pi^t_{kl}.
\end{align*}
Thus,
\begin{align}\label{eq: prop 8 a}
    \lim_{t\to 0}&\frac{GW^2(\mathbb X,Exp_{[\mathbb Y]}(tv))-GW^2(\mathbb X,\mathbb Y)}{t}\geq 2\sum_{i,j\in [M]}(\mathbf X_{ij}-\mathbf Y_{ij})v_{ij}\q_i\q_j.
\end{align}    
\end{itemize}
As a consequence of \eqref{eq: tan auxxx}, \eqref{eq: prop 8 b} and \eqref{eq: prop 8 a}, for a fixed $[\mathbb X]$, the directional derivative of the functional $G([\mathbb Y])= \frac{1}{2}GW^2(\mathbb X,\mathbb Y)$ in the direction of $[v]\in\mathrm{Tan}_{[\mathbb Y]}$ at $[\mathbb Y]$ is given by
\begin{align*}
   D_vG([\mathbb Y])&=\frac{1}{2}\lim_{t\to 0}\frac{GW^2(\mathbb X,Exp_{[\mathbb Y]}(tv))-GW^2(\mathbb X,\mathbb Y)}{t}\\
   &=\sum_{i,j\in [M]}(\mathbf X_{ij}-\mathbf Y_{ij})v_{ij}\q_i\q_j\\
   &=\langle [\mathbf X - \mathbf Y],[v] \rangle_{\mathrm{Tan}_{[\mathbb Y]}}.  
\end{align*}
Therefore, the gradient of the functional $[\mathbb Y]\longmapsto \frac{1}{2}GW^2(\mathbb X,\mathbb Y)$ can be identified with the matrix $\mathbf X-\mathbf Y$, obtained after blow-up.

Moreover, notice that for the case $S=1$, we have:
\begin{equation*}
    GW^2(\mathbb X,\mathbb Y)=\mathrm{tr}_\q\left((\mathbf X-\mathbf Y)^T(\mathbf X-\mathbf Y)\right)= \|[\mathbf X-\mathbf Y]\|^2_{\mathrm{Tan}_{[\mathbb Y]}}.
\end{equation*}

\textbf{General Case $S\geq 1$:}
Let us fix a collection $\{[\mathbb{X}^s]\}_{s=1}^S$ of classes of finite networks, $\lambda\in \Delta_{S-1}$, and the functional $G:\mathcal{GW}\to \R_{\geq 0}$ given in \eqref{eq: barycenter functional}.
Given a finite network $[\mathbb Y]$, we want to compute $\nabla G([\mathbb Y])$ according to Definition \ref{def: grad}. 

Following the outline provided by the authors in \cite{chowdhury2020gromov}, suppose that each $\mathbb X^s$ has been aligned to $\mathbb Y$, that is, suppose we have applied Algorithm \ref{alg: blow up}  so that each network of
$\{\mathbb Y=(\mathbf Y,\q),\mathbb X^1=(\mathbf X^1,\q),\dots, \mathbb X^S=(\mathbf X^S,\q)\}$ has $M$ nodes with mass distribution $\q$, and the identity map from $[M]$ to $[M]$ is optimal for each $GW(\mathbb X^s, \mathbb Y)$, in the sense that: 
\begin{equation*}
    GW^2(\mathbb X^s,\mathbb Y)
    =\sum_{i,j\in [M]}|\mathbf X_{ij}^s-\mathbf Y_{ij}|^2 \q_i\q_j, \qquad \forall s\in [S]. 
\end{equation*}
Consider the $M\times M$ matrix $\sum_{s=1}^S\lambda_s \mathbf X^s-\mathbf Y$.   The candidate for $\nabla G([\mathbb Y])$ is $[\sum_{s=1}^S\lambda_s \mathbf X^s-\mathbf Y]\in \mathrm{Tan}_{[\mathbb Y]}$. 
Let $[v]\in\mathrm{Tan}_{[\mathbb Y]}$ be an arbitrary tangent vector at $[\mathbb Y]$. As before, we can consider representatives of $[\sum_{s=1}^S\lambda_s \mathbf X^s-\mathbf Y]$ and $[v]$ that are matrices of the same size. By abuse of notation, we denote such representatives as $\sum_{s=1}^S\lambda_s \mathbf X^s-\mathbf Y$ and $v$,  assuming they belong to the space of $M\times M$ matrices endowed with the weighted Frobenius norm, with weight $\q\otimes \q$:
\begin{align}
\langle [v], [\sum_{s\in [S]}\lambda_s\mathbf X^s-\mathbf Y]\rangle_{\mathrm{Tan}_{[\mathbb Y]}}&=\sum_{i,j\in [M]}\left(\sum_{s\in [S]}\lambda_s\mathbf X_{ij}^s-\mathbf Y_{ij}\right)v_{ij} \,  \q_i\q_j\notag \\
&= \sum_{s\in [S]}\lambda_s\sum_{i,j\in [M]}\left(\mathbf X_{ij}^s-\mathbf Y_{ij}\right)v_{ij} \,  \q_i\q_j, \label{eq: S>1 aux}
\end{align}
where the second equality holds from the fact that $\sum_{s\in S}\lambda_s=1$.
We want to show that $D_vG_\lambda([\mathbb Y])=\langle [v], [\sum_{s\in [S]}\lambda_s\mathbf X^s-\mathbf Y]\rangle_{\mathrm{Tan}_{[\mathbb Y]}}$. 

We repeat the same calculations as in the case $S=1$, now for each $s\in [S]$, by considering $\pi^{s,t}$ a GW optimal  $M\times M$ coupling in $\Pi(\q,\q)$ between $(\mathbf X^s,\q)$ and $(\mathbf Y+tv,\q)$ (where the last network is a representative of $Exp_{[\mathbb Y]}(tv)$), for $t\geq 0$,  such that $\lim_{t\to 0}\pi^{s,t}=\mathrm{diag}(\q)$, obtaining
\begin{align*}
   D_vG_\lambda([\mathbb Y])&=\frac{1}{2}\sum_{s\in [S]}\lambda_s\lim_{t\to 0}\frac{GW^2(\mathbb X^s,Exp_{[\mathbb Y]}(tv))-GW^2(\mathbb X^s,\mathbb Y)}{t}\\
   &=\sum_{s\in [S]}\lambda_s\sum_{i,j\in [M]}(\mathbf X_{ij}^s-\mathbf Y_{ij})v_{ij}\q_i\q_j\\
   &=\langle [\sum_{s\in [S]}\lambda_s\mathbf X^s - \mathbf Y],[v] \rangle_{\mathrm{Tan}_{[\mathbb Y]}},  
\end{align*}
where the last equality holds from \eqref{eq: S>1 aux}.
Therefore, $\nabla G_\lambda([\mathbb Y])$ can be identified with the matrix $\sum_{s\in [S]}\lambda_s\mathbf X^s-\mathbf Y$, obtained after applying the blow-up algorithm \ref{alg: blow up}.
\end{proof}

We now proceed with a remark on blow-ups.

\begin{remark}\label{rem: remark trivial}
    Let $[\mathbb Y]$ be any finite network class, and let us apply Algorithm \ref{alg: blow up} with respect to the given templates $\{[\mathbb X^s]\}_{s=1}^S$ to obtain the corresponding blow-ups. On the one hand, we have that the blow-up 3-tuple $(Y_b,\omega_{Y_b},\mu_{Y_b})$ is a representative of the class $\mathbb{[Y]}$. On the other hand, as a consequence of Proposition \ref{prop: grad via blow up}, $$\nabla G_\lambda([\mathbb Y])= \left[\sum_{s=1}^S\lambda_s\omega_{X_b^s}-{\omega_{Y_b}}\right]\in \mathrm{Tan}_{[\mathbb Y]},$$ 
    and so
\begin{equation*}
Exp_{[\mathbb Y]}([\nabla G_\lambda([\mathbb Y])])=[(Y_b,\cancel{\omega_{Y_b}}+\sum_{s=1}^S\lambda_s\omega_{X_b^s}-\cancel{\omega_{Y_b}},\mu_{Y_b})]= [(Y_b,\sum_{s=1}^S\lambda_s\omega_{X_b^s},\mu_{Y_b})].
\end{equation*}
Now, assume that $\nabla G_\lambda([\mathbb Y])=0$ for some $\lambda\in \Delta_{S-1}$. Hence, 
\begin{align*}
   [(Y_b,\omega_{Y_b},\mu_{Y_b})] =[\mathbb Y]=Exp_{[\mathbb Y]}(0)=Exp_{[\mathbb Y]}([\nabla G_\lambda([\mathbb Y])])=[(Y_b,\sum_{s=1}^S\lambda_s\omega_{X_b^s},\mu_{Y_b})]
\end{align*}
implying that
\begin{equation}\label{eq: shape of bary grad method}
    \omega_{Y_b}=\sum_{s=1}^S\lambda_s\omega_{X_b^s} \quad \text{ or, equivalently, } \quad \mathbf Y_b =\sum_{s=1}^S\lambda_s \mathbf X^s_b.
\end{equation}
See also \cite[Rmk. 9]{chowdhury2020gromov}. It is important to remark that the expression $\sum_{s=1}^S\lambda_s\omega_{X_b^s}$ depends on the given network $\mathbb{Y}$.
That is, given another network $\mathbb Z$, one has to perform Algorithm \ref{alg: blow up} to $\mathbb Z$ and the templates $\{\mathbb X^s\}$ getting new representatives of all the networks involved and then the new corresponding  expression $\sum_{s=1}^S\lambda_s\omega_{X_b}$ will be different to the above one \eqref{eq: shape of bary grad method}. Thus, it would be better to denote the new representatives of the templates after  blow-up with respect to $\mathbb Y$ as 
$$\mathbb X^s\sim^w(\mathbf X_{b,Y}^s, \, \q_{b,Y}) \qquad \forall s\in [S], $$
and  after  blow-up with respect to $\mathbb Z$ as 
$$\mathbb X^s\sim^w(\mathbf X_{b,Z}^s, \, \q_{b,Z}) \qquad \forall s\in [S].$$
By transitivity, it holds that 
$$(\mathbf X_{b,Y}^s, \, \q_{b,Y})\sim^w(\mathbf X_{b,Z}^s, \, \q_{b,Z}) \qquad \forall s\in [S]$$
However, in general
$$\left(\sum_{\lambda_s}\mathbf X_{b,Y}^s, \,  \q_{b,Y}\right)\not\sim^w\left(\sum_{\lambda_s}\mathbf X_{b,Z}^s, \,  \q_{b,Z}\right).$$
In particular, there may be more than one critical point of $G_\lambda$. The discussion above was meant to avoid confusion: \emph{the expression $\sum_{s=1}^S\lambda_s \mathbf X^s_b$ does not depend solely on the templates.}
\end{remark}

We conclude this section by providing a proof of the second part of Lemma \ref{thm: equiv}, which does not rely on the first-order optimality condition in GW space derived from the first-order calculus introduced in \cite{chowdhury2020gromov} and reviewed throughout this work.

\begin{proof}[Proof of Lemma \ref{thm: equiv}]
    Our goal is to provide a direct proof for the implication $$[\mathbb Y^*]\in \mathrm{Bary}(\{[\mathbb X^s]\}_{s\in [S]})\Longrightarrow \mathbb Y^*\sim^w\left(\sum_{s=1}^S\lambda_s\mathbf X^s_b,\q_b\right).$$
    The spaces $\mathbb X_b^s=(\mathbf X^s_b, \q_b)$, for $s\in [S]$, and $\mathbb Y^*_b=(\mathbf Y_b^*, \q_b)$ are obtained by applying Algorithm \ref{alg: blow up}. Let $M_b$ be the cardinality of these new spaces, that is, $\mathbf X_b^s,\mathbf Y_b^*\in \R^{M_b\times M_b}$, for all $s\in [S]$, and $\q_b\in \mathcal{P}_{M_b}$.
    By re-aligning the new nodes, one can consider the identity map as an OT map for each problem  $GW(\mathbb X^s_b, \mathbb Y_b^*)$, $s=1,\dots, S$. In other words,   as in \eqref{eq: GW as Frob}, we have:
    \begin{align*}
    GW^2(\mathbb X^s, \mathbb Y^*)&= GW^2(\mathbb X^s_b, \mathbb Y^*_b) \\&= \sum_{j,k\in [M_b]}|\mathbf X_b^s[j,k]- \mathbf Y_b^*[j,k]|^2 \, \q_b[j]\q_b[k]\\
    &=\mathrm{tr}_{\q_b}\left((\mathbf X_b^s -\mathbf Y_b^*)^T(\mathbf X_b^s -\mathbf Y_b^*)\right).
\end{align*}
    We proceed as in the proof of Theorem \ref{prop: formula_bary} on the fixed-point approach. 
    Since one can consider the identity map on $[M_b]$ as an optimal map for the problem $GW((\mathbf X^s_b,\q_b), (\mathbf Y^*_b,\q_b))$,  in this case,  the functional $J_\lambda:\R^{M_b\times M_b}\to \R_{\geq 0}$  can be written as
     \begin{equation*}
        J_\lambda(\mathbf Y)=\sum_{s\in [S]}\lambda_s\sum_{j,k\in [M_b]}|\mathbf X_b^s[j,k]- \mathbf Y[j,k]|^2 \, \q_b[j]\q_b[k]=\sum_{s\in [S]}\lambda_s \, \mathrm{tr}_{\q_b}\left((\mathbf X_b^s -\mathbf Y)^T(\mathbf X_b^s -\mathbf Y)\right)
    \end{equation*}
    (compare to the expression \eqref{eq: J functional}).
    We know that it admits one and only one minimizer, which is in fact given by the expression $$\sum_{s=1}^S\lambda_s \mathbf X^s_b.$$ This can be deduced from \eqref{eq: relation} or simply because minimizing $J_\lambda$ coincides with the barycentric problem within the Euclidean geometry (see \eqref{eq: bary ad convex comb2} in Section \ref{sec: BCM}). 
    Also, since $\mathbb{Y}^*$ is a GW barycenter, we have
        \begin{align*}
        \sum_{s\in [S]}\lambda_s \, GW^2(\mathbb X^s, \mathbb Y^*)
        &=\min_{[\mathbb Y]}\sum_{s\in [S]}\lambda_s \, GW^2(\mathbb X^s, \mathbb Y)\\
        &=\min_{[\mathbb Y]}\sum_{s\in [S]}\lambda_s \, GW^2(\mathbb X^s_b, \mathbb Y)
        \\&\leq\sum_{s\in [S]}\lambda_s \, GW^2\left((\mathbf X_b^s, \, \q_b), \left(\sum_{r\in [S]}\lambda_r\mathbf X_b^r, \, \q_b\right)\right)\\
        &\leq \displaystyle\sum_{s\in [S]}\lambda_s\sum_{j,k\in [M_b]}\left|\mathbf X_b^s[j,k]- \left(\sum_{r\in [S]}\lambda_r\mathbf X_b^r\right)[j,k]\right|^2  \q_b[j]\q_b[k]\\&=\min_{\mathbf Y\in \R^{M_b\times M_b}}J_\lambda(\mathbf Y)\\
        &\leq \sum_{s\in [S]}\lambda_s\sum_{j,k\in [M_b]}\left|\mathbf X_b^s[j,k]- \mathbf Y_b^*[j,k]\right|^2  \q_b[j]\q_b[k]\\
        &=\sum_{s\in [S]}\lambda_s \, GW^2(\mathbb X^s_b, \mathbb Y^*_b)
        \\&=\sum_{s\in [S]}\lambda_s \, GW^2(\mathbb X^s, \mathbb Y^*) 
        \end{align*}
    (where the second to last of the inequalities follows from the fact that $\mathrm{diag}(\q_b)\in \Pi(\q_b,\q_b)$ is an admissible plan for $GW\left((\mathbf X_b^s, \, \q_b), \left(\sum_{r\in [S]}\lambda_r\mathbf X_b^r, \, \q_b\right)\right)$ but not necessarily an optimal coupling),   
    establishing that all the inequalities in the chain are actually equalities. Therefore, the matrix $\mathbf Y_b^*$ is a minimizer of $J_\lambda$, implying that $\mathbf Y_b^*=\sum_{r\in [S]}\lambda_s\mathbf X_b^s$ (by uniqueness of the minimum of $J_\lambda$, which is strictly convex due to having $\q_b$ with all entries strictly positive). Hence, 
    \begin{align*}
     \mathbb Y^*\sim^w(\mathbf Y_b^*,\q_b)=\left(\sum_{s\in [S]}\lambda_s\mathbf X_b^s,\q_b\right).    
    \end{align*}

\end{proof}

\section{Computational Complexity of Algorithms \ref{alg: analysis} and \ref{alg: analysis blow up}}\label{app: comp comp}

For simplicity, let us assume that the number of nodes is the same across the templates, that is, assume $N^s=N$ for all $s\in [S]$. Assume the size of the input network $\mathbb Y=(\mathbf Y,\q)$ is $M$ (i.e., $\mathbf Y]in \R^{M\times M}$ and $\q\in \R^M$).
\\

\noindent \textbf{Computation of GW Distances:}
The Gromov-Wasserstein problem is inherently NP-hard due to its formulation as a quadratic assignment problem \cite{kravtsova2024np, scetbon2022linear, chen2024semidefinite, COTFNT, zhang2023duality}. In practice, approximation schemes are employed to mitigate this computational burden.
For example, if we consider two networks with $N$ and $M$ nodes, respectively,  evaluating the GW objective for a fixed coupling matrix incurs a computational cost of $\mathcal{O}(N^2 M^2)$ that can be approximated using entropic regularization \cite{peyre2016gromov} reducing its cost to  $\mathcal{O}((N^2 M + N M^2))$. 
The original optimization problem can then be approximately solved using a sequence of Sinkhorn iterations, with each iteration cost being dominated by the complexity of evaluating the GW objective. Taking $L$ iterations to convergence, the total cost of solving a single GW problem becomes $\mathcal{O}((N^2 M + N M^2)L)$.

In our experiments, from the POT Library \cite{flamary2021pot}, 
we employed the functions:
\begin{itemize}
    \item \texttt{ot.gromov.gromov\_wasserstein}, which returns an optimal GW transport plan (by solving the optimization problem via a conditional gradient --Frank-Wolfe-- method, for which one iteration requires approximately $N^2M+NM^2$ operations plus the cost of solving one linear OT subproblem), and 
    \item \texttt{ot.gromov.gromov\_barycenters} (which uses block coordinate descent when optimizing) based on the fixed-point algorithm proposed in \cite{peyre2016gromov}, which  
is similar to Algorithm \ref{alg: synthesis}.
\end{itemize}

\bigskip

\noindent \textbf{Algorithm \ref{alg: analysis}:}
\begin{enumerate}[leftmargin=*]
    \item {Computation of GW distances:} 
    First, one needs to solve (or estimate) as many GW problems as templates ($S$). Thus,  
    extending the  discussion above  to $S$ templates, the overall complexity to \emph{estimate} $\{\pi(\mathbf Y,s)\}_{s\in [S]}$ is essentially $\mathcal{O}((N^2 M + NM^2) LS)$. 
    \item Computing each matrix $F(\mathbf Y, s)$ requires $\mathcal{O}(N^2M+NM^2)$ operations, leading to a total complexity of $\mathcal{O}((N^2M+NM^2)S)$.
    \item Each trace operation between two $M\times M$ matrices costs $\mathcal{O}(M^2)$. Therefore, computing the matrix $\mathcal{Q}$ requires $\mathcal{O}(M^2S^2)$ operations (or, more precisely, $\mathcal{O}(M^2\frac{S(S+1)}{2})$ leveraging on the fact that $\mathcal{Q}$ is symmetric).
    \item     Solving the quadratic program  $\min_{\lambda\in \Delta_{S-1}} \lambda^T\mathcal{Q}\lambda$ is equivalent to $\min_{\lambda\in \Delta_{S-1}} \lambda^TK\lambda -2\lambda^T b$ (see \eqref{eq: quadratic program 1}). 

    In our simulations, this is done using the iterative Alternating Direction Method of Multipliers (ADMM)-based first-order method, implemented via the function \texttt{cvxpy.OSQP} within the \texttt{cvxpy} optimization library in Python, which has a computational complexity of $\mathcal O{(S^3+L'S^2)}$, where $L'$ is the number of iterations. 

    In practice, the number of iterations $L'$ can be much larger than $S$ (since, in this work,  our experiments involve a small number $S$) thus making this unnecessarily costly. A heuristic to reduce this is to relax the linear problem $K\lambda=b$ s.t. $\lambda\in \Delta_{S-1}$ by only considering the equality restriction $\sum_{s\in [S]} \lambda_s=1$. This is done using linear solvers that cost $\mathcal O(S^3)$.
\end{enumerate}
This yields a total complexity of order essentially
$\mathcal{O}((N^2M+NM^2)LS +M^2S^2+S^3)$.
Since the number of templates 
$S$ is typically much smaller than the number of nodes $N$, we can conclude that the computational complexity of Algorithm \ref{alg: analysis} is comparable to solving $S$ Gromov-Wasserstein optimization problems \eqref{eq: rest_gw}. This could be further improved  by parallelizing the  
$S$ Gromov-Wasserstein computations.
\\

\noindent \textbf{Algorithm \ref{alg: analysis blow up}:}
\begin{enumerate}[leftmargin=*]
    \item Blow-up step: The computational complexity of Algorithm \ref{alg: blow up}, is essentially equivalent to solving $S$ Gromov-Wasserstein optimization problems.  

    The $S$ concatenated GW problems that we need to solve are defined between networks whose sizes increase with $S$. Specifically, after performing the first blow-up between the first template $\mathbb X^1$, with $N$ nodes, and $\mathbb Y$, with $M$ nodes, we obtain new networks that are weakly isomorphic to the originals but have size $|\mathrm{supp}{(\pi^1)}|\geq\max\{N,M\}$, where $\pi^1\in \Pi(\p^1,\q)$ is an optimal coupling for $GW(\mathbb X^1,\mathbb Y)$. This process continues similarly for subsequent steps, involving blow-ups for the rest of the templates $\mathbb X^2, \dots, \mathbb X^S$ with respect to $\mathbb Y$, possibly increasing the number of nodes at each iteration $s\in [S]$.

    In the case of classical OT, any optimal coupling--which is a feasible solution of the corresponding linear program--has support size at most 
    $N+M-1$, assuming the source and target discrete measures have support of size $N$ and $M$, respectively (see, e.g., \cite[Prop. 3.4]{COTFNT}). Linear programming theory does not apply in the GW case due to the quadratic nature of the objective, even though the optimization is performed over the same set of transport plans (or couplings). As a result, no support size bound analogous to that in classical OT is currently known. For empirical insight, we refer the reader e.g. to \cite[Suppl. Mat., Sec. C]{chowdhury2020gromov}, where the authors observe that representations often occur in a space much smaller than the naive requirement of size $NM$.

    As a result, accurately estimating the total computational cost of this step is challenging.

\item Suppose that,  after the blow-up step,  each $\mathbb X^s$ has been aligned to $\mathbb Y$,   so that the collection $\{\mathbb Y_b=(\mathbf Y,\q_b),\mathbb X^1_b=(\mathbf X^1,\q_b),\dots, \mathbb X^S_b=(\mathbf X^S,\q_b)\}$ consists of networks with $M_b$ nodes. Then, similarly to Algorithm \ref{alg: analysis}, computing the matrix $\mathcal{A}$ is of order $\mathcal O({M_b^2 S^2})$, and  solving the quadratic program over the simplex $\min_{\lambda\in \Delta_{S-1}} \lambda^T\mathcal{A}\lambda$ costs approximately $\mathcal{O}(S^3)$. 
\end{enumerate}

\section{Experiments}\label{app: experiments}

This section aims to supplement Section \ref{sec: experiments} with additional visualizations and experiments.
Suppl. Mat. \ref{app: classif} expands on the material in Subsection \ref{sec: classif main}. Suppl. Mat. \ref{app: experiments - corr} presents an additional application of the proposed GW-BCM.

\subsection{Clustering, Classification, and Visualizations}\label{app: classif}

We provide supplementary material that builds on Subsection \ref{sec: classif main}.

\subsubsection{Additional Visualizations}\label{app: classif max}
In addition to Figures \ref{fig: baryspace01}, we provide a complementary illustration that helps interpret the meaning of GW barycentric coordinates.

\textbf{Description of Fig. \ref{fig: Delta_2}:} While Figure \ref{fig: baryspace01} considers two classes from the Point Cloud MNIST 2D dataset \cite{Garcia2023PointCloudMNIST2D}, the new Fig. \ref{fig: Delta_2} extends the setting to three classes. From the Point Cloud MNIST 2D dataset \cite{Garcia2023PointCloudMNIST2D}, we consider three digit classes: 0, 4, and 7. 
We select one reference point cloud per class, and compute and plot  GW barycentric coordinates $(\lambda_1,\lambda_2,\lambda_3)$ for each point cloud in the set. 
The visualization provided in Fig. \ref{fig: Delta_2} highlights how GW barycentric coordinates capture meaningful structure in the data, with a particularly clear clustering around the template for digit class 0. In contrast, digit classes 4 and 7 exhibit overlap, suggesting the presence of shared geometric shapes between them.

\begin{figure}[h!]
    \centering
    \includegraphics[width=1\linewidth]{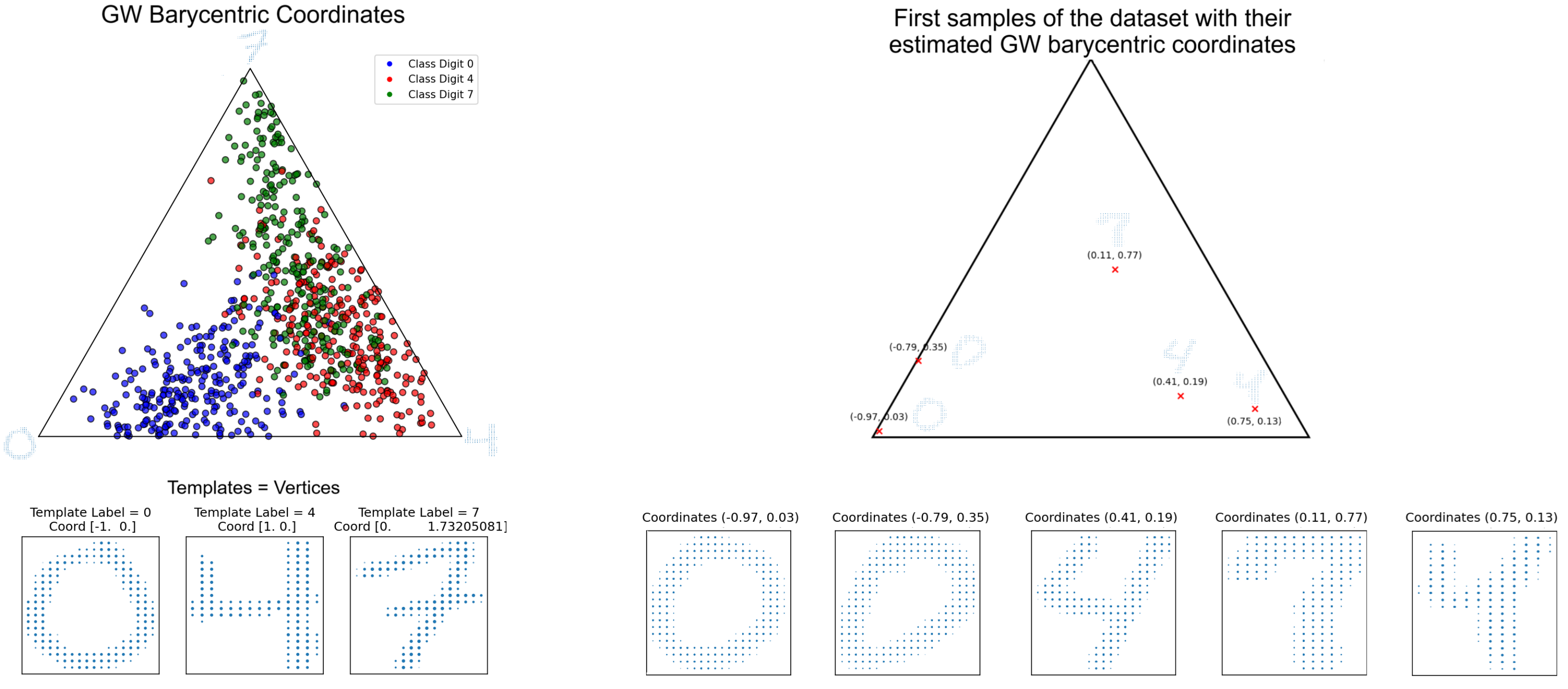}
    \caption{\small{    
    From the Point Cloud MNIST 2D dataset \cite{Garcia2023PointCloudMNIST2D}, we consider three digit classes: 0, 4, and 7. One representative sample from each class is selected and used as a template, placed at the vertices of an equilateral triangle that represents the 2-simplex $\Delta_2$. Then, for 900 randomly selected samples from the dataset, we estimate their GW barycentric coordinates using Algorithm \ref{alg: analysis}, and project them into the triangle (left panel). Each point is colored according to its true class label. As an additional illustration, the right panel shows the first five point clouds from the dataset together with their positions within the triangle according to their computed GW barycentric coordinates.    }}
    \label{fig: Delta_2}
\end{figure}

\subsubsection{Additional Classification Experiments} \label{sec: additional_classification_app}

We run additional classification experiments on graph datasets from the public TU Dortmund collection~\cite{morris2020tudataset}, in order to compare our analysis method (in particular, Algorithm \ref{alg: analysis}) with the GW-Dictionary Learning approach proposed in \cite{D-vincent2021online}. In both cases, we treat the resulting representations as graph embeddings and perform classification in the corresponding estimated GW-coordinate spaces. The results are reported in Table \ref{table: app other datasets}.

For these datasets, pairwise node dissimilarities are no longer Euclidean and are defined as shortest-path distances on the graph. The dimensionality of the resulting (low-dimensional) GW-coordinate spaces was chosen to be the same for both approaches, and is determined as follows: (a) in our method, by analysis with respect to a set of three randomly selected templates per class; (b) in the GW-Dictionary approach, by the unmixing weights associated with a dictionary of $3 \times C$ learned atoms, where $C$ denotes the number of classes.

In the GW-Dictionary approach, the dictionary atoms are learned in an offline phase using the full dataset, while the embedding of each graph is obtained by solving a GW linear-like problem to compute the corresponding unmixing weights (please refer to \cite{D-vincent2021online} for the precise setup and details). In contrast, our method does not require any training phase to compute the embedding coordinates.

We performed 20 repeated stratified $70/30$ train/test splits. For each split, a 5-fold cross-validation procedure was used to select the optimal number of neighbors $k$ for the $k$-NN classifier, which was then evaluated on the corresponding test split. The embedding times reported in Table \ref{table: app other datasets} correspond to the time required to compute the embedding coordinates for the full dataset. For the GW-Dictionary method (DICT), the reported offline time is the cost of learning the dictionary atoms. The comparable classification accuracies indicate that both embedding strategies capture similar discriminative information from the graph data. While the DICT baseline achieves slightly higher accuracies, our method is competitive (there is not statistically significant difference) and does not require the offline training time reported for learning the dictionary atoms.

\begin{table}[h]
\centering
\setlength{\tabcolsep}{6pt}
\renewcommand{\arraystretch}{1.3}
\begin{tabular}{lcccc}
\hline
 & \multicolumn{2}{c}{\textbf{Accuracy}}
 & \multicolumn{2}{c}{\textbf{Avg Embedding Time (s)}} \\
\textbf{Dataset}
 & \textbf{FP} & \textbf{DICT}
 & \textbf{FP} & \textbf{DICT} \\
\hline
\textbf{PROTEINS}
 & $0.6678 \pm 0.0249$
 & $0.6711 \pm 0.0269$
 & 26.93
 & 14.37 \,(offl.\ 569.23) \\
\textbf{MUTAG}
 & $0.6904 \pm 0.0558$
 & $0.7035 \pm 0.0557$
 & 2.29
 & 2.76 \,(offl.\ 42.38) \\
\textbf{COX2}
 & $0.7656 \pm 0.0153$
 & $0.7741 \pm 0.0171$
 & 10.06
 & 10.77 \,(offl.\ 124.10) \\
\hline
\end{tabular}
\caption{
Classification accuracy and average runtimes.
For GW-Dictionary Learning (DICT) \cite{D-vincent2021online}, embedding time corresponds to unmixing the full dataset once;
dictionary learning (atoms) is performed offline and reported in parentheses.
}
\label{table: app other datasets}
\end{table} 

\subsubsection{Minimum GW Barycenter
Loss - Experimental Setup}\label{app: classif min}\,
The following is an additional classification method based on the proposed GW-BCM approach.
\begin{enumerate}[leftmargin=*]
    \item Let $n$ be the number of classes in a dataset. Select $S$ random points of each class.
    \item If  needed, preprocess each data point, obtaining elements of the form $(\mathbf{X}, \p)$: a dissimilarity or edge weight matrix between nodes, together with a mass distribution on its nodes.
    \item Given an input $(\mathbf{X}, \p)$ from the data set, compute the GW barycentric coordinates (i.e., apply Algorithm \ref{alg: analysis} or \ref{alg: analysis blow up}) using the references from each class separately getting a set of $n$ vectors $\{\lambda_i\}_{i=1}^n$, where $\lambda_i\in \Delta_{S-1}$. 
    \item Synthesis set: For each vector $\lambda_i$, synthesize the corresponding GW barycenters using the $S$ templates, obtaining a total of $n$ matrices $\{\mathbf X_{\lambda_i}\}_{i=1}^n$ (with probability vectors $\{\p_{\lambda_i}\}_{i=1}^n$).
    \item Compute GW distances between input and the synthesized GW barycenters, and select the class with the minimum GW distance, that is:
    \begin{equation}\label{eq: pred}
      \text{Predicted class for input } (\mathbf X,\p) = \argmin_{1\leq i\leq n}GW((\mathbf X,\p),(\mathbf X_{\lambda_i},\p_{\lambda_i})) . 
    \end{equation}
\end{enumerate}

As a proof of concept, we apply the methodology described above to the MNIST PointCloud dataset \cite{Garcia2023PointCloudMNIST2D} and compute the accuracy between the true and predicted labels. Results are reported below:
\begin{itemize}[leftmargin=*]
    \item For classifying digit classes 0 and 1 within a set of 400 point cloud digits, the overall accuracy obtained using $S = 1$ or $S = 2$ sample point clouds per class is around 0.91. When increasing $S$ to 4, the accuracy improves to 0.96.
    \item For classifying digit classes 0, 1, and 2 within a set of 600 point cloud digits, the overall accuracy obtained using $S = 1$ or $S = 2$ sample point clouds per class is 0.8. When increasing $S$ to 4, the accuracy does not improve.
\end{itemize}
In conclusion, the performance of this approach is not particularly strong, and since it involves GW synthesis with estimated GW barycentric coordinates, along with multiple GW computations for classification (see \eqref{eq: pred}), it is computationally expensive. 

\subsection{Point Cloud Corruption 
due to Occlusion}\label{app: experiments - corr}

In real-world applications, data corruption is inevitable due to factors such as scene complexity, sensor inaccuracies, and processing imprecision. For point cloud data, we focus on a specific form of corruption: the occlusion of portions of the point cloud.
Below, we describe the experimental setup used to address this issue.
In our experiments we use the Euclidean distance (in $\R^2$ or $\R^3$) for the dissimilarity matrices, i.e., $\mathbf Y_{ij}=\|y_i-y_j\|$. All the synthesized GW barycenters are computing by applying the function   \texttt{ot.gromov.gromov\_barycenters} from the POT Library \cite{flamary2021pot}.     

\medskip

\noindent \textbf{Experimental Setup:}

\begin{enumerate}[leftmargin=*]
    \item Input  a 2D or 3D point cloud represented by the discrete probability measure $\mu_Y=\sum_{j=1}^M \q[j]\delta_{y_j}$, where $\q\in \mathcal{P}_M$. Compute the dissimilarity matrix $\mathbf {Y}$ between the nodes $\{y_j\}_{j=1}^M$.  
    Simulate occlusion by computing
    $\widetilde{\mu_{Y}}$ from $\mu_Y$ with a removed portion. Re-normalize the corresponding vector of weights $\widetilde \q$ so that it has total mass $1$.  Compute the dissimilarity matrix $\widetilde{\mathbf{Y}}$ between the remaining nodes.
    \item From the dataset that is used as a model for the given input, select a set of samples  to serve as the reference measures $\mu_{X^s}=\sum_{i=1}^{N^s}\p^s[i]\delta_{x_i^s}$, for $s\in [S]$ (if $\mu_Y$ was taken from such dataset, select templates in the same class as $\mu_Y$). Compute the dissimilarity matrices $\mathbf X^s$. 
    Perturb the templates letting $\widetilde{\mu_{X^s}}$  be as $\mu_{X^s}$ with a removed portion (the same one as the simulated in step (1)), and re-normalized the vector of weights getting $\widetilde{\p^s}$, for each $s\in [S]$. Compute the new dissimilarity matrices $\widetilde{\mathbf X^s}$ across the remaining nodes.  
    \item Apply Algorithm \ref{alg: analysis} (or Algorithm \ref{alg: analysis blow up}) with inputs $\{(\widetilde{\mathbf X^s},\widetilde{\p^s})\}_{s\in [S]}$, $(\widetilde{\mathbf Y},\widetilde{\q})$ to compute estimated coordinates $\widetilde \lambda$. After that, apply Algorithm \ref{alg: synthesis} (or Algorithm \ref{alg: analysis blow up}) with inputs $\{({\mathbf X^s},{\p^s})\}_{s\in [S]}$, $\widetilde \lambda$, ${\q}$, $M$ in order to synthesize a barycenter $\mathbf{Y}_{\widetilde{\lambda}}$ from the     original (unperturbed) templates. We refer to $\mathbf{Y}_{\widetilde{\lambda}}$  as a  reconstruction of $\mathbf Y$.  Apply MDS embedding (with 2 or 3 components) to $\mathbf{Y}_{\widetilde{\lambda}}$ in order to visualize the recovered point cloud. 
\end{enumerate}

\begin{figure}[h!]
    \centering
    \begin{subfigure}[c]{0.49\linewidth}
        \centering
        \fbox{\includegraphics[width=0.97\linewidth]{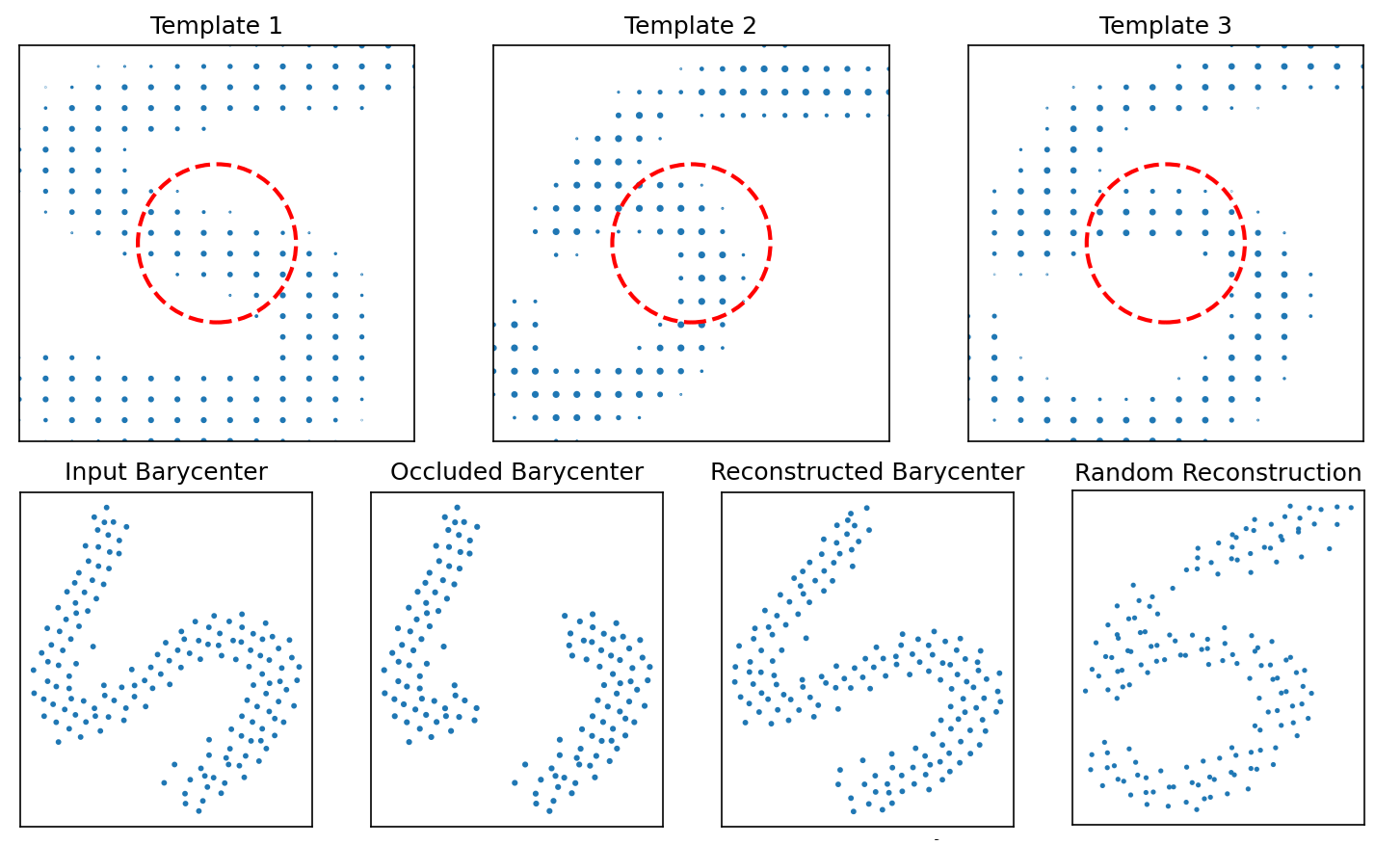}}
    \end{subfigure}
    \hfill
    \begin{subfigure}[c]{0.48\linewidth}
        \centering
        \fbox{\includegraphics[width=0.97\linewidth]{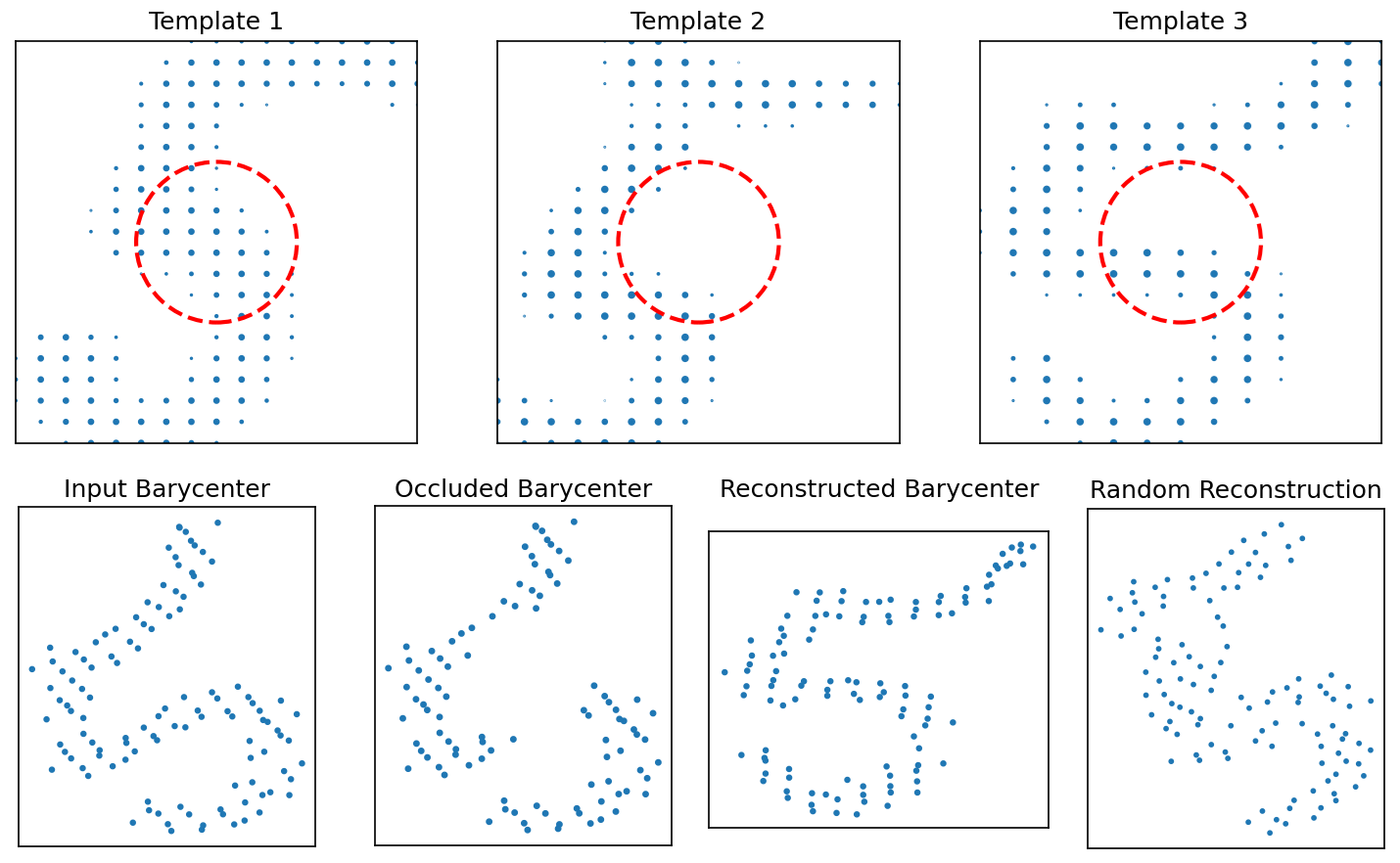}}
    \end{subfigure}
    \caption{\small{Occlusion with a circular mask (rotation-invariant) and reconstruction (two examples). \textbf{Top row:} Templates. \textbf{Bottom row from left to right:} Synthetic GW barycenter; its perturbation (via occlusion of the central region using a circular mask); the reconstructed point cloud using our proposed method (i.e., GW barycenter coordinates to the corrupted templates are computed using Algorithm \ref{alg: analysis} and then used to synthesize a GW barycenter with the unperturbed templates); and a random reconstruction (i.e., a synthesized barycenter using a random weight vector $\lambda \in \Delta_1$ and the unperturbed templates).}}
    \label{fig: occ_circular_and_rand}
\end{figure}

In the examples shown in Figure
\ref{fig: occ_circular_and_rand}, the BCM paradigm proves to be an effective approach to image reconstruction, as evidenced by its ability to recover the missing regions of a perturbed Gromov-Wasserstein (GW) barycenter. A synthetic barycenter is partially occluded using a circular mask, simulating structured data loss. Leveraging geometric information from the original templates (top panel), the BCM framework successfully reconstructs the barycenter (bottom right).

\subsection{
Exploring Alternative Node Dissimilarities: Euclidean and Non-Euclidean}\label{sec:NonEuc}

From the experimental perspective, further explorations using different node dissimilarities would be interesting to analyze. In this section we show some initial empirical results.

\begin{wrapfigure}[40]{r}{0.5\textwidth}

\centering

\begin{subfigure}[t]{0.9\linewidth}
  \centering
  \includegraphics[width=\linewidth]{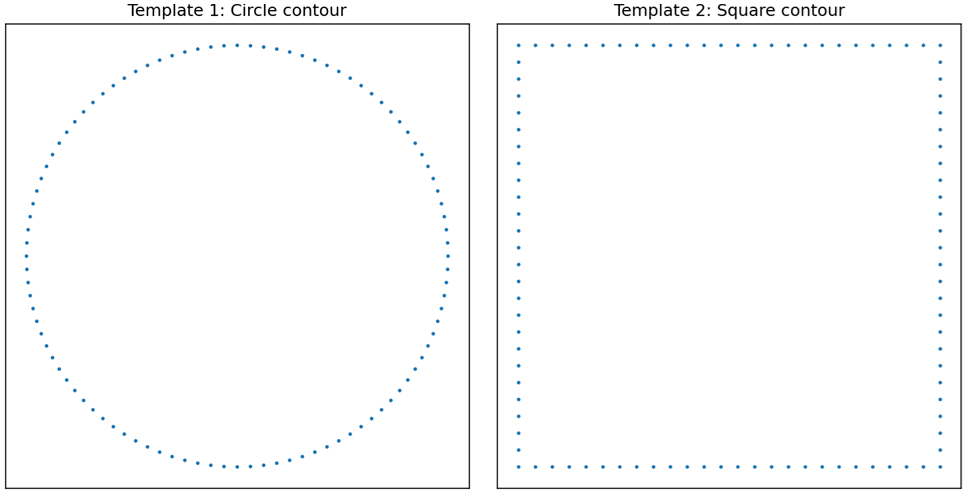}
  \caption{\small Templates.}
  \label{fig:sub1}
\end{subfigure}\par

\begin{subfigure}[t]{0.9\linewidth}
  \centering
  \includegraphics[width=\linewidth]{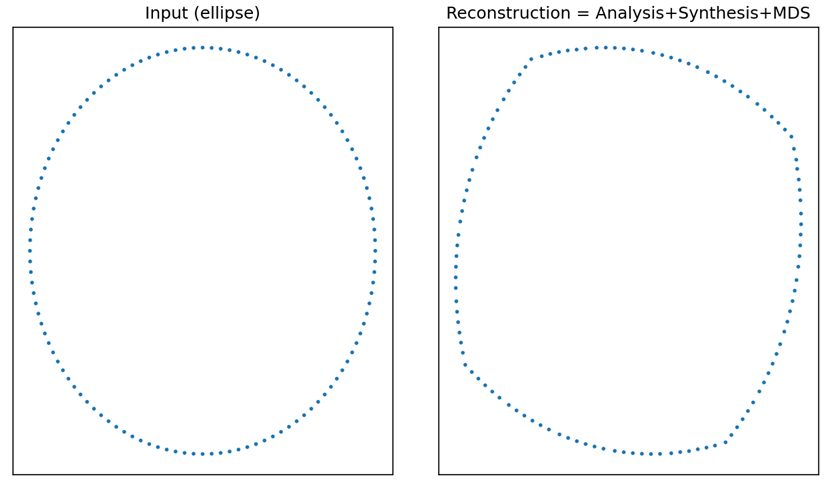}
  \caption{\small Target and its reconstruction.}
  \label{fig:sub2}
\end{subfigure}\par

\begin{subfigure}[t]{0.9\linewidth}
  \centering
  \includegraphics[width=\linewidth]{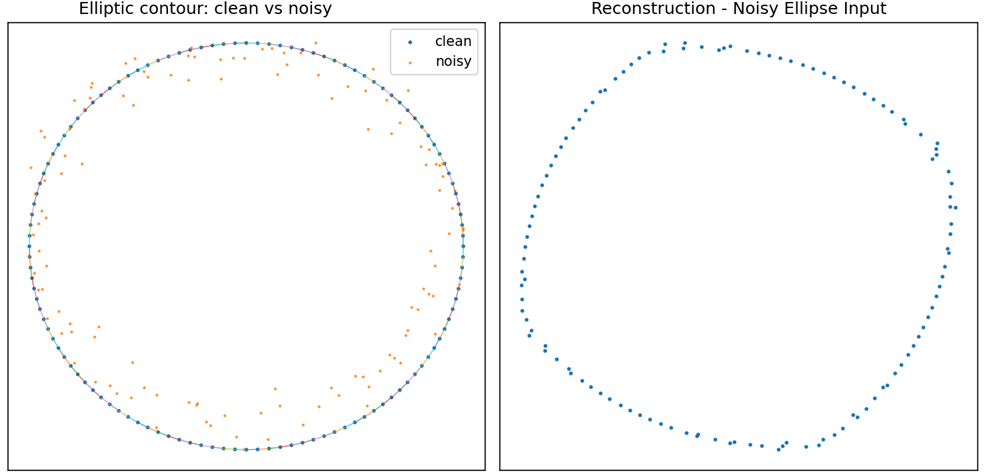}
  \caption{\small Noisy target and its reconstruction.}
  \label{fig:sub3}
\end{subfigure}

\caption{\small{Ambient distance is considered to compute distance matrices between the points of the above shapes, i.e., the Euclidean distance in $\mathbb R^2$. 
In (b) and (c), reconstruction is performed by first solving the  analysis problem to estimate GW $\lambda$-coordinates with respect to the templates in (a); then synthesizing the corresponding GW barycenter with those weights and templates from (a); and finally using MDS on the recovered matrix for visualization. 
In (b) the recovered weight via the analysis algorithm \ref{alg: analysis} is (0.6,0.4).}}
\label{fig:three-in-row}
\end{wrapfigure}

We consider two fixed templates: a circle and a square shape as illustrated in Fig. \ref{fig:sub1}. We want to explore how the analysis-synthesis pipeline behaves with different node dissimilarities. Using the natural ambient-space node dissimilarity (Euclidean distance in $\mathbb{R}^2$), we observe in Fig. \ref{fig:sub2} ``artifacts'' or ``corners'' that prevent perfect recovery of a circular input (elliptic contour), likely due to the interaction between the square template and the use of the Euclidean metric in the plane. 
Here, recovery is done by first estimating weights via GW analysis and then synthesizing a GW barycenter with such weights.
Nevertheless, when the input ellipse is perturbed by noise, the reconstruction remains reasonably stable (Fig. \ref{fig:sub3}) in comparison with the original reconstruction with no noise (Fig. \ref{fig:sub2}). 

It is natural to ask ourselves how this would be influenced if intrinsic distances are considered, such as geodesic distances on manifolds. As a preliminary exploration, in Table \ref{fig:table pq}, we consider $p$-Fermat distances for different input shapes (also called ``power-weighted shortest-path distances''), see \cite[Definition 1.1]{james-little2022balancing}. The inputs (or targets) are 
defined as the set of points $(x,y)\in \mathbb R^2$ satisfying  $|x|^q+|y|^q=1$ (i.e., boundaries of a $q$-ball in the plane), for different values of $q$. Thus, when $q=2$ the input shape is a circle and when $q=1$, a square. We sample them uniformly.

For the geodesic approach on the templates, each template is represented by an intrinsic distance matrix: for the circle we use the exact arc-length (geodesic) distance along the curve, and for the square we use shortest-path distances along the ordered boundary cycle (i.e., sums of edge lengths along the contour). Moreover, to make the square case more informative (and avoid it effectively reducing to the circle case), we sample the square more densely near its corners while keeping a uniform measure on the sampled points; this concentrates more mass at the corners and accentuates their geometric influence. These choices replace ambient Euclidean distances between node points with distances measured along the shape itself, i.e., within the template’s one-dimensional geometry. In this preliminary exposition, we consider uniform mass across sample nodes in the target and template 1D-shapes as well as equal number of nodes.

Table \ref{fig:table pq} reports the estimated GW coordinates $\lambda=(\lambda_1,\lambda_2)$ for different values of $p$ and $q$ along with the GW distance between the input shape  ($q$-boundary ball with its corresponding $p$-Fermat distance matrix) and the recovered one. Reconstruction is done in the same fashion as in Fig. \ref{fig:three-in-row}: after weight ($\lambda$) estimation though GW-analysis, GW-synthesis is performed. The $\lambda_1$ coordinate corresponds to the circle shape and $\lambda_2$ to the square template.  

The only case in Table \ref{fig:table pq} where we obtain perfect reconstruction is for $p=1$ and $q=2$.
This corresponds to the circle as the target, with its distance matrix computed via a path (shortest-path) metric, since $p=1$ yields the standard graph/path distance.
As expected, the reconstruction is essentially exact in this setting $GW\approx 0$ and $\lambda$ concentrated in the first component (circle template).
Another notable trend is that for $q=1$ (the square) and for large $q$ (i.e., shapes increasingly close to a square), the barycentric coordinates concentrate on the square template, with $\lambda \approx (0,1)$.

\begin{figure}[h!]
    \centering
    \includegraphics[width=\linewidth]{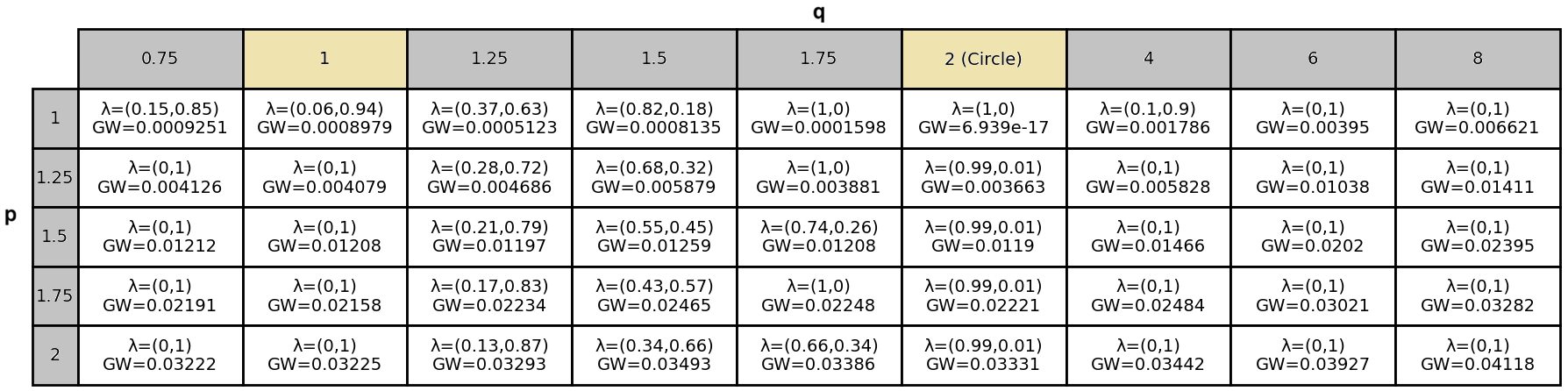}
    \caption{Report of $\lambda$-estimation for target shapes $|x|^q+|y|^q=1$  with distance matrices calculated via $p$-Fermat distances \cite[Definition 1.1]{james-little2022balancing}, for different values of $q$ (columns) and $p$ (rows). The templates in Fig. \ref{fig:sub1} are considered with internal geodesic distances. Shape reconstruction is performed by solving first the GW-analysis problem (via Algorithm \ref{alg: analysis}), and then synthesizing with the estimated weights and the fixed templates for all the experiments. GW distances between the initial target and its reconstruction are presented for each $(p,q)$-cell. }
    \label{fig:table pq}
\end{figure}

\section{The Analysis Problem Outside the Barycenter Span: A Preliminary and Empirical Exploration}\label{sec supp: projection}
Assume that $\{\mathbb X^s\}_{s\in [S]}$ are fixed templates,  and recall that we denote the elements of their GW barycenter span as $\mathbb Y_\lambda$, for $\lambda\in \Delta_{S-1}$. The objective of this section is to visualize, for an arbitrary target $\mathbb Y$, the heat map of the distance functional 
\begin{equation}\label{eq: functional J rev in text}
    \mathbf J(\lambda) := GW^2(\mathbb Y_\lambda, \mathbb Y),
\end{equation}
on the simplex $\Delta_{S-1}$, and the relation between its minima and the output of the analysis Algorithms \ref{alg: analysis} and \ref{alg: analysis blow up} when they are applied to the target $\mathbb Y$. 

By showing how such functional `looks like' over the simplex of weights, 
we aim to provide preliminary experimental results to the following questions: \emph{How the proposed GW analysis methods perform when the target shape $\mathbb Y$ does not belong to the GW barycenter space spanned by predetermined templates?} 
\emph{What happens when the minimum of the above functional \eqref{eq: functional J rev in text} is non-zero or cannot be attained exactly?} \emph{Do the weights $\lambda$ produced by the proposed GW analysis algorithms coincide with, or at least lie near, a minimizer of $\mathbf J$?}\\

\paragraph{Heatmap Generation} To visualize the heatmaps of $\mathbf J$ we will have to:
\begin{enumerate}
    \item Generate templates and a target.  
    \item Create a grid to discretize the simplex $\Delta_{S-1}$ and synthesize one GW barycenter per point $\lambda$; and 
    \item Compute the GW distance between the target and each synthesized barycenter of the grid.
\end{enumerate}

Assume that we are given three fixed templates ($S=3$) and a target. If we naively follow the steps above and synthesize GW barycenters in step 2 using already available libraries like \texttt{ot.gromov} \cite{flamary2021pot}, when running the procedure again we encounter different landscapes as seen in Fig. \ref{fig: bad lanscapces}. This behavior is mostly due to the issues inherent to instability of the GW barycenter synthesis problem. Although improving the synthesis algorithms is not essentially in the scope of this work, we added some tests to mitigate this variability. \\

\begin{figure}[h!]
    \centering
    \includegraphics[width=1\linewidth]{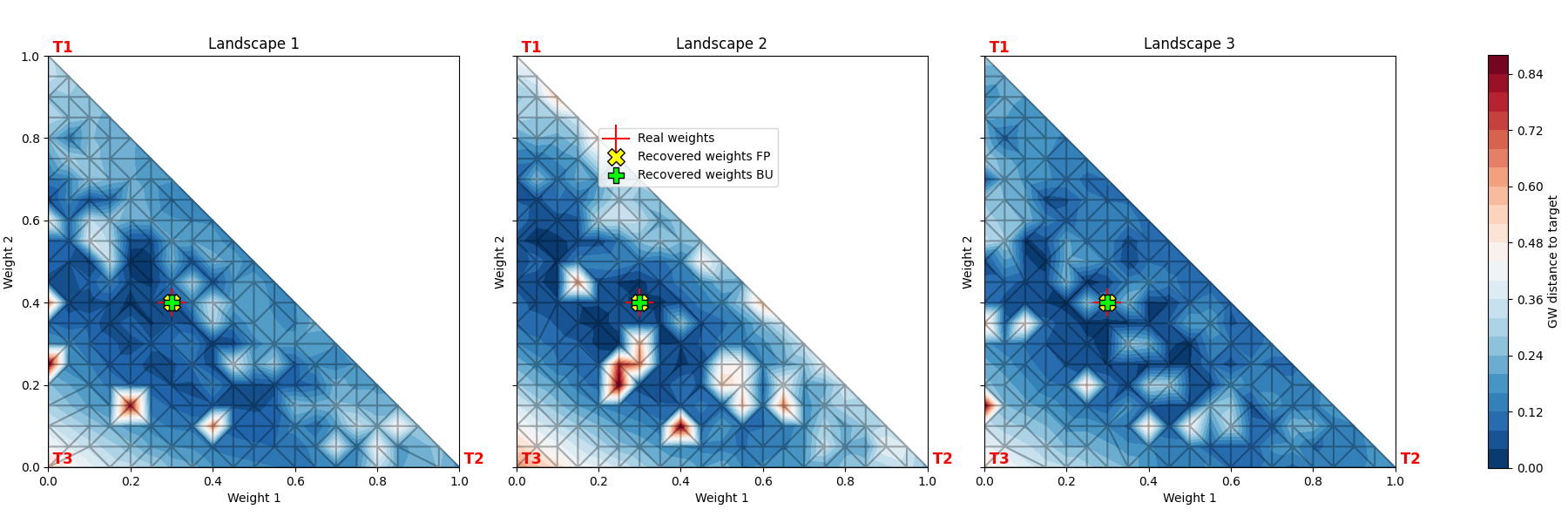}
    \caption{\small{Heatmaps illustrating the landscape of  $\mathbf{J}$ \eqref{eq: functional J rev in text} where no safeguards for GW barycenter synthesis process are applied for obtaining $\{\mathbb Y_\lambda\}_{\lambda\in \Delta_2}$. For the three plots, the same three fixed templates, and the same target $\mathbb Y$. Variability across the three landscapes is evident, even though the selected target lies within the GW barycenter space, as indicated by the red cross (weight).}}
    \label{fig: bad lanscapces}
\end{figure}

\paragraph{\textbf{Safeguards for the GW Barycenter Synthesis.}} 
The synthesis procedure requires to solve the non-convex minimization problem 
\begin{equation}\label{eq: syn app}
    \mathbb Y_\lambda \in \argmin_{\mathbb Z \in \mathrm{Bary}(\{[\mathbb X^s]\}_{s\in [S]})} \sum_{s=1}^S \lambda_s GW^2(\mathbb Z, \mathbb X^s).
\end{equation}
Non-convex problems often have no guarantees of uniqueness for global minimizers (for the same weight $\lambda$), and moreover, numerical solvers might obtain only local minima or may suffer from convergence issues. In particular, the function \texttt{ot.gromov.gromov\_barycenters} from the POT Library \cite{flamary2021pot}, used to generate GW barycenters turns to be numerically sensitive under different initializations: if we fix a size $M$, a vector $\mathrm q$ of weights, a set of templates $\{\mathbb X^s\}_{s\in [S]}$, and $\lambda\in \Delta_{S-1}$ and then use such function twice to synthesize $\mathbb Y$ and $\mathbb Y'$ (with the same inputs), they do not necessarily have GW distance $\approx 0$.

This already indicates that the landscapes for $\mathbf J$ (with the same target $\mathbb Y$) are going to change. Even worse, it also indicates that the synthesized object might not even be a GW barycenter. We added two safeguards to make sure that that is not the case: 
\begin{enumerate}
    \item[(a)] After synthesizing with the POT Library function, we test if the result is a weak GW barycenter by following Remark \ref{remark: test}. 
    \item[(b)] For a fixed $\lambda$ we generate several candidate barycenters, and choose the one that has minimum synthesis cost (i.e., minimum for the functional in \eqref{eq: syn app}).
\end{enumerate}

With this in mind, we created templates and targets inside and outside of the GW barycenter space, and reproduced the landscapes  in Fig. \ref{fig: landscape in Bary} and Fig. \ref{fig: landscape out Bary}, respectively. \\

\newpage

\paragraph{\textbf{Template Generation.}} We considered $S=3$ and generated three random point clouds of $N$ points each. For each, we assign the Euclidean distance between points as distance matrix, and assumed a uniform measure. In practice, using uniform measures helps make the POT Library’s barycenter synthesis solver output compatible with the blow-up procedure (as empirically noted in Figures \ref{fig: analysis_alg} and \ref{fig: 100p}), which in turn helps satisfy safeguard (a). We additionally tested the pairwise GW distance between templates to be bigger than a minimum threshold to ensure that the distance between grid points would end being (heuristically) bigger than the distance between different synthesized barycenter with the same weights. See Fig. \ref{fig: temp for land} for a visualization of a choice of such type of templates. \\
\begin{figure}[ht!]
    \centering
    \includegraphics[width=0.32\linewidth]{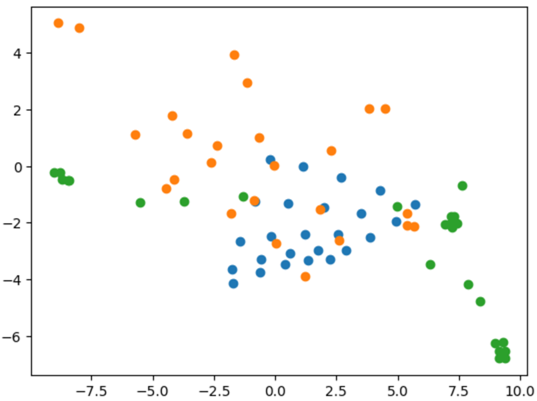}
    \caption{\small{Orange, blue and green point clouds illustrating a choice of 3 templates. }}
    \label{fig: temp for land}
\end{figure}

\paragraph{\textbf{Target Generation.}}
(1) To obtain a target outside the barycenter space, we generated one extra random point cloud of size $N$ as when creating the templates. (2) To obtain a target in the GW barycenter space of the templates, we artificially synthesized a new GW barycenter with a chosen weight $\lambda$ using the safeguards mentioned above. In our visualizations we considered $\lambda=(0.3,0.4,0.3)$.
Since the weight $\lambda$ was taken inside the grid, we made sure that the target was created before (and independently) of the grid barycenters to avoid forcing a zero of $\mathbf J$ at that point.\\

\paragraph{\textbf{Conclusions and Observations.}} 
Through this brute-force method for visualizing the $\mathbf J$ landscapes (Figures \ref{fig: landscape in Bary} and \ref{fig: landscape out Bary}), one can informally assess the relation between the proposed analysis Algorithms \ref{alg: analysis} and \ref{alg: analysis blow up} and the minima of the functional $\mathbf J$ for targets inside and outside the GW barycenter space. When the target lies in the barycenter space, both algorithms recover the expected weights. When outside the GW barycenter space, the recovered weights still coincide for both algorithms. In both cases, the recovered weights seem to lie in regions where $\mathbf J$ is small (deep blue regions in the heatmaps)
showing qualitative agreement with the landscape structure. Consequently, the out-of-model results are satisfactory. 

It is worth noticing that the lack of uniqueness for the synthesis problem (from both, theoretical and computational, perspectives) makes it very difficult to properly visualize the global minima of $\mathbf J$ for targets outside the barycenter space. 
Even though we included  all the technicalities mentioned above, trying to guarantee less variability of the (brute-force) heatmaps, still the three plots within Figures \ref{fig: landscape in Bary} and \ref{fig: landscape out Bary}, respectively, do not coincide exactly.

\begin{figure}[H]
    \centering
    \includegraphics[width=\linewidth]{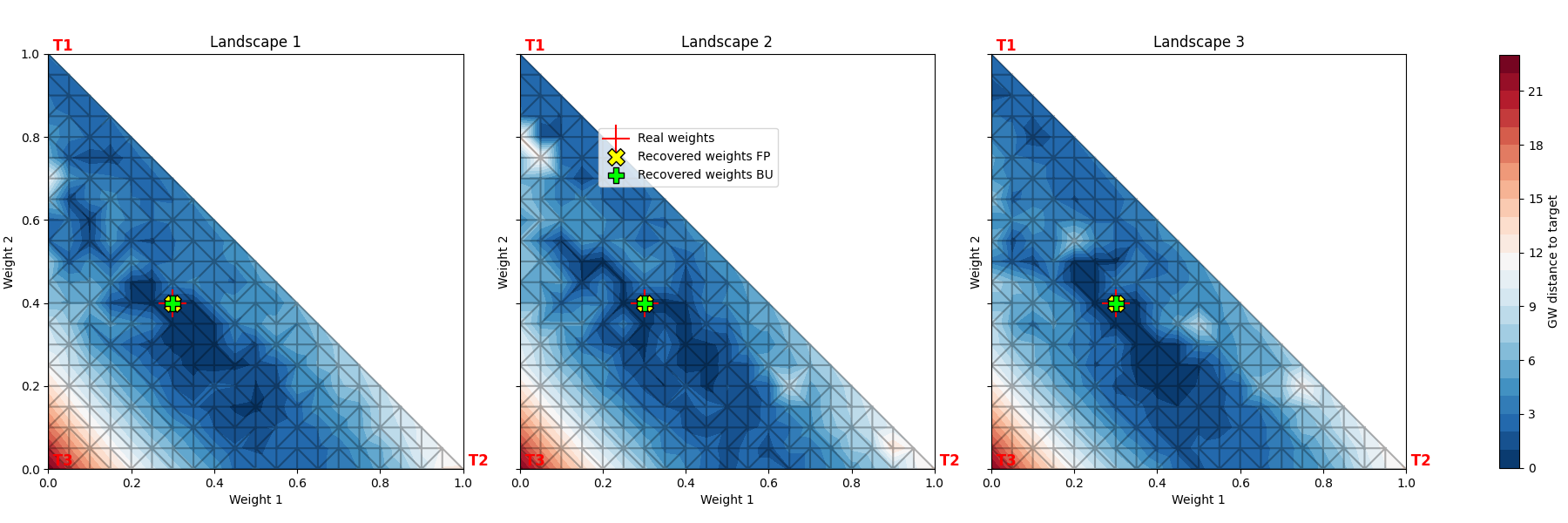}
    \caption{\small{
    Heatmaps illustrating the landscape of $\mathbf{J}$ over the first two components of the simplex $\Delta_2$.
For the same three fixed templates and the same target $\mathbb{Y}$ \emph{lying in their GW barycenter space}, we evaluate $\mathbf{J}(\lambda^{(i)})$ on a grid $\{\lambda^{(i)}\}\subset \Delta_2$.
The ground-truth weight $\lambda^*$ corresponding to $\mathbb{Y}$ is marked by a red cross.
The yellow and green crosses indicate the estimates of $\lambda^*$ produced by Algorithms \ref{alg: analysis} and \ref{alg: analysis blow up}, respectively.
The three panels correspond to three independent synthesis runs used to generate the barycenters $\{\mathbb{Y}_{\lambda^{(i)}}\}$ over the simplex grid, incorporating the technical safeguards introduced to mitigate large run-to-run variability.
}}
    \label{fig: landscape in Bary}

\end{figure}

\begin{figure}[H]
    \centering
    \includegraphics[width=\linewidth]{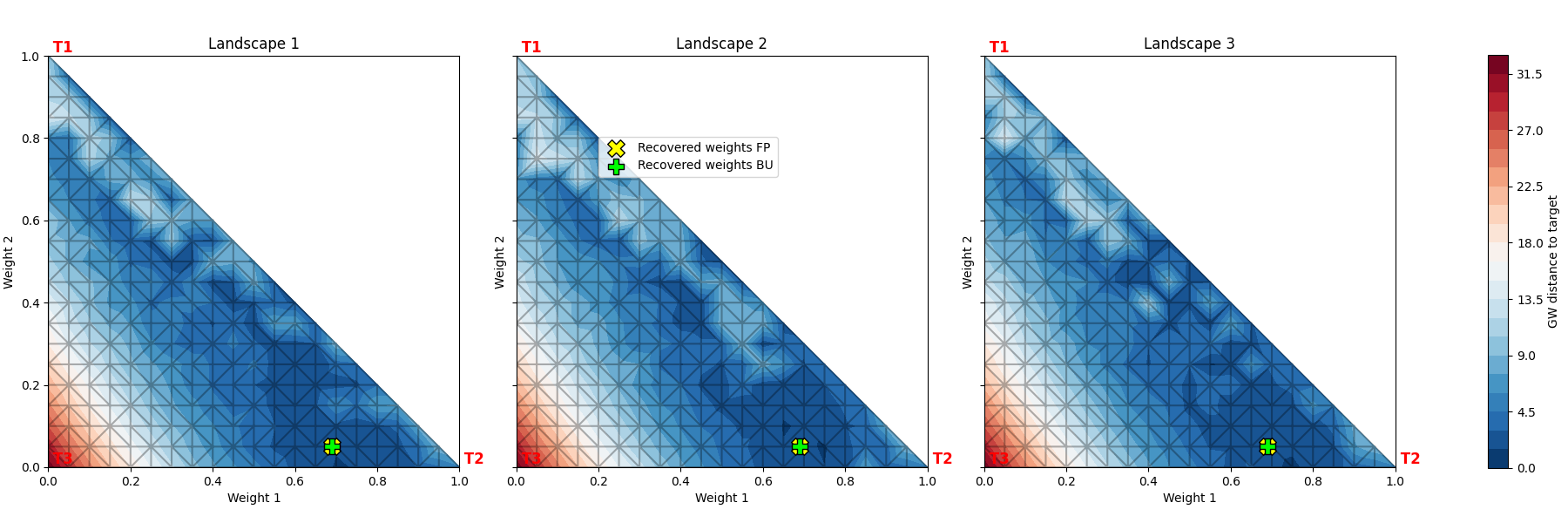}
    \caption{\small{
    Heatmaps illustrating the landscape of $\mathbf{J}$ over the first two components of the simplex $\Delta_2$.
For the same three fixed templates and the same target $\mathbb{Y}$ \emph{outside their GW barycenter space}, we evaluate $\mathbf{J}(\lambda^{(i)})$ on a grid $\{\lambda^{(i)}\}\subset \Delta_2$.
The yellow and green crosses indicate the estimates of GW $\lambda$-coordinates corresponding to $\mathbb Y$ produced by Algorithms \ref{alg: analysis} and \ref{alg: analysis blow up}, respectively.
The three panels correspond to three independent synthesis runs used to generate  barycenters $\{\mathbb{Y}_{\lambda^{(i)}}\}$ over the simplex grid, incorporating the technical safeguards introduced to mitigate large run-to-run variability.}}
    \label{fig: landscape out Bary}
\end{figure}

\subsection{Behavior Under Permutation of the Proposed Analysis Algorithms \ref{alg: analysis} and \ref{alg: analysis blow up}}\label{sec: permutation}

From the theoretical perspective, both of our analysis methods are permutation invariant: if we relabel the nodes of a network, the $\lambda$-estimation does not change.

\begin{itemize}[leftmargin=*]
    \item 
    In Section \ref{sec: Analysis}, we address  permutation/relabeling invariance in Remark \ref{remark: permutation}: The formulation of the analysis problem considered there (see problem \eqref{eq: analysis_gw_bary_problem}), in which the GW barycenter model is taken over shapes in $\mathrm{Bary}_{M,\q}$  (Definition \ref{eq: bary space for M and q fixed}),  together with its surrogate \eqref{eq: our_analysis_bary_problem}, 
    are  permutation/relabeling invariant.
    In both problems  \eqref{eq: analysis_gw_bary_problem}, \eqref{eq: our_analysis_bary_problem} we take the divergence $d_a(\cdot,\cdot)$ to be the Frobenius norm on the space of $M\times M$ matrices (which allows us to develop Algorithm \ref{alg: analysis}). Precisely, for fixed templates $\{(\mathbf X^s,\p^s)\}_{s\in[S]}$, given a network $(\mathbf Y,\q)$ and using the notation from Section \ref{sec: Analysis}, the landscape of the functional  

    \begin{equation}\label{eq: landscape FP}
        \Delta_{S-1}\ni\lambda\longmapsto\left\|\mathbf Y-\sum_{s\in [S]}\lambda_s \, F(\mathbf Y, s)\right\|_{\mathrm{Frob}}=\lambda^T\mathcal Q \lambda
\end{equation}
    is invariant under relabeling the nodes of $(\mathbf Y,\q)$.

    \item 
    In Section \ref{sec: analysis bu}, in formulating the analysis problem that leads to Algorithm \ref{alg: analysis blow up}, we work with the weak barycenter model $\widetilde{\mathrm{Bary}}$. Moreover, we consider equivalence classes (denoted as $[\mathbb Y]$ for weakly isomorphic networks) which already incorporate permutation/relabeling invariance: that is,  if the network $\mathbb Y$ is a representative of the class $[\mathbb Y]$, and we relabel the nodes of such network creating a new one $\mathbb Y'$, then $[\mathbb Y]=[\mathbb Y']$. Precisely, for fixed templates $\{[\mathbb X^s]\}_{s\in[S]}$, given a network $\mathbb Y=(\mathbf Y,\q)$ and using the notation from Section \ref{sec: analysis bu}, the landscape of the functional  
    \begin{equation}\label{eq: landscape BU}
        \Delta_{S-1}\ni\lambda\longmapsto\mathrm{tr}_{\q_b}\left(\left(\sum_{s=1}^S\lambda_s\mathbf X^s_b-\mathbf Y_b\right)^T\left(\sum_{r=1}^S\lambda_r\mathbf X^r_b-\mathbf Y_b\right)\right)=\lambda^T\mathcal A \lambda
\end{equation}
    is invariant under relabeling the nodes of $\mathbb Y$.
\end{itemize}

In this section we demonstrate this invariance through experimental visualizations. Please see Fig. \ref{fig:permutation_all}. For all the plots, we consider $3$ templates as point clouds selected as describe above in the main part of \ref{sec supp: projection}. Equal sized point clouds, with uniform weights across nodes, where used in all these illustrative experiments. 

\begin{figure}[H]
    \centering
    \begin{subfigure}[b]{0.49\linewidth}
        \centering
        \includegraphics[width=\linewidth]{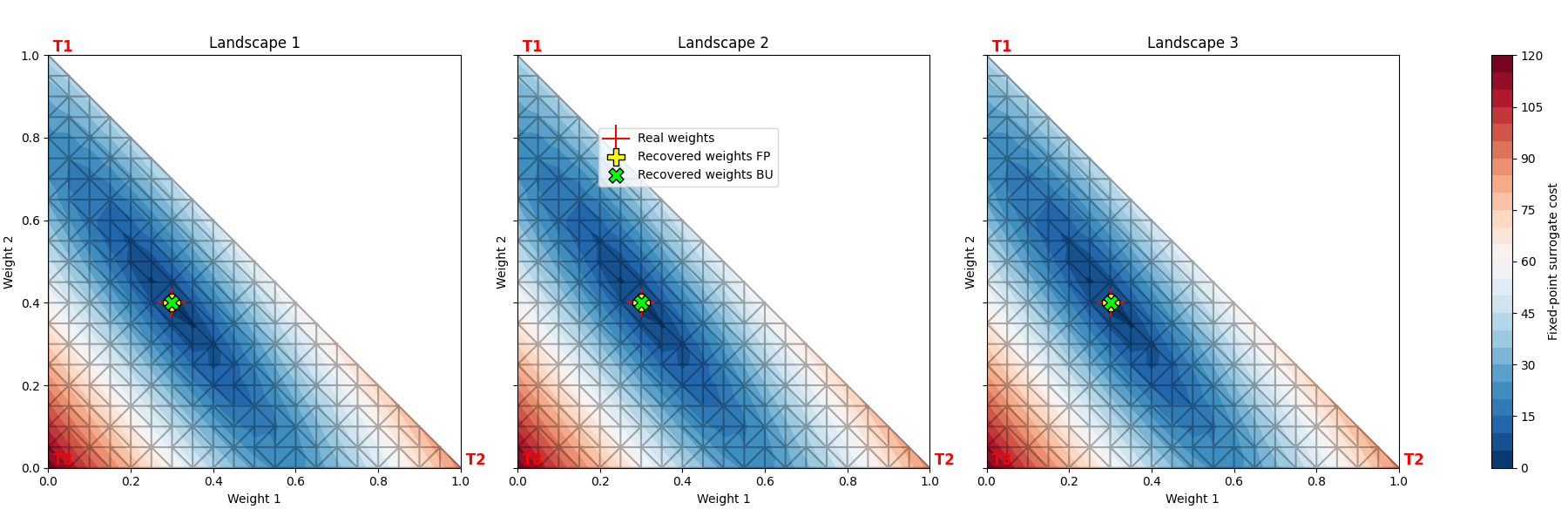}
        \caption{\small{Inside Barycenter Space -- FP landscape.}}
        \label{fig:inside_fp}
    \end{subfigure}\hfill
    \begin{subfigure}[b]{0.49\linewidth}
        \centering
        \includegraphics[width=\linewidth]{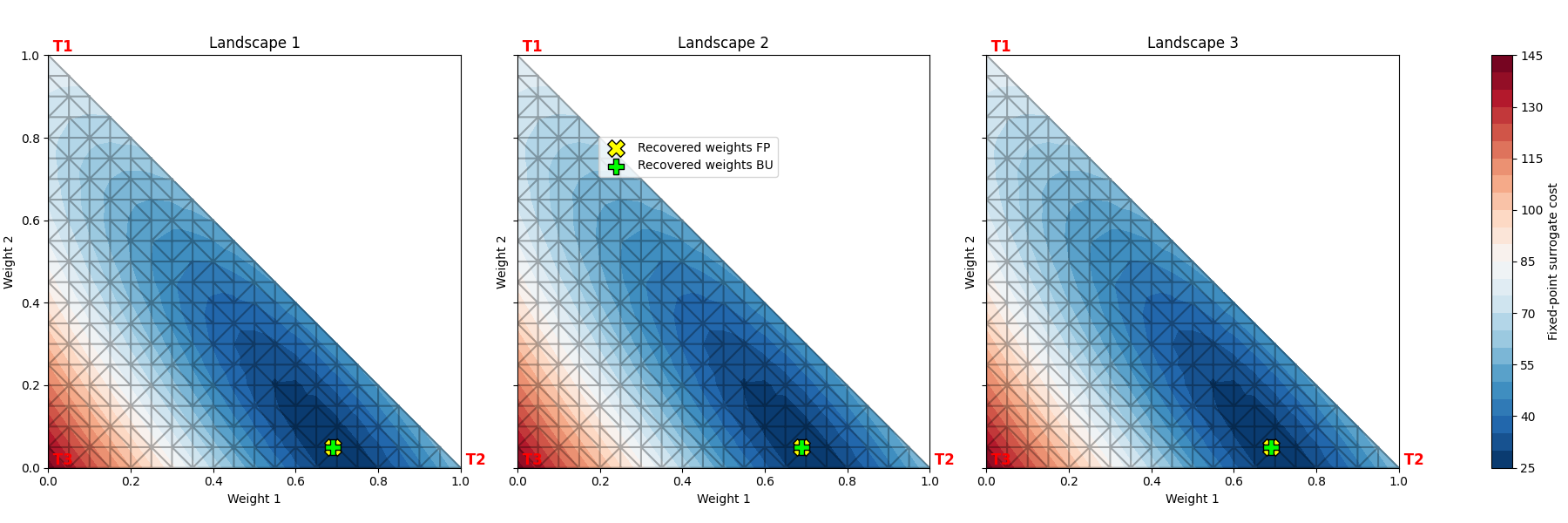}
        \caption{\small{Outside Barycenter Space - FP landscape.}}
        \label{fig:outside_fp}
    \end{subfigure}
    \begin{subfigure}[b]{0.49\linewidth}
        \centering
        \includegraphics[width=\linewidth]{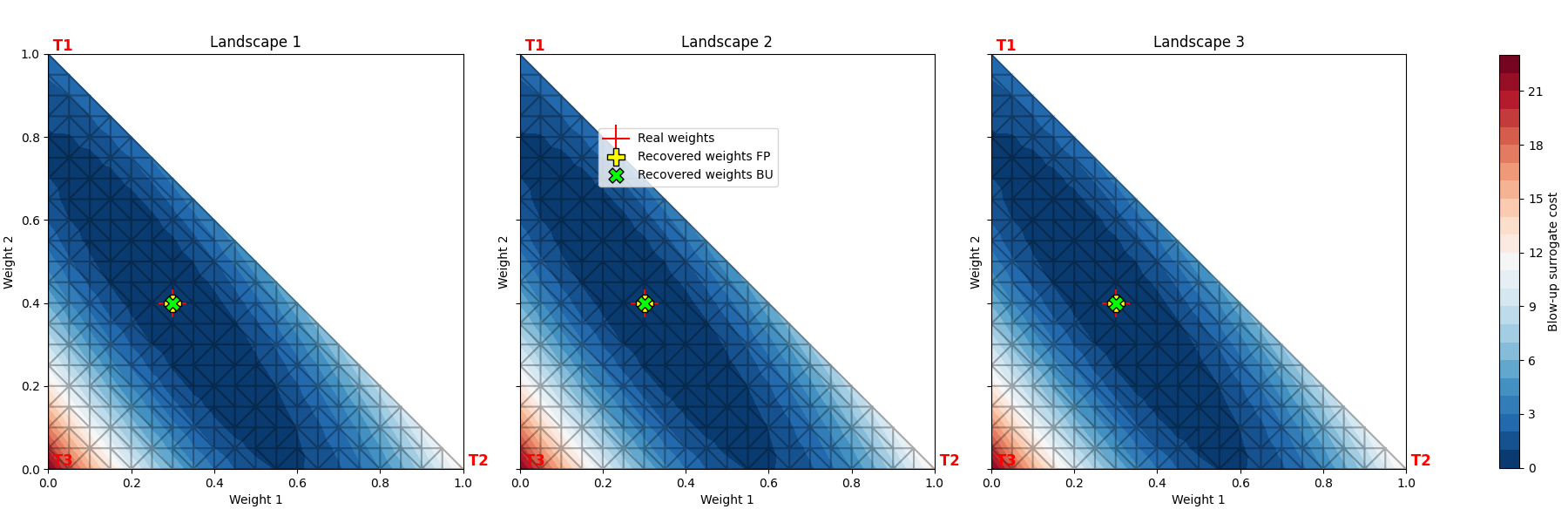}
        \caption{\small{Inside Barycenter Space -- BU landscape.}}
        \label{fig:inside_bu}
    \end{subfigure}\hfill
    \begin{subfigure}[b]{0.49\linewidth}
        \centering
        \includegraphics[width=\linewidth]{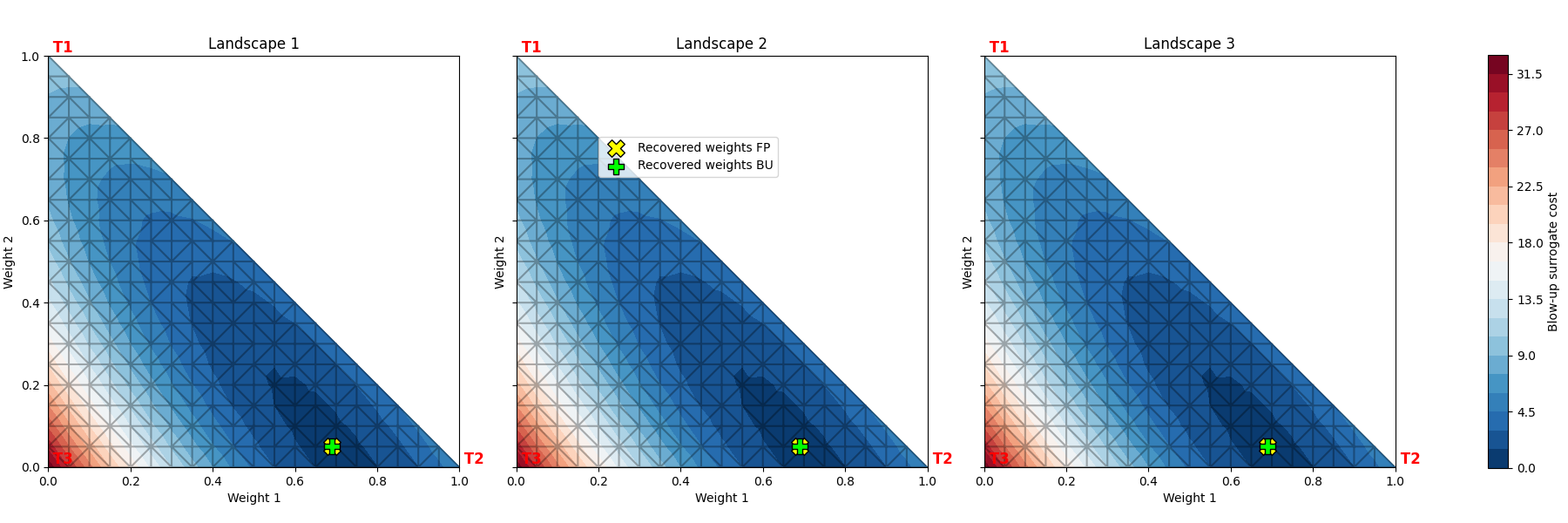}
        \caption{\small{Outside Barycenter Space -- BU landscape.}}
        \label{fig:outside_bu}
    \end{subfigure}

    \caption{\small{Permutation/relabeling invariance examples for fixed-point (FP) analysis Algorithm \ref{alg: analysis} and blow-up (BU) analysis Algorithm \ref{alg: analysis blow up}. We only visualize the first two coordinates of the simplex $\Delta_2$.
    For each subplot (a)--(d), we consider a point cloud target $(\mathbf Y,\q)$ and two permutations (node relabeling) of it. On the left ((a) and (c)), the two input targets are synthetically generated as a GW barycenter (the red cross indicates the true $=(0.3,0.4,0.3)$).  On the right ((b) and (d)), two targets 
    are considered. 
    Top panels (a) and (b) depict the heatmap of the landscape \eqref{eq: landscape FP} when varying $\lambda\in \Delta_2$. 
    Bottom panels (c) and (d) depict the heatmap of the landscape \eqref{eq: landscape BU} when varying weights in $ \Delta_2$. 
    For all subplots (a)--(d), $\lambda$-estimation through Algorithm \ref{alg: analysis} is marked as a yellow cross, and  $\lambda$-estimation via Algorithm \ref{alg: analysis blow up} as a green cross in all plots. In each of the subfigures (a)-(d), the three corresponding landscapes are almost identical, demonstrating the invariance.}}
    \label{fig:permutation_all}
\end{figure}

\section{Wasserstein Barycenters and Gromov-Wasserstein Barycenters}\label{app: GW Was}

In the context of the Wasserstein space, both approaches for computing Wasserstein barycenters, namely the fixed-point scheme and the gradient method, can be interpreted as the same. Indeed, let us consider templates $\{\mu^s\}_{s=1}^S$ in the space $\mathcal{P}_2(\mathbb{R}^d)$ of probability measures on $\mathbb{R}^d$ with finite second moments. Given $\lambda \in \Delta_{S-1}$, the functional $G_\lambda$ in \eqref{eq: barycenter functional} is replaced by
\begin{equation}\label{eq: wass functional}
    \mathcal{P}_2(\mathbb{R}^d) \ni \nu \longmapsto \frac{1}{2} \sum_{s=1}^S \lambda_s \, W^2_2(\nu, \mu^s).
\end{equation}
If all measures are absolutely continuous with respect to the Lebesgue measure on $\mathbb{R}^d$, its Fréchet derivative \cite{frechet1948elements} can be written as
\begin{equation}\label{eq: frechet deriv Wass}
    \sum_{s=1}^S \lambda_s (Id - T^{\nu \to \mu^s}),
\end{equation}
where $Id: \mathbb{R}^d \to \mathbb{R}^d$ is the identity map, and $T^{\nu \to \mu^s}: \mathbb{R}^d \to \mathbb{R}^d$ is the \emph{OT map}\footnote{$T^{\nu \to \mu^s}$ is the \emph{OT map} in the sense that $T^{\nu \to \mu^s}_\# \nu = \mu^s$, and it realizes the Wasserstein distance as
$
W_2^2(\nu, \mu^s) = \int_{\mathbb{R}^d} \|x - T^{\nu \to \mu^s}(x)\|^2 \, d\nu(x)
$.} from $\nu$ to $\mu^s$. 
Thus, if $\nu_\lambda$ is a Wasserstein barycenter, it must satisfy the first-order optimality condition for the functional \eqref{eq: wass functional}\footnote{In other words, $\nu_\lambda$ is a probability measure such that the identity map equals the weighted average of the OT maps from $\nu_\lambda$ to the templates $\{\mu^s\}_{s=1}^S$, i.e.,
$
Id = \sum_{s=1}^S \lambda_s T^{\nu_\lambda \to \mu^s}.
$}, that is,
\begin{equation}\label{eq: wass cp}
\sum_{s=1}^S \lambda_s (Id - T^{\nu_\lambda \to \mu^s}) = 0 
\quad \text{almost everywhere w.r.t. } \nu_\lambda.    
\end{equation}
Moreover, it can be proven that under certain assumptions  critical points (or Karcher means), that is, measures $\nu_\lambda$ satisfying \eqref{eq: wass cp} and Wasserstein barycenters (i.e., minimizer of the functional described by \eqref{eq: wass functional}) coincide (see \cite{werenski2022measure} and \cite[Thm. 3.1.15]{panaretos2020invitation}). 
As a consequence, to find such a Wasserstein barycenter $\nu_\lambda$, one can adopt the following fixed-point scheme based on (Wasserstein) gradient descent on the variance functional \eqref{eq: wass functional}:
\begin{equation*}
\nu_{n+1} = \left(Id - \eta \sum_{s=1}^S \lambda_s (Id - T_n^s)\right)_\# \nu_n,    
\end{equation*}
where $T_n^s$ is the OT map from the current iterate $\nu_n$ to $\mu^s$, and $\eta > 0$ is a suitable step size. This iteration seeks a stationary point of the Wasserstein barycenter functional by successively updating $\nu_n$ via a transport-based averaging of the maps $T_n^s$ (we refer the reader to \cite{alvarez2016fixed}, and also to the recent work \cite{chen2025provably}).
In conclusion, the well-defined geometric structure of the Wasserstein space allows gradient and fixed-point methods to work together, complementing each other, when addressing the Wasserstein synthesis problem.

In the context of this paper, that is, under the Gromov-Wasserstein (GW) structure, the objective functional 
$G_\lambda$  (see \eqref{eq: barycenter functional}) plays the role analogous to that of \eqref{eq: wass functional} in the Wasserstein setting. Its gradient can be computed using the blow-up technique (see Proposition \ref{prop: grad via blow up}), and it serves as the counterpart to the Fréchet derivative in \eqref{eq: frechet deriv Wass}.  In \cite{chowdhury2020gromov}, the authors propose an iterative Fréchet mean algorithm based on the work of Pennec \cite{pennec2006intrinsic}, interpreting it as a gradient descent method for minimizing the Fréchet functional $G_\lambda$ over the space of finite networks. This approach parallels the procedure described in the Wasserstein setting but is adapted to the GW framework, following Pennec’s ideas on averaging in finite-dimensional Riemannian manifolds. However, to the best of our knowledge, there are no theoretical guarantees in the literature that critical points of $G_\lambda$ are minimizers of \eqref{eq: syn gw}, that is, true GW barycenters. Our Lemma \ref{thm: equiv} represents a step in this direction.

\bibliographystyle{siam}
\bibliography{references}

\end{document}